\documentclass[11pt]{article}
\usepackage{amsmath}
\usepackage{amsfonts}
\usepackage{amssymb}
\usepackage{mathrsfs}
\usepackage{amsthm}
\usepackage{palatino}
\usepackage[margin=2.5cm, vmargin={1.5cm}]{geometry}

\usepackage{dcolumn}

\usepackage{times}
\usepackage{graphicx}
\usepackage{xcolor}

\usepackage{pstricks}
\usepackage{pst-plot}
\usepackage{pst-grad}

\usepackage{wrapfig}

\usepackage{bm}

\usepackage{caption}
\renewcommand\footnotemark{}

\begin{document}

\title{Patterns in the Markov numbers and their generalizations}

\author{
  Cormac ~O'Sullivan\footnote{{\it Date:} Sept 9, 2026.
\newline \indent \ \ \
\newline \indent \ \ \
  \newline \indent \ \ \
Support for this project was provided by a PSC-CUNY Award, jointly funded by The Professional Staff Congress and The City
\newline \indent \ \ \
University of New York.}
  }

\date{}

\maketitle


\def\s#1#2{\langle \,#1 , #2 \,\rangle}

\def\H{{\mathbb{H}}}
\def\F{{\mathfrak F}}
\def\Fd{{\mathcal F}}
\def\Gd{{\mathcal G}}
\def\C{{\mathbb C}}
\def\R{{\mathbb R}}
\def\Z{{\mathbb Z}}
\def\Q{{\mathbb Q}}
\def\N{{\mathbb N}}
\def\P{{\mathbb P}}
\def\G{{\Gamma}}
\def\GH{{\G \backslash \H}}
\def\g{{\gamma}}
\def\L{{\Lambda}}
\def\ee{{\varepsilon}}
\def\K{{\mathcal K}}
\def\Re{\mathrm{Re}}
\def\Im{\mathrm{Im}}
\def\PSL{\mathrm{PSL}}
\def\SL{\mathrm{SL}}
\def\GL{\mathrm{GL}}
\def\Vol{\operatorname{Vol}}
\def\lqs{\leqslant}
\def\gqs{\geqslant}
\def\sgn{\operatorname{sgn}}
\def\res{\operatornamewithlimits{Res}}
\def\li{\operatorname{Li_2}}
\def\lip{\operatorname{Li}'_2}
\def\pl{\operatorname{Li}}

\def\ch{\operatorname{ch}}
\def\tr{\operatorname{tr}}

\def\ei{\mathrm{Ei}}

\def\clp{\operatorname{Cl}'_2}
\def\clpp{\operatorname{Cl}''_2}
\def\farey{\mathscr F}

\def\dm{{\mathcal A}}
\def\ov{{\overline{p}}}
\def\ja{{K}}

\def\nb{{\mathcal B}}
\def\cc{{\mathcal C}}
\def\nd{{\mathcal D}}

\def\u{{\text{\rm u}}}
\def\v{{\text{\rm v}}}
\def\ib{{\text{\,\rm i}}}
\def\jb{{\text{\,\rm j}}}
\def\qq{{f}}

\def\mg{{\mathcal M_{G}}}
\def\mz{{\mathcal M_{Z}}}
\def\stab{{\text{\rm Stab}}}
\def\aut{{\text{\rm Aut}}}
\def\ze{{z}}

\newcommand{\stira}[2]{{\genfrac{[}{]}{0pt}{}{#1}{#2}}}
\newcommand{\stirb}[2]{{\genfrac{\{}{\}}{0pt}{}{#1}{#2}}}
\newcommand{\eu}[2]{{\left\langle\!\! \genfrac{\langle}{\rangle}{0pt}{}{#1}{#2}\!\!\right\rangle}}
\newcommand{\eud}[2]{{\big\langle\! \genfrac{\langle}{\rangle}{0pt}{}{#1}{#2}\!\big\rangle}}
\newcommand{\norm}[1]{\left\lVert #1 \right\rVert}
\newcommand{\dx}[1]{\overset{*}{#1}}

\newcommand{\e}{\eqref}
\newcommand{\la}{\label}
\newcommand{\bo}[1]{O\left( #1 \right)}
\newcommand{\ol}[1]{\,\overline{\!{#1}}} 


\newtheorem{theorem}{Theorem}[section]
\newtheorem{lemma}[theorem]{Lemma}
\newtheorem{prop}[theorem]{Proposition}
\newtheorem{conj}[theorem]{Conjecture}
\newtheorem{cor}[theorem]{Corollary}
\newtheorem{assume}[theorem]{Assumptions}
\newtheorem{adef}[theorem]{Definition}
\newtheorem{ex}[theorem]{Example}

\newtheorem*{algo}{Reduction algorithm}
\newtheorem*{algo2}{Continued fraction algorithm}
\newtheorem*{algo3}{General continued fraction algorithm}
\newtheorem*{algo4}{Algorithm to find $\varepsilon_D$ and $\varepsilon^*_D$}

\numberwithin{figure}{section}
\numberwithin{table}{section}


\newcounter{counrem}
\newtheorem{remark}[counrem]{Remark}

\renewcommand{\labelenumi}{(\roman{enumi})}
\newcommand{\spr}[2]{\sideset{}{_{#2}^{-1}}{\textstyle \prod}({#1})}
\newcommand{\spn}[2]{\sideset{}{_{#2}}{\textstyle \prod}({#1})}

\numberwithin{equation}{section}

\let\originalleft\left
\let\originalright\right
\renewcommand{\left}{\mathopen{}\mathclose\bgroup\originalleft}
\renewcommand{\right}{\aftergroup\egroup\originalright}

\bibliographystyle{alpha}

\setcounter{topnumber}{1}

\begin{abstract}
Positive integer solutions to $x^2+y^2+z^2- x y z =D$  correspond to the important Markov (Markoff) numbers when $D=0$. From a given solution triple, three more are found with Vieta involutions, making an infinite tree of solutions. Starting instead with three real numbers greater than $2$ gives lengths of closed geodesics in a  punctured torus. In this paper we study these tree structures for any real $D$.   A continuous function, originally related to a norm on homology,  encodes all the numbers on each of these trees. The  properties of this function are developed here in general, with a  self-contained exposition, showing that the usual $D=0$ case is part of a bigger picture. Encoding function graphs are shown to be convex for $D<4$, straight lines for $D=4$ and concave for $D>4$. 
Among other consequences  are generalizations to all $D$ of: estimates for counting numbers on these trees, descriptions of the geometry of the corresponding lattice curves, uniqueness conditions, and   identities of McShane and Hines.  
\end{abstract}

\section{Introduction}
The Markov (or Markoff) numbers $1, 2, 5, 13, 29, 34, 89, 169, \dots$ are those appearing in positive integer solutions $(x,y,z)$ of the equation
$
  x^2+y^2+z^2 = 3 x y z
$. They have a long history, see \cite[Chap. 1, 2]{aig}, with their initial importance  related to the Lagrange spectrum,  showing  fundamental limits to the approximation of irrationals by rationals. We will also see their connections to hyperbolic geometry.  Multiplying 
by $3$ to get $3, 6, 15, 39, \dots$ allows a convenient equivalent formulation  that we will use, as solutions of
\begin{equation} \la{mar2}
  x^2+y^2+z^2- x y z =D,
\end{equation}
for $D=0$. In this paper, we allow $x, y, z$ and $D$ in \e{mar2} to be real numbers with 
\begin{equation}\label{restr}
  x, y, z \in (-\infty,-2] \cup [2,\infty).
\end{equation}

\SpecialCoor
\psset{griddots=5,subgriddiv=0,gridlabels=0pt}
\psset{xunit=0.35cm, yunit=0.35cm, runit=0.5cm}
\psset{linewidth=1pt}
\psset{dotsize=4pt 0,dotstyle=*}
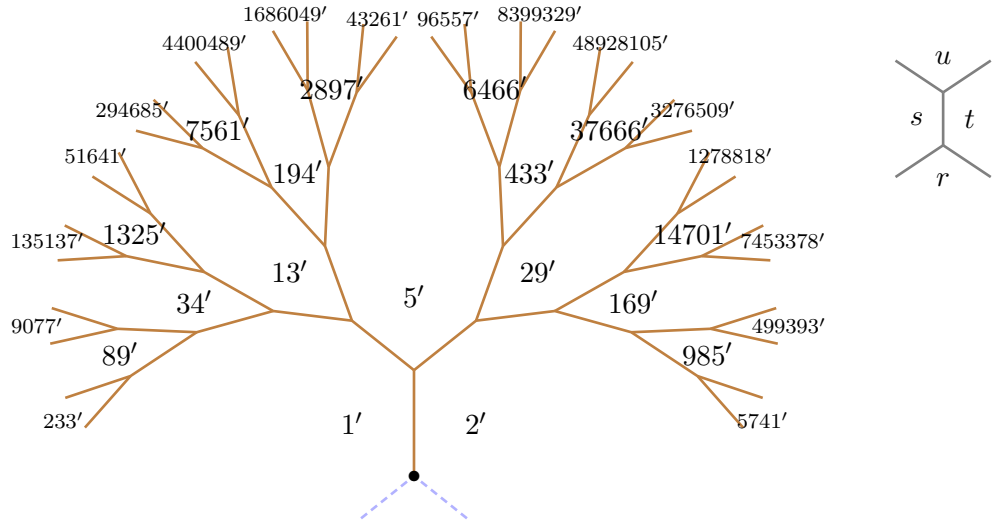
\begin{figure}[ht]
\centering

\psset{arrowscale=1.4,arrowinset=0.3,arrowlength=1.1}
\newrgbcolor{light}{0.8 0.8 1.0}
\newrgbcolor{pale}{1 0.7 1}
\newrgbcolor{pale}{1 0.7 0.4}

\begin{pspicture}(-15,-1.6)(15,19) 

\psset{arrowscale=1.8,arrowinset=0.3,arrowlength=1.1}
\newrgbcolor{light}{0.9 0.7 1.0}
\newrgbcolor{blue2}{0.3 0.3 0.8}
\newrgbcolor{light}{0.7 0.7 1.0}



\psset{linecolor=brown}

\psline(0,0)(0,4)

\psline[linecolor=light,linestyle=dashed,dash=3pt 2pt](-2,-1.6)(0,0)(2,-1.6)
\psdot[linecolor=black](0,0)

\psline(0,4)(2.34549,5.87047)(5.32318,6.23569)(8.21397,5.4336)(10.7154,3.77746)(12.438,1.82995)

\psline(10.7154,3.77746)(13.1809,2.95215)(13.1809,2.95215)(13.1809,2.95215)(13.1809,2.95215)(13.1809,2.95215)

\psline(8.21397,5.4336)(11.2111,5.56406)(13.7494,5.00099)(13.7494,5.00099)(13.7494,5.00099)(13.7494,5.00099)

\psline(11.2111,5.56406)(13.6909,6.34558)(13.6909,6.34558)(13.6909,6.34558)(13.6909,6.34558)(13.6909,6.34558)

\psline(5.32318,6.23569)(7.93444,7.71262)(10.8743,8.31034)(13.4694,8.15128)(13.4694,8.15128)(13.4694,8.15128)

\psline(10.8743,8.31034)(13.2013,9.47015)(13.2013,9.47015)(13.2013,9.47015)(13.2013,9.47015)(13.2013,9.47015)

\psline(7.93444,7.71262)(9.9615,9.92419)(12.1545,11.3209)(12.1545,11.3209)(12.1545,11.3209)(12.1545,11.3209)

\psline(9.9615,9.92419)(11.1623,12.2303)(11.1623,12.2303)(11.1623,12.2303)(11.1623,12.2303)(11.1623,12.2303)

\psline(2.34549,5.87047)(3.36415,8.69223)(5.38512,10.9094)(7.9923,12.3935)(10.5078,13.0511)(10.5078,13.0511)

\psline(7.9923,12.3935)(9.84196,14.2207)(9.84196,14.2207)(9.84196,14.2207)(9.84196,14.2207)(9.84196,14.2207)

\psline(5.38512,10.9094)(6.62203,13.6425)(8.27056,15.6531)(8.27056,15.6531)(8.27056,15.6531)(8.27056,15.6531)

\psline(6.62203,13.6425)(7.04442,16.208)(7.04442,16.208)(7.04442,16.208)(7.04442,16.208)(7.04442,16.208)

\psline(3.36415,8.69223)(3.22543,11.689)(4.01956,14.582)(5.33328,16.8257)(5.33328,16.8257)(5.33328,16.8257)

\psline(4.01956,14.582)(4.03543,17.182)(4.03543,17.182)(4.03543,17.182)(4.03543,17.182)(4.03543,17.182)

\psline(3.22543,11.689)(2.16744,14.4963)(1.91145,17.0836)(1.91145,17.0836)(1.91145,17.0836)(1.91145,17.0836)

\psline(2.16744,14.4963)(0.652066,16.609)(0.652066,16.609)(0.652066,16.609)(0.652066,16.609)(0.652066,16.609)

\psline(0,4)(-2.34549,5.87047)(-3.36415,8.69223)(-3.22543,11.689)(-2.16744,14.4963)(-0.652066,16.609)

\psline(-2.16744,14.4963)(-1.91145,17.0836)(-1.91145,17.0836)(-1.91145,17.0836)(-1.91145,17.0836)(-1.91145,17.0836)

\psline(-3.22543,11.689)(-4.01956,14.582)(-4.03543,17.182)(-4.03543,17.182)(-4.03543,17.182)(-4.03543,17.182)

\psline(-4.01956,14.582)(-5.33328,16.8257)(-5.33328,16.8257)(-5.33328,16.8257)(-5.33328,16.8257)(-5.33328,16.8257)

\psline(-3.36415,8.69223)(-5.38512,10.9094)(-6.62203,13.6425)(-7.04442,16.208)(-7.04442,16.208)(-7.04442,16.208)

\psline(-6.62203,13.6425)(-8.27056,15.6531)(-8.27056,15.6531)(-8.27056,15.6531)(-8.27056,15.6531)(-8.27056,15.6531)

\psline(-5.38512,10.9094)(-7.9923,12.3935)(-9.84196,14.2207)(-9.84196,14.2207)(-9.84196,14.2207)(-9.84196,14.2207)

\psline(-7.9923,12.3935)(-10.5078,13.0511)(-10.5078,13.0511)(-10.5078,13.0511)(-10.5078,13.0511)(-10.5078,13.0511)

\psline(-2.34549,5.87047)(-5.32318,6.23569)(-7.93444,7.71262)(-9.9615,9.92419)(-11.1623,12.2303)(-11.1623,12.2303)

\psline(-9.9615,9.92419)(-12.1545,11.3209)(-12.1545,11.3209)(-12.1545,11.3209)(-12.1545,11.3209)(-12.1545,11.3209)

\psline(-7.93444,7.71262)(-10.8743,8.31034)(-13.2013,9.47015)(-13.2013,9.47015)(-13.2013,9.47015)(-13.2013,9.47015)

\psline(-10.8743,8.31034)(-13.4694,8.15128)(-13.4694,8.15128)(-13.4694,8.15128)(-13.4694,8.15128)(-13.4694,8.15128)

\psline(-5.32318,6.23569)(-8.21397,5.4336)(-11.2111,5.56406)(-13.6909,6.34558)(-13.6909,6.34558)(-13.6909,6.34558)

\psline(-11.2111,5.56406)(-13.7494,5.00099)(-13.7494,5.00099)(-13.7494,5.00099)(-13.7494,5.00099)(-13.7494,5.00099)

\psline(-8.21397,5.4336)(-10.7154,3.77746)(-13.1809,2.95215)(-13.1809,2.95215)(-13.1809,2.95215)(-13.1809,2.95215)

\psline(-10.7154,3.77746)(-12.438,1.82995)(-12.438,1.82995)(-12.438,1.82995)(-12.438,1.82995)(-12.438,1.82995)

\rput(-2.34549,2){$1'$}
\rput(2.34549,2){$2'$}

\rput(-14.2083,5.69452){$_{9077'}$}
\rput(-13.8141,8.90805){$_{135137'}$}
\rput(-13.2168,2.12131){$_{233'}$}
\rput(-11.9886,12.1357){$_{51641'}$}
\rput(-11.1048,4.6315){$89'$}
\rput(-10.5995,13.8776){$_{294685'}$}
\rput(-10.5457,9.18954){$1325'$}
\rput(-8.30087,6.6009){$34'$}
\rput(-7.85893,16.3757){$_{4400489'}$}
\rput(-7.40607,13.1265){$7561'$}
\rput(-4.81368,17.475){$_{1686049'}$}
\rput(-4.38281,11.514){$194'$}
\rput(-4.69099,7.74094){$13'$} 
\rput(-3.0867,14.6858){$2897'$}
\rput(-1.34,17.3035){$_{43261'}$}  
\rput(0.,6.8){$5'$}
\rput(1.34,17.3035){$_{96557'}$} 
\rput(3.0867,14.6858){$6466'$}
\rput(4.69099,7.74094){$29'$}
\rput(4.38281,11.514){$433'$}
\rput(4.81368,17.475){$_{8399329'}$}
\rput(7.40607,13.1265){$37666'$}
\rput(7.85893,16.3757){$_{48928105'}$}
\rput(8.30087,6.6009){$169'$}
\rput(10.5457,9.18954){$14701'$}
\rput(10.5995,13.8776){$_{3276509'}$}
\rput(11.1048,4.6315){$985'$}
\rput(11.9886,12.1357){$_{1278818'}$}
\rput(13.2168,2.12131){$_{5741'}$}
\rput(13.9941,8.90805){$_{7453378'}$}
\rput(14.2083,5.69452){$_{499393'}$}

\psset{linecolor=gray}
\rput(20,13){%
        \begin{pspicture}(-2,-1.5)(2,4.5)
          \psline(1.8,-0.2)(0,1)
\psline(-1.8,-0.2)(0,1)
\psline(0,1)(0,3)

\psline(-1.8,4.2)(0,3)
\psline(1.8,4.2)(0,3)

\rput(0,-0.3){$r$}
\rput(1,1.99){$t$}
\rput(-1,1.99){$s$}
\rput(0,4.3){$u$}

        \end{pspicture}}

\psset{linecolor=black}

\end{pspicture}
\caption{The beginning of the original Markov tree. Here $x'$ means $3x$.}
\label{mart}
\end{figure}
For a fixed $D$, every solution $(x,y,z)$ of \e{mar2}  gives rise to three others, from the alternate solutions of the quadratic equations in each variable, producing an infinite tree. Beginning with $(3,15,6)$ at the bottom gives the rooted Markov tree displayed in Figure \ref{mart}. 
The three region labels surrounding each vertex give a solution to \e{mar2} for $D=0$. Four regions $r, s, t, u$ surrounding an edge, as on the right of Figure \ref{mart}, must satisfy the local rule $r+u = st$. 

We may replace $(3,15,6)$ with any triple of real numbers $(s_0,u_0,t_0)$ with $2<s_0, t_0 \lqs u_0$, find $D$ from \e{mar2}, and fill out the tree using the same local rule. As before, the three regions at each vertex give a solution to \e{mar2} for this fixed $D$. This  gives an example of a  {\em rising} tree. See Section \ref{asc} for the full definition. Letting $s_0$ or $t_0$ be $2$ above gives a {\em progressing} tree which has some different properties. 
To index the regions in each such tree $T$ we may use the rational numbers $q$ in the interval $[0,1]$, with their positions in the Farey tree shown in Figure \ref{farey}. 
\SpecialCoor
\psset{griddots=5,subgriddiv=0,gridlabels=0pt}
\psset{xunit=0.35cm, yunit=0.35cm, runit=0.5cm}
\psset{linewidth=1pt}
\psset{dotsize=4pt 0,dotstyle=*}
\begin{figure}[htb]
\centering

\psset{arrowscale=1.4,arrowinset=0.3,arrowlength=1.1}
\newrgbcolor{light}{0.8 0.8 1.0}
\newrgbcolor{pale}{1 0.7 1}
\newrgbcolor{pale}{1 0.7 0.4}

\begin{pspicture}(-15,-3)(15,19) 

\psset{arrowscale=1.8,arrowinset=0.3,arrowlength=1.1}
\newrgbcolor{light}{0.9 0.7 1.0}
\newrgbcolor{light}{0.7 0.7 1.0}
\newrgbcolor{blue2}{0.3 0.3 0.8}
\newrgbcolor{darkbrown}{0.5 0.324219 0.257813}



\psset{linecolor=light}
\psline{->}(-3,-2)(0,0)
\psline{->}(0,0)(3,-2)

\psdot[linecolor=black](0,0)

\psset{linecolor=black}

\psline{->}(0,0)(0,4)

\psset{linecolor=darkbrown}

\psline(0,4)(2.34549,5.87047)(5.32318,6.23569)(8.21397,5.4336)(10.7154,3.77746)(12.438,1.82995)

\psline(10.7154,3.77746)(13.1809,2.95215)(13.1809,2.95215)(13.1809,2.95215)(13.1809,2.95215)(13.1809,2.95215)

\psline(8.21397,5.4336)(11.2111,5.56406)(13.7494,5.00099)(13.7494,5.00099)(13.7494,5.00099)(13.7494,5.00099)

\psline(11.2111,5.56406)(13.6909,6.34558)(13.6909,6.34558)(13.6909,6.34558)(13.6909,6.34558)(13.6909,6.34558)

\psline(5.32318,6.23569)(7.93444,7.71262)(10.8743,8.31034)(13.4694,8.15128)(13.4694,8.15128)(13.4694,8.15128)

\psline(10.8743,8.31034)(13.2013,9.47015)(13.2013,9.47015)(13.2013,9.47015)(13.2013,9.47015)(13.2013,9.47015)

\psline(7.93444,7.71262)(9.9615,9.92419)(12.1545,11.3209)(12.1545,11.3209)(12.1545,11.3209)(12.1545,11.3209)

\psline(9.9615,9.92419)(11.1623,12.2303)(11.1623,12.2303)(11.1623,12.2303)(11.1623,12.2303)(11.1623,12.2303)

\psline(2.34549,5.87047)(3.36415,8.69223)(5.38512,10.9094)(7.9923,12.3935)(10.5078,13.0511)(10.5078,13.0511)

\psline(7.9923,12.3935)(9.84196,14.2207)(9.84196,14.2207)(9.84196,14.2207)(9.84196,14.2207)(9.84196,14.2207)

\psline(5.38512,10.9094)(6.62203,13.6425)(8.27056,15.6531)(8.27056,15.6531)(8.27056,15.6531)(8.27056,15.6531)

\psline(6.62203,13.6425)(7.04442,16.208)(7.04442,16.208)(7.04442,16.208)(7.04442,16.208)(7.04442,16.208)

\psline(3.36415,8.69223)(3.22543,11.689)(4.01956,14.582)(5.33328,16.8257)(5.33328,16.8257)(5.33328,16.8257)

\psline(4.01956,14.582)(4.03543,17.182)(4.03543,17.182)(4.03543,17.182)(4.03543,17.182)(4.03543,17.182)

\psline(3.22543,11.689)(2.16744,14.4963)(1.91145,17.0836)(1.91145,17.0836)(1.91145,17.0836)(1.91145,17.0836)

\psline(2.16744,14.4963)(0.652066,16.609)(0.652066,16.609)(0.652066,16.609)(0.652066,16.609)(0.652066,16.609)

\psline(0,4)(-2.34549,5.87047)(-3.36415,8.69223)(-3.22543,11.689)(-2.16744,14.4963)(-0.652066,16.609)

\psline(-2.16744,14.4963)(-1.91145,17.0836)(-1.91145,17.0836)(-1.91145,17.0836)(-1.91145,17.0836)(-1.91145,17.0836)

\psline(-3.22543,11.689)(-4.01956,14.582)(-4.03543,17.182)(-4.03543,17.182)(-4.03543,17.182)(-4.03543,17.182)

\psline(-4.01956,14.582)(-5.33328,16.8257)(-5.33328,16.8257)(-5.33328,16.8257)(-5.33328,16.8257)(-5.33328,16.8257)

\psline(-3.36415,8.69223)(-5.38512,10.9094)(-6.62203,13.6425)(-7.04442,16.208)(-7.04442,16.208)(-7.04442,16.208)

\psline(-6.62203,13.6425)(-8.27056,15.6531)(-8.27056,15.6531)(-8.27056,15.6531)(-8.27056,15.6531)(-8.27056,15.6531)

\psline(-5.38512,10.9094)(-7.9923,12.3935)(-9.84196,14.2207)(-9.84196,14.2207)(-9.84196,14.2207)(-9.84196,14.2207)

\psline(-7.9923,12.3935)(-10.5078,13.0511)(-10.5078,13.0511)(-10.5078,13.0511)(-10.5078,13.0511)(-10.5078,13.0511)

\psline(-2.34549,5.87047)(-5.32318,6.23569)(-7.93444,7.71262)(-9.9615,9.92419)(-11.1623,12.2303)(-11.1623,12.2303)

\psline(-9.9615,9.92419)(-12.1545,11.3209)(-12.1545,11.3209)(-12.1545,11.3209)(-12.1545,11.3209)(-12.1545,11.3209)

\psline(-7.93444,7.71262)(-10.8743,8.31034)(-13.2013,9.47015)(-13.2013,9.47015)(-13.2013,9.47015)(-13.2013,9.47015)

\psline(-10.8743,8.31034)(-13.4694,8.15128)(-13.4694,8.15128)(-13.4694,8.15128)(-13.4694,8.15128)(-13.4694,8.15128)

\psline(-5.32318,6.23569)(-8.21397,5.4336)(-11.2111,5.56406)(-13.6909,6.34558)(-13.6909,6.34558)(-13.6909,6.34558)

\psline(-11.2111,5.56406)(-13.7494,5.00099)(-13.7494,5.00099)(-13.7494,5.00099)(-13.7494,5.00099)(-13.7494,5.00099)

\psline(-8.21397,5.4336)(-10.7154,3.77746)(-13.1809,2.95215)(-13.1809,2.95215)(-13.1809,2.95215)(-13.1809,2.95215)

\psline(-10.7154,3.77746)(-12.438,1.82995)(-12.438,1.82995)(-12.438,1.82995)(-12.438,1.82995)(-12.438,1.82995)

\rput(-2.34549,2){$\frac{0}{1}$}
\rput(2.34549,2){$\frac{1}{1}$}

\rput(-14.2083,5.69452){$\frac{2}{9}$}
\rput(-13.8141,8.90805){$\frac{3}{11}$}
\rput(-13.2168,2.12131){$_{\frac{1}{6}}$}
\rput(-11.9886,12.1357){$\frac{3}{10}$}
\rput(-11.1048,4.6315){$\frac{1}{5}$}
\rput(-10.5995,13.8776){$\frac{4}{11}$}
\rput(-10.5457,9.18954){$\frac{2}{7}$}
\rput(-8.30087,6.6009){$\frac{1}{4}$}
\rput(-7.85893,16.3757){$\frac{5}{13}$}
\rput(-7.40607,13.1265){$\frac{3}{8}$}
\rput(-4.81368,17.475){$\frac{5}{12}$}
\rput(-4.38281,11.514){$\frac{2}{5}$}
\rput(-4.69099,7.74094){$\frac{1}{3}$} 
\rput(-3.0867,14.6858){$\frac{3}{7}$}
\rput(-1.3,17.3035){$\frac{4}{9}$}  
\rput(0.,6.8){$\frac{1}{2}$}
\rput(1.3,17.3035){$\frac{5}{9}$} 
\rput(3.0867,14.6858){$\frac{4}{7}$}
\rput(4.69099,7.74094){$\frac{2}{3}$}
\rput(4.38281,11.514){$\frac{3}{5}$}
\rput(4.81368,17.475){$\frac{7}{12}$}
\rput(7.40607,13.1265){$\frac{5}{8}$}
\rput(7.85893,16.3757){$\frac{8}{13}$}
\rput(8.30087,6.6009){$\frac{3}{4}$}
\rput(10.5457,9.18954){$\frac{5}{7}$}
\rput(10.5995,13.8776){$\frac{7}{11}$}
\rput(11.1048,4.6315){$\frac{4}{5}$}
\rput(11.9886,12.1357){$\frac{7}{10}$}
\rput(13.2168,2.12131){$\frac{5}{6}$}
\rput(13.9941,8.90805){$\frac{8}{11}$}
\rput(14.2083,5.69452){$\frac{7}{9}$}

\rput(0,-2){$\frac 10$}

\psset{linecolor=gray}
\rput(20,13){%
        \begin{pspicture}(1,-1)(7,8)

\psline(4,5)(6,6.333)
\psline(4,5)(2,6.333)

\psline{->}(4,2)(4,5)

\rput(3,3.5){$\frac{a}{b}$}

\rput(5,3.5){$\frac{c}{d}$}

\rput(4,7){$\frac{a+c}{b+d}$}

\end{pspicture}}

\psset{linecolor=black}

\end{pspicture}
\caption{The Farey tree}
\label{farey}
\end{figure}
Originally,  Frobenius in \cite[Sect. 8]{fro} used indices from the Stern-Brocot tree,  containing the Farey tree as its left branch.
The notation $m(T)_q=m_q$, for $q\in \Q \cap [0,1]$, indicates the tree number in position $q$. For example, overlaying  Figure \ref{mart} 
on Figure \ref{farey} gives
\begin{equation} \la{mkno}
  m_0= 1'=3, \quad m_1 = 2'=6, \quad m_{1/2} = 5'=15, \quad m_{1/3} = 13', \quad m_{2/3} = 29',
\end{equation}
and so on. The reason this is so useful is because the additive mediants in the Farey tree correspond to logarithms of the multiplicative numbers in $T$ with $u \approx st$. For example, with the Markov tree in Figure \ref{mart} it may be verified numerically that
\begin{align}
  \log m_{k/n} &\approx 0.8 k +  0.962 n, \notag\\
  & = n(0.8 k/n +  0.962). \label{mv}
\end{align}
Our first main result is the following. 

\begin{theorem} \la{mthm}
Let $T$ be a rising tree or a progressing tree, with Farey-indexed numbers $m_{k/n}=m(T)_{k/n}$.  There is a continuous function $\Phi=\Phi_T: [0,1] \to \R_{\gqs 0}$ that encodes these numbers  with
\begin{equation}\label{mknc}
  m_{k/n}=2 \cosh(n \cdot \Phi(k/n)).
\end{equation}
The  graph of $\Phi$ is strictly convex if $D<4$, a straight line if $D=4$, and strictly concave when $D>4$.  Also $\Phi$ is differentiable at irrational numbers, while at rationals it is only semi-differentiable if $D \neq 4$.
\end{theorem}

Figure \ref{phplot} displays the graph of the encoding function $\Phi$ for the original Markov tree. For this, a recursive procedure using the local rule computed region numbers, up to a given size, and their corresponding mediants.
The graph is strictly convex, lying a little below the line joining the endpoints. The points at $(k/n,\Phi(k/n))$ for small $n$ are highlighted. We will see in \e{sps} that the difference in the slopes of the graph to the left and right of $1/2$ is approximately $0.0359$.
Figure \ref{phplots} shows  $\Phi$ graphs for other rising trees. 

\SpecialCoor
\psset{griddots=5,subgriddiv=0,gridlabels=0pt}
\psset{xunit=7cm, yunit=6cm, runit=5cm}
\psset{linewidth=1pt}
\psset{dotsize=1.4pt,dotstyle=*}
\begin{figure}[htb]
\centering

\psset{arrowscale=1.4,arrowinset=0.3,arrowlength=1.1}
\newrgbcolor{light}{0.8 0.8 1.0}
\newrgbcolor{pale}{1 0.7 1}
\newrgbcolor{pale}{1 0.7 0.4}

\newrgbcolor{markov}{0.4 0.2 0.8}
\psset{linecolor=markov}

\begin{pspicture}(-0.2,0.8)(1.2,1.8) 

\psset{linecolor=lightgray}

\psline(-0.025,1.8)(0.025,1.8)
\psline(-0.025,1.6)(0.025,1.6)
\psline(-0.025,1.4)(0.025,1.4)
\psline(-0.025,1.2)(0.025,1.2)
\psline(-0.025,1)(0.025,1)

\rput(-0.07,1.8){$_{1.8}$}
\rput(-0.07,1.6){$_{1.6}$}
\rput(-0.07,1.4){$_{1.4}$}
\rput(-0.07,1.2){$_{1.2}$}
\rput(-0.07,1.0){$_{1.0}$}

\psline(0.2,0.925)(0.2,0.875)
\psline(0.4,0.925)(0.4,0.875)
\psline(0.6,0.925)(0.6,0.875)
\psline(0.8,0.925)(0.8,0.875)
\psline(1,0.925)(1,0.875)

\rput(0.2,0.84){$_{0.2}$}
\rput(0.4,0.84){$_{0.4}$}
\rput(0.6,0.84){$_{0.6}$}
\rput(0.8,0.84){$_{0.8}$}
\rput(1.0,0.84){$_{1.0}$}

\psline(0,0.85)(0,1.8)


\psline(-0.1,0.9)(1.1,0.9)

\psset{dotsize=1.3pt,dotstyle=*}
\psset{linewidth=0.9pt}


\savedata{\mydata}[
{{0., 0.962424}, {0.1, 1.03993}, {0.111111, 1.04855}, {0.125, 
  1.05931}, {0.142857, 1.07315}, {0.166667, 1.09161}, {0.2, 
  1.11745}, {0.222222, 1.13468}, {0.25, 1.15622}, {0.285714, 
  1.18397}, {0.3, 1.19507}, {0.333333, 1.22097}, {0.375, 
  1.25367}, {0.4, 1.27329}, {0.428571, 1.29572}, {0.444444, 
  1.30818}, {0.5, 1.35179}, {0.555556, 1.39739}, {0.571429, 
  1.41042}, {0.6, 1.43387}, {0.625, 1.45439}, {0.666667, 
  1.48859}, {0.7, 1.51601}, {0.714286, 1.52775}, {0.75, 
  1.55713}, {0.777778, 1.57997}, {0.8, 1.59825}, {0.833333, 
  1.62567}, {0.857143, 1.64525}, {0.875, 1.65994}, {0.888889, 
  1.67136}, {0.9, 1.6805}, {1., 1.76275}}
  ]
  
\dataplot[linecolor=cyan]{\mydata}

\savedata{\mydatab}[
{{0., 0.962424}, {1., 1.76275}, {0.5, 1.35179}, {0.333333, 
  1.22097}, {0.666667, 1.48859}, {0.25, 1.15622}, {0.4, 
  1.27329}, {0.6, 1.43387}, {0.75, 1.55713}, {0.2, 
  1.11745}, {0.285714, 1.18397}, {0.375, 1.25367}, {0.428571, 
  1.29572}, {0.571429, 1.41042}, {0.625, 1.45439}, {0.714286, 
  1.52775}, {0.8, 1.59825}, {0.166667, 1.09161}, {0.222222, 
  1.13468}, {0.272727, 1.17388}, {0.3, 1.19507}, {0.363636, 
  1.24475}, {0.384615, 1.26122}, {0.416667, 1.28638}, {0.444444, 
  1.30818}, {0.555556, 1.39739}, {0.583333, 1.42019}, {0.615385, 
  1.4465}, {0.636364, 1.46372}, {0.7, 1.51601}, {0.727273, 
  1.53844}, {0.777778, 1.57997}, {0.833333, 1.62567}, {0.142857, 
  1.07315}, {0.181818, 1.10335}, {0.214286, 1.12852}, {0.230769, 
  1.14131}, {0.266667, 1.16917}, {0.277778, 1.1778}, {0.294118, 
  1.1905}, {0.307692, 1.20105}, {0.357143, 1.23966}, {0.368421, 
  1.24851}, {0.380952, 1.25834}, {0.388889, 1.26457}, {0.411765, 
  1.28253}, {0.421053, 1.28982}, {0.4375, 1.30273}, {0.454545, 
  1.31611}, {0.545455, 1.3891}, {0.5625, 1.40309}, {0.578947, 
  1.41659}, {0.588235, 1.42421}, {0.611111, 1.44299}, {0.619048, 
  1.4495}, {0.631579, 1.45979}, {0.642857, 1.46905}, {0.692308, 
  1.50968}, {0.705882, 1.52084}, {0.722222, 1.53428}, {0.733333, 
  1.54342}, {0.769231, 1.57294}, {0.785714, 1.5865}, {0.818182, 
  1.6132}, {0.857143, 1.64525}, {0.125, 1.05931}, {0.153846, 
  1.08167}, {0.176471, 1.09921}, {0.1875, 1.10776}, {0.210526, 
  1.12561}, {0.217391, 1.13093}, {0.227273, 1.1386}, {0.235294, 
  1.14482}, {0.263158, 1.16644}, {0.269231, 1.17116}, {0.275862, 
  1.17631}, {0.28, 1.17953}, {0.291667, 1.18859}, {0.296296, 
  1.19219}, {0.304348, 1.19845}, {0.3125, 1.20478}, {0.352941, 
  1.23636}, {0.36, 1.2419}, {0.366667, 1.24713}, {0.37037, 
  1.25004}, {0.37931, 1.25705}, {0.382353, 1.25944}, {0.387097, 
  1.26317}, {0.391304, 1.26647}, {0.409091, 1.28043}, {0.413793, 
  1.28412}, {0.419355, 1.28849}, {0.423077, 1.29141}, {0.434783, 
  1.3006}, {0.44, 1.30469}, {0.45, 1.31254}, {0.461538, 
  1.3216}, {0.538462, 1.38336}, {0.55, 1.39283}, {0.56, 
  1.40104}, {0.565217, 1.40532}, {0.576923, 1.41493}, {0.580645, 
  1.41798}, {0.586207, 1.42255}, {0.590909, 1.42641}, {0.608696, 
  1.44101}, {0.612903, 1.44446}, {0.617647, 1.44836}, {0.62069, 
  1.45085}, {0.62963, 1.45819}, {0.633333, 1.46123}, {0.64, 
  1.4667}, {0.647059, 1.4725}, {0.6875, 1.50573}, {0.695652, 
  1.51243}, {0.703704, 1.51905}, {0.708333, 1.52286}, {0.72, 
  1.53245}, {0.724138, 1.53586}, {0.730769, 1.54131}, {0.736842, 
  1.54631}, {0.764706, 1.56922}, {0.772727, 1.57582}, {0.782609, 
  1.58395}, {0.789474, 1.58959}, {0.8125, 1.60853}, {0.823529, 
  1.6176}, {0.846154, 1.63621}, {0.875, 1.65994}, {0.111111, 
  1.04855}, {0.133333, 1.06577}, {0.15, 1.07869}, {0.157895, 
  1.08481}, {0.173913, 1.09723}, {0.178571, 1.10084}, {0.185185, 
  1.10596}, {0.190476, 1.11006}, {0.208333, 1.12391}, {0.212121, 
  1.12685}, {0.216216, 1.13002}, {0.21875, 1.13199}, {0.225806, 
  1.13746}, {0.228571, 1.1396}, {0.233333, 1.1433}, {0.238095, 
  1.14699}, {0.26087, 1.16466}, {0.264706, 1.16765}, {0.268293, 
  1.17043}, {0.27027, 1.17197}, {0.275, 1.17564}, {0.276596, 
  1.17688}, {0.27907, 1.17881}, {0.28125, 1.1805}, {0.290323, 
  1.18755}, {0.292683, 1.18938}, {0.295455, 1.19154}, {0.297297, 
  1.19297}, {0.30303, 1.19742}, {0.305556, 1.19938}, {0.310345, 
  1.20311}, {0.315789, 1.20734}, {0.35, 1.23405}, {0.354839, 
  1.23785}, {0.358974, 1.24109}, {0.361111, 1.24277}, {0.365854, 
  1.24649}, {0.367347, 1.24766}, {0.369565, 1.24941}, {0.371429, 
  1.25087}, {0.378378, 1.25632}, {0.38, 1.2576}, {0.381818, 
  1.25902}, {0.382979, 1.25993}, {0.386364, 1.26259}, {0.387755, 
  1.26368}, {0.390244, 1.26564}, {0.392857, 1.26769}, {0.407407, 
  1.27911}, {0.410256, 1.28134}, {0.413043, 1.28353}, {0.414634, 
  1.28478}, {0.418605, 1.2879}, {0.42, 1.28899}, {0.422222, 
  1.29074}, {0.424242, 1.29232}, {0.433333, 1.29946}, {0.435897, 
  1.30147}, {0.439024, 1.30393}, {0.441176, 1.30561}, {0.448276, 
  1.31119}, {0.451613, 1.31381}, {0.458333, 1.31908}, {0.466667, 
  1.32562}, {0.533333, 1.37915}, {0.541667, 1.38599}, {0.548387, 
  1.3915}, {0.551724, 1.39424}, {0.558824, 1.40007}, {0.560976, 
  1.40184}, {0.564103, 1.4044}, {0.566667, 1.40651}, {0.575758, 
  1.41397}, {0.577778, 1.41563}, {0.58, 1.41745}, {0.581395, 
  1.4186}, {0.585366, 1.42186}, {0.586957, 1.42316}, {0.589744, 
  1.42545}, {0.592593, 1.42779}, {0.607143, 1.43973}, {0.609756, 
  1.44188}, {0.612245, 1.44392}, {0.613636, 1.44506}, {0.617021, 
  1.44784}, {0.618182, 1.44879}, {0.62, 1.45029}, {0.621622, 
  1.45162}, {0.628571, 1.45732}, {0.630435, 1.45885}, {0.632653, 
  1.46067}, {0.634146, 1.4619}, {0.638889, 1.46579}, {0.641026, 
  1.46755}, {0.645161, 1.47094}, {0.65, 1.47491}, {0.684211, 
  1.50302}, {0.689655, 1.5075}, {0.694444, 1.51144}, {0.69697, 
  1.51351}, {0.702703, 1.51823}, {0.704545, 1.51974}, {0.707317, 
  1.52202}, {0.709677, 1.52396}, {0.71875, 1.53143}, {0.72093, 
  1.53322}, {0.723404, 1.53525}, {0.725, 1.53657}, {0.72973, 
  1.54046}, {0.731707, 1.54208}, {0.735294, 1.54503}, {0.73913, 
  1.54819}, {0.761905, 1.56692}, {0.766667, 1.57083}, {0.771429, 
  1.57475}, {0.774194, 1.57703}, {0.78125, 1.58283}, {0.783784, 
  1.58491}, {0.787879, 1.58828}, {0.791667, 1.5914}, {0.809524, 
  1.60608}, {0.814815, 1.61044}, {0.821429, 1.61588}, {0.826087, 
  1.61971}, {0.842105, 1.63288}, {0.85, 1.63937}, {0.866667, 
  1.65308}, {0.888889, 1.67136}, {0.1, 1.03993}, {0.117647, 
  1.05361}, {0.130435, 1.06352}, {0.136364, 1.06812}, {0.148148, 
  1.07725}, {0.151515, 1.07986}, {0.15625, 1.08353}, {0.16, 
  1.08644}, {0.172414, 1.09606}, {0.175, 1.09807}, {0.177778, 
  1.10022}, {0.179487, 1.10155}, {0.184211, 1.10521}, {0.186047, 
  1.10663}, {0.189189, 1.10907}, {0.192308, 1.11148}, {0.206897, 
  1.12279}, {0.209302, 1.12466}, {0.211538, 1.12639}, {0.212766, 
  1.12735}, {0.215686, 1.12961}, {0.216667, 1.13037}, {0.218182, 
  1.13155}, {0.219512, 1.13258}, {0.225, 1.13683}, {0.226415, 
  1.13793}, {0.22807, 1.13921}, {0.229167, 1.14006}, {0.232558, 
  1.14269}, {0.234043, 1.14385}, {0.236842, 1.14602}, {0.24, 
  1.14846}, {0.259259, 1.16341}, {0.261905, 1.16547}, {0.264151, 
  1.16721}, {0.265306, 1.16811}, {0.267857, 1.17009}, {0.268657, 
  1.17072}, {0.269841, 1.17164}, {0.270833, 1.17241}, {0.27451, 
  1.17526}, {0.275362, 1.17593}, {0.276316, 1.17667}, {0.276923, 
  1.17714}, {0.278689, 1.17851}, {0.279412, 1.17907}, {0.280702, 
  1.18007}, {0.282051, 1.18112}, {0.289474, 1.18689}, {0.290909, 
  1.188}, {0.292308, 1.18909}, {0.293103, 1.18971}, {0.295082, 
  1.19125}, {0.295775, 1.19179}, {0.296875, 1.19264}, {0.297872, 
  1.19342}, {0.302326, 1.19688}, {0.303571, 1.19784}, {0.305085, 
  1.19902}, {0.306122, 1.19983}, {0.309524, 1.20247}, {0.311111, 
  1.2037}, {0.314286, 1.20617}, {0.318182, 1.2092}, {0.347826, 
  1.23234}, {0.351351, 1.23511}, {0.354167, 1.23732}, {0.355556, 
  1.23841}, {0.358491, 1.24071}, {0.359375, 1.24141}, {0.360656, 
  1.24241}, {0.361702, 1.24323}, {0.365385, 1.24612}, {0.366197, 
  1.24676}, {0.367089, 1.24746}, {0.367647, 1.2479}, {0.369231, 
  1.24914}, {0.369863, 1.24964}, {0.370968, 1.25051}, {0.372093, 
  1.25139}, {0.377778, 1.25585}, {0.378788, 1.25664}, {0.379747, 
  1.2574}, {0.380282, 1.25782}, {0.381579, 1.25884}, {0.382022, 
  1.25918}, {0.382716, 1.25973}, {0.383333, 1.26021}, {0.385965, 
  1.26228}, {0.386667, 1.26283}, {0.3875, 1.26348}, {0.38806, 
  1.26392}, {0.389831, 1.26531}, {0.390625, 1.26594}, {0.392157, 
  1.26714}, {0.393939, 1.26854}, {0.40625, 1.2782}, {0.408163, 
  1.2797}, {0.409836, 1.28101}, {0.410714, 1.2817}, {0.412698, 
  1.28326}, {0.413333, 1.28376}, {0.414286, 1.28451}, {0.415094, 
  1.28514}, {0.418182, 1.28757}, {0.418919, 1.28814}, {0.419753, 
  1.2888}, {0.42029, 1.28922}, {0.421875, 1.29046}, {0.422535, 
  1.29098}, {0.423729, 1.29192}, {0.425, 1.29292}, {0.432432, 
  1.29875}, {0.433962, 1.29995}, {0.435484, 1.30115}, {0.436364, 
  1.30184}, {0.438596, 1.30359}, {0.439394, 1.30422}, {0.440678, 
  1.30522}, {0.44186, 1.30615}, {0.447368, 1.31048}, {0.44898, 
  1.31174}, {0.45098, 1.31331}, {0.452381, 1.31441}, {0.457143, 
  1.31815}, {0.459459, 1.31997}, {0.464286, 1.32375}, {0.470588, 
  1.3287}, {0.529412, 1.37593}, {0.535714, 1.3811}, {0.540541, 
  1.38506}, {0.542857, 1.38697}, {0.547619, 1.39087}, {0.54902, 
  1.39202}, {0.55102, 1.39367}, {0.552632, 1.39499}, {0.55814, 
  1.39951}, {0.559322, 1.40048}, {0.560606, 1.40153}, {0.561404, 
  1.40219}, {0.563636, 1.40402}, {0.564516, 1.40474}, {0.566038, 
  1.40599}, {0.567568, 1.40725}, {0.575, 1.41335}, {0.576271, 
  1.41439}, {0.577465, 1.41537}, {0.578125, 1.41591}, {0.57971, 
  1.41722}, {0.580247, 1.41766}, {0.581081, 1.41834}, {0.581818, 
  1.41895}, {0.584906, 1.42148}, {0.585714, 1.42214}, {0.586667, 
  1.42293}, {0.587302, 1.42345}, {0.589286, 1.42508}, {0.590164, 
  1.4258}, {0.591837, 1.42717}, {0.59375, 1.42874}, {0.606061, 
  1.43884}, {0.607843, 1.44031}, {0.609375, 1.44157}, {0.610169, 
  1.44222}, {0.61194, 1.44367}, {0.6125, 1.44413}, {0.613333, 
  1.44481}, {0.614035, 1.44539}, {0.616667, 1.44755}, {0.617284, 
  1.44806}, {0.617978, 1.44863}, {0.618421, 1.44899}, {0.619718, 
  1.45006}, {0.620253, 1.45049}, {0.621212, 1.45128}, {0.622222, 
  1.45211}, {0.627907, 1.45678}, {0.629032, 1.4577}, {0.630137, 
  1.45861}, {0.630769, 1.45913}, {0.632353, 1.46043}, {0.632911, 
  1.46088}, {0.633803, 1.46162}, {0.634615, 1.46228}, {0.638298, 
  1.46531}, {0.639344, 1.46616}, {0.640625, 1.46722}, {0.641509, 
  1.46794}, {0.644444, 1.47035}, {0.645833, 1.47149}, {0.648649, 
  1.4738}, {0.652174, 1.4767}, {0.681818, 1.50105}, {0.685714, 
  1.50426}, {0.688889, 1.50687}, {0.690476, 1.50817}, {0.693878, 
  1.51097}, {0.694915, 1.51182}, {0.696429, 1.51307}, {0.697674, 
  1.51409}, {0.702128, 1.51776}, {0.703125, 1.51858}, {0.704225, 
  1.51948}, {0.704918, 1.52005}, {0.706897, 1.52168}, {0.707692, 
  1.52233}, {0.709091, 1.52348}, {0.710526, 1.52466}, {0.717949, 
  1.53077}, {0.719298, 1.53188}, {0.720588, 1.53294}, {0.721311, 
  1.53353}, {0.723077, 1.53498}, {0.723684, 1.53548}, {0.724638, 
  1.53627}, {0.72549, 1.53697}, {0.729167, 1.53999}, {0.730159, 
  1.54081}, {0.731343, 1.54178}, {0.732143, 1.54244}, {0.734694, 
  1.54454}, {0.735849, 1.54549}, {0.738095, 1.54734}, {0.740741, 
  1.54951}, {0.76, 1.56535}, {0.763158, 1.56795}}
]

\dataplot[plotstyle=dots,linecolor=blue]{\mydatab}

\rput(0.7,1.7){$\Phi_T$}

\rput(0.8,1.15){$T=(3, 15, 6)$}
\rput(0.8,1.07){$D=0$}

\end{pspicture}
\caption{A very slightly crooked line: the graph of $\Phi$ for the Markov tree $(3,15,6)$.}
\label{phplot}
\end{figure}

The novelty of Theorem \ref{mthm} is that it is valid for all real $D$, and that the proof given is self contained 
 and does not require any hyperbolic geometry. 
Special cases  have appeared in the literature in several forms, based on the well-known interpretation of Markov numbers in terms of lengths of simple closed geodesics on the once-punctured torus. Fock in \cite[Sect. 7.3]{fo97} defined a similar function to $\Phi$, indexed with fractions in $[0,1/2]$, for the  original Markov tree and $D=0$. He showed it was continuous and convex using torus geometry. Sorrentino and Veselov went further in \cite{sove}, using the stable norm on the homology of the torus and Aubrey-Mather theory to show that Fock's function is differentiable at irrationals and not differentiable at rationals. Their methods should also apply for $D<0$ since that corresponds to a torus with a hole of geodesic length $2\cosh^{-1}(1-D/2)$, see \cite[p. 2154]{sove}.

Earlier, McShane and Rivin in \cite{mr95b} observed that the  boundary of the stable norm ball is strictly convex with corners in rational slope directions.  This corresponds to Theorem \ref{mthm} for $D=0$, and we explain the exact connection between these norm boundaries and $\Phi$ in Section \ref{norm}. 
The closest to our approach to Theorem \ref{mthm} is that of Hines in \cite[p. 6]{hin} who proves  it in the case when $D=0$. 
The author was at first unaware of all this work, and rediscovered the function $\Phi$ for the Markov tree in Figure \ref{mart} after noticing the striking regularity of the points  $\log(m_{k/n}) e^{\pi i k/n}$ in $\C$.

With Theorem \ref{mthm} in hand it is straightforward to explain  basic patterns appearing in the Markov numbers, and in the numbers on these kinds of trees in general. Further results in this paper come under the following headings.

\vskip 3mm
\noindent
$\bullet$ {\bf Counting tree numbers below a given bound.} 
The following corollary is a special case of results obtained in Sections \ref{num} and \ref{belo}.

\begin{cor} \la{zm}
For $C \approx 0.1807171$, which can be found to any accuracy, the number of Markov numbers $1, 2, 5, 13, 29, 34, 89, 169, \dots$ below $X$, counted with their multiplicity on the Markov tree, is
\begin{equation}\label{asy}
M(X) = C (\log X)^2 +O\left(\log X \log \log X \right).
\end{equation}
\end{cor}

 Corollary \ref{zm} first appeared in the work of Gurwood \cite{gur} and Zagier \cite{z82} with  slightly weaker error estimates. 
 McShane and Rivin proved \e{asy} as a consequence of the more general result \cite[Thm. 4.1]{mr95b}. The precise link between the geodesics counted in \cite[Thm. 4.1]{mr95b} and  \e{asy} is discussed in Section \ref{bowd}. Our Theorem \ref{zzzb} also extends the McShane-Rivin result to all  $D$.  The main term coefficient $C$ in \e{asy} equals $6\text{\rm Area}(\mathcal W)/\pi^2$ for $\mathcal W$ the wedge-shaped region  pictured in Figure \ref{wplot}. Zagier's formula for $C$ is generalized in Proposition \ref{wedg}.


\vskip 3mm
\noindent
$\bullet$ {\bf Uniqueness.} 
The uniqueness conjecture states that the largest number in any triple on the Markov tree $(3,15,6)$ determines the other two. This is equivalent to there being no repeated numbers on the tree, as seen in Section  \ref{niqu}.

Looking for repeats, we may take a line $\ell$ that passes through some lattice points $(x,y)=(n,k)\in \Z^2$ with $0\lqs k \lqs n$ and $n, k$ relatively prime.  Then the relative sizes of the numbers $m_{k/n}$ moving along this `index line' can be considered. Aigner did this and conjectured in \cite[Conj. 10.11]{aig} that the Markov numbers strictly increase with $x=n$ along lines of slope $0$ and $-1$, and strictly increase with $y=k$ along vertical lines. These statements were proved and generalized by several authors using a variety of techniques; see the discussions in the introductions to \cite{llrs23,ga22}. A more general conjecture was provided in \cite[Sect. 6]{llrs23}, with the main part of this proved by Gaster in \cite{ga22}. This is Corollary \ref{zak} that we in turn prove in a wider context in Section  \ref{niqu}. 
Set
\begin{equation} \la{sigs}
  \sigma(0)=-\frac{\log(\tfrac 12(3+\sqrt{5}))}{\log(\tfrac 3{10}(5+\sqrt{5}))} \approx -1.242, \qquad 
  \sigma(1)=-\frac{\log(\tfrac 34(2+\sqrt{2}))}{\log(\tfrac 2{3}(2+\sqrt{2}))} \approx -1.143.
\end{equation}

\begin{cor}  \cite[Thm. 1.1]{ga22} \la{zak}
The  Markov numbers  $m_{k/n}$ in \e{mkno} strictly increase with $y=k$ along index lines of slope $s$ with $-\infty \lqs s \lqs \sigma(0)$. 
They strictly increase with $x=n$ along  lines with $\sigma(1) \lqs s < \infty$. 
These statements are sharp in the sense that $\sigma(0)$ cannot be increased and $\sigma(1)$ cannot be decreased. 
\end{cor}

In Section \ref{d4} we find trees, with labels in a quadratic field, that do have uniqueness. Theorem \ref{thm-rat} in fact gives an exact criterion for when a tree with $D=4$ has uniqueness or repeated numbers.

\vskip 3mm
\noindent
$\bullet$ {\bf Identities.} 
As hinted at in Figure \ref{mart}, these trees may be continued past the root to obtain an infinite 3-regular tree, as in Figure \ref{topb}. In Conway's terminology these completed trees are {\em topographs}, and the same local rule $r+u=st$ is used to fill in all their regions. Bowditch used them to give a short proof of McShane's famous identity for geodesic lengths on any hyperbolic punctured torus. In Section \ref{tmc} we give a new proof of the generalization
\begin{equation} \label{pa}
  \sum_{ m  \in \mathcal G} \left[ \left(1+\frac{D/3}{m^2-D}\right)\left(\frac 2{e^{2 \cosh^{-1}(|m|/2)}+1}-1\right)+1\right] =1,
\end{equation}
 which is due to Hu,  Tan, and  Zhang in \cite{tan15}.
In \e{pa} we are summing over all regions $m$ of any topograph $\mathcal G$  with labels in $(-\infty,-2] \cup [2,\infty)$ and $D\neq 4$. The case $D=0$ corresponds to McShane's identity as explained in Section \ref{bowd}.  An identity of Hines is similarly extended to all $D$ in Section \ref{hiid}.

\vskip 3mm
\noindent
$\bullet$ {\bf Properties of $\Phi$.} 
Detailed properties of $\Phi$ are shown in Section \ref{y5}, with properties of its precursor function $\phi$ on rationals  developed in Section \ref{y4}. McShane and Rivin claimed in \cite[Thm. 2.1]{mr95b} that the  boundary of the stable norm ball is  infinitely flat in all irrational directions, though their promised proof never appeared. This translates into  $\Phi(x)$ being infinitely flat (all derivatives after the first being zero) at every irrational $x$. We prove this for a dense set of irrationals in Section \ref{ite}. However, the recent preprint  \cite{nxv} of  Doan,  Li, and  Nguyen shows that McShane and Rivin's claim is not true in general, and we give an explicit example of an irrational $x$ where $\Phi(x)$ is not infinitely flat, based on their work.

Section \ref{seri} looks at approximations to $\Phi$, $\Phi'$ and in particular how to evaluate them at irrationals. 
We find  series that give them at each irrational to remarkable accuracy, with the number of correct decimal places exponential in the number of terms taken.

\vskip 3mm
\noindent
$\bullet$ {\bf Classification.} The results we have discussed all depend on the properties of the underlying trees and topographs that satisfy the local rule $r+u=st$. For trees with all regions real and at least $2$, we show in Theorem \ref{class2} that they come in only four distinct types: rising, progressing, leveling and constant, with the names reflecting the possible growth of regions. Topographs with regions satisfying \e{restr} similarly break into four types. They all have a vertex, edge or region at the `bottom' (which may be at infinity) from which regions increase. These results are established in Sections \ref{d4} and \ref{rea}. We begin in Sections \ref{mtop} and \ref{asc} with the basic properties of these `Markov'  topographs and trees.

\SpecialCoor
\psset{griddots=5,subgriddiv=0,gridlabels=0pt}
\psset{xunit=5.5cm, yunit=5.5cm, runit=5.5cm}
\psset{linewidth=1pt}
\psset{dotsize=1.4pt,dotstyle=*}
\begin{figure}[ht]
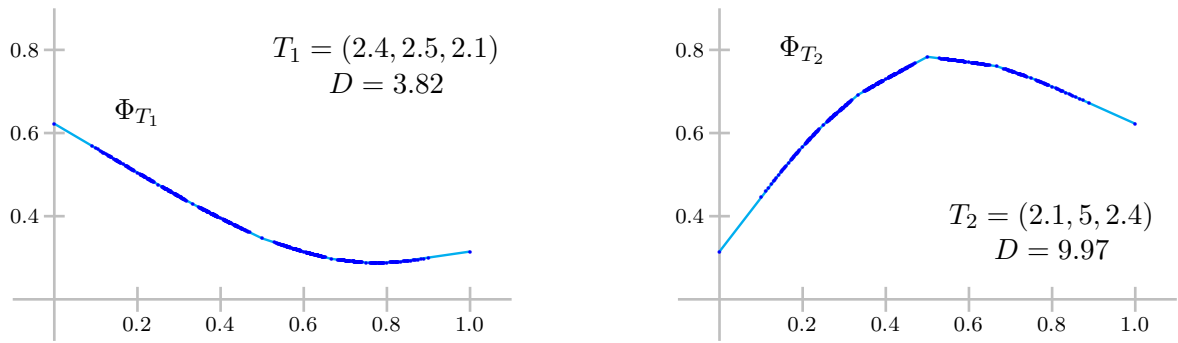

\centering

\psset{arrowscale=1.4,arrowinset=0.3,arrowlength=1.1}
\newrgbcolor{light}{0.8 0.8 1.0}
\newrgbcolor{pale}{1 0.7 1}
\newrgbcolor{pale}{1 0.7 0.4}

\newrgbcolor{markov}{0.4 0.2 0.8}
\psset{linecolor=markov}

}

\end{pspicture}
\caption{Examples of encoding functions $\Phi$  that are convex and concave. The trees $T_1$, $T_2$ are strictly rising.}
\label{phplots}
\end{figure}

\vskip 3mm

A coherent picture emerges in this paper for the properties of numbers on these trees.
The restriction \e{restr} for them is a natural separating line, excluding more complicated behavior. For future work it would be interesting to understand the wilder regime beyond this restriction. We also note that in \cite{sv17} a Lyapunov exponent is defined, closely linked to $\Phi$ for the original Markov tree. This exponent is generalized in \cite{bhi} to incorporate properties of weakly holomorphic modular functions such as the $j$ function. The results here should be relevant for those studies.

\section{Preliminaries}

In Sections \ref{mtop} and \ref{asc} we set up the basic objects of study and introduce some descriptive terminology. 

\subsection{Markov topographs} \la{mtop}
\begin{adef} \la{tdef}
{\rm A $3$-regular tree is drawn in the plane and every complementary region  is given a label from a ring. This forms a {\em Markov topograph} $\mathcal G$ if the labels of the four regions surrounding every edge always satisfy the  local rule $r+u = st$, for the configuration in the top right of Figure \ref{mart}. 
}
\end{adef}

These $3$-regular trees are pictured in Figure \ref{topb}. We are concerned with real Markov topographs where the ring is $\R$. Bowditch in \cite{bow96} used these `trees of Markov triples' 
 to give an elegant proof of McShane's identity, see Section \ref{tops}. Changing the local rule gives other kinds of topographs as discussed in   \cite{nor}. Conway in \cite{con97} used the
same trivalent tree structure to describe values of binary quadratic forms, needing the 
 rule $r+u=2(s+t)$. We follow Conway in calling these types of trees topographs. 

For convenience we usually refer to regions using their labels, even though these might not be distinct. 
If three regions  $r, s, t$ meet at a vertex and are listed in clockwise order, we call them a {\em triple} $(r,s,t)$ of the topograph. Since the rest of the topograph is determined from one triple, we say two topographs are the same if they have a triple in common. It is straightforward to verify with the local rule that the components of
each triple $(r,s,t)$ of a given topograph $\mathcal G$ satisfy
\begin{equation}\label{suwx}
  r^2+s^2+t^2-rst =D
\end{equation}
for a fixed $D=D_{\mathcal G}$ that we call the {\em discriminator} of $\mathcal G$.  Hurwitz   used the same notation $D$ in \cite[Eq. (26)]{hu07}, noting its parallels to a discriminant.
(See \cite[Eq. (3.4)]{OStop} where the corresponding invariant for a quadratic form topograph is $D  = r^2+s^2+t^2 -2rs-2rt-2st$,  the discriminant of the  equivalence class of forms it represents.)

The edge shown  in the middle  of Figure \ref{root}, connecting the regions $r$ and $u$, may be given the label $u-r$ when directed from $r$ to $u$. With the opposite direction its label is $r-u$. So directed edges in a real topograph may be positive, negative or zero.
 We often specify the {\em positive edge direction} of a real edge. This is the direction that makes its label positive, while $0$-labeled edges are left undirected.

\SpecialCoor
\psset{griddots=5,subgriddiv=0,gridlabels=0pt}
\psset{xunit=0.4cm, yunit=0.4cm, runit=0.4cm}
\psset{linewidth=1pt}
\psset{dotsize=7pt 0,dotstyle=*}
\begin{figure}[!htb]
\centering
\begin{pspicture}(0,-1.1)(21,8) 

\psset{arrowscale=1.8,arrowinset=0.3,arrowlength=1.1}
\newrgbcolor{light}{0.7 0.7 1.0}
\newrgbcolor{blue2}{0.3 0.3 0.8}

\rput(4,3.5){%
        \begin{pspicture}(1,0)(7,7)

\psline[linecolor=light](1,0)(4,2)
\psline[linecolor=light](7,0)(4,2)
\psline[linecolor=light]{->}(4,5)(7,7)
\psline[linecolor=light]{->}(4,5)(1,7)

\psline{->}(4,2)(4,5)

\rput(1,3.5){$a$}
\rput(3.3,3.5){\textcolor{red}{$b$}}
\rput(7,3.5){$c$}

\rput(4,-0.5){$(ac-b)/2$}
\rput(4,7.5){$(ac+b)/2$}

\rput(1.8,5.5){\textcolor{red}{$b_L$}}
\rput(6.2,5.5){\textcolor{red}{$b_R$}}

\end{pspicture}}

\rput(22,3.5){%
\begin{pspicture}(0,-4)(10,4)  
$\begin{aligned}
   b_L & = a(ac+b)/2-2c, \phantom{\displaystyle \frac 12}\\
    b_R & = c(ac+b)/2-2a
\end{aligned}$
\end{pspicture}}

\end{pspicture}
\caption{The configuration $[a, b, c]$ in a Markov topograph}
\label{yy}
\end{figure}


Let $[a, b, c]$ denote the configuration shown in Figure \ref{yy} having a directed edge with label $b$ adjacent to regions $a$ and $c$. The basic result on how numbers in a real Markov topograph grow is given in the following version of Conway's quadratic form climbing lemma.

\begin{lemma} [Climbing lemma] \la{climb}
If a Markov topograph contains $[a,b,c]$, 
\begin{equation*}
  (i) \quad \text{with $a, c \gqs 2$ and $b>0$, \quad  \qquad or} \quad \qquad (ii) \quad \text{with $a, c >2$ and $b\gqs 0$},
\end{equation*}
 then all region labels must grow as we continue moving forward and away from $[a,b,c]$. Precisely, for any path moving forward, the sequence of regions encountered on the left of the path and the sequence of those on the right are both strictly increasing. Also, all the edges of these paths are positive when directed forwards.
\end{lemma}
\begin{proof}
Refer to Figure \ref{yy} and let $r=(ac+b)/2$. Assuming the conditions for part (i) we find $r>a$, $r>c$, $b_L\gqs b$ and $b_R\gqs b$. So the region labels strictly increased going forward one edge to the left or right. The new configurations $[a,b_L,r]$ and $[r,b_R,c]$ satisfy the part (i) conditions, so the argument repeats going forward. Similarly, the part (ii) conditions imply $r>a$, $r>c$, $b_L> b$ and $b_R> b$.
\end{proof}

\begin{adef}\la{well}
{\rm Let $\mathcal G$ be a Markov topograph with all regions $\gqs 2$.
A {\em vertex well} is a vertex in $\mathcal G$ with all edges positively directed out from it. An {\em edge well} is a $0$-labeled edge  in $\mathcal G$ whose two adjacent regions are $>2$.}
\end{adef}

So  a vertex well occurs in the middle of the triple $(r,s,t)$ if and only if $st>2r, rt>2s$, and $rs>2t$. An edge well adjacent to the regions $s, t$ has the triples $(r,s,t)$, $(r,t,s)$ at each endpoint, with $st=2r$.
By the climbing lemma,  regions must strictly increase moving away from these wells. Hence topographs with all regions $\gqs 2$ can have at most one well, and if they have one then their basic structure is straightforward.

\begin{lemma} \la{wells}
If a Markov topograph with all regions $\gqs 2$ has an edge or vertex well, then its discriminator  satisfies $D<4$.
\end{lemma}
\begin{proof}
 If it has an edge well then there is a triple $(r,s,t)$ with $s,t>2$ and $s t = 2r$. Equation \e{suwx} becomes
\begin{equation} \la{eq}
  D=4-\tfrac 14(s^2-4)(t^2-4),
\end{equation}
implying $D<4$. If it has a vertex well with surrounding triple $(r,s,t)$ then $r=2$ implies $s<t$ and $t<s$. So we may assume  $2<r \lqs s \lqs t$, say, with $t< r s/2$. As a function of $t$, $$D= (t-rs/2)^2  + r^2 + s^2 -r^2s^2/4$$ has its maximum at $t=s$. Hence 
\begin{align*}
  D  &\lqs 2s^2+r^2-rs^2 \\
  & = 4+(r-2)(r+2-s^2)\\
  & \lqs 4+(r-2)(s+2-s^2)\\
  & = 4 - (r-2)(s-2)(s+1),
\end{align*}
making $D<4$ in this case also.
\end{proof}

Changing the signs of two of the labels  $r, s, t$  gives another solution to \e{suwx}. This extends to a new topograph that is closely related to the original containing $(r,s,t)$, with the signs of certain regions changed. This sign pattern can be described using the topograph $3$-coloring in  
Figure \ref{topv}, with adjacent regions having different colors and three different colors meeting at each vertex. 
\SpecialCoor
\psset{griddots=5,subgriddiv=0,gridlabels=0pt}
\psset{xunit=0.2cm, yunit=0.2cm, runit=0.2cm}
\psset{linewidth=1pt}
\psset{dotsize=7pt 0,dotstyle=*}
\begin{figure}[htb]
\centering

\psset{arrowscale=1.4,arrowinset=0.3,arrowlength=1.1}
\newrgbcolor{light}{0.8 0.8 1.0}
\newrgbcolor{pale}{1 0.7 1}
\newrgbcolor{pale}{1 0.7 0.4}

\newrgbcolor{cta}{0.6 0.8 1}
\newrgbcolor{ctb}{0.5 0.2 0.9}
\newrgbcolor{ctc}{0.0 0.0 0.6}

\newrgbcolor{cta}{1 0.8 0.6}
\newrgbcolor{cta}{0.6 0.8 1}
\newrgbcolor{ctb}{0.9 0.2 0.5}
\newrgbcolor{ctc}{0.3 0.3 0.9}
\psset{linecolor=black}

\newrgbcolor{cta}{0.9 0.9 0.5}

\begin{pspicture}(-15,-12)(15,13.8) 

\rput(25,0){%
\begin{pspicture}(-6,-6)(6,6)

\pspolygon[linecolor=cta,fillstyle=solid,fillcolor=cta](-5,-5)(5,-5)(5,5)(-5,5)(-5,-5)

\pscustom[fillstyle=solid,fillcolor=ctc]{
\psline(0,-2.25)(0,2.25)
\psecurve(-0.5,1.5)(0,2.25)(1.5,3)(3,0)(1.5,-3)(0,-2.25)(-0.5,-1.5)
}

\pscustom[fillstyle=solid,fillcolor=ctb]{
\psline(0,-2.25)(0,2.25)
\psecurve(0.5,1.5)(0,2.25)(-1.5,3)(-3,0)(-1.5,-3)(0,-2.25)(0.5,-1.5)
}

\end{pspicture}}

\pswedge*[linecolor=ctb](0,0){10.95}{90.}{210.}

\pswedge*[linecolor=ctc](0,0){10.95}{-30.}{90.}

\pswedge*[linecolor=cta](0,3){9.66}{52.5}{127.5}

\pswedge*[linecolor=ctb](1.82628,5.38006){7.35}{22.5}{82.5}

\pswedge*[linecolor=ctc](-1.82628,5.38006){7.35}{97.5}{157.5}

\pswedge*[linecolor=cta](4.59792,6.52811){4.59}{2.5}{42.5}

\pswedge*[linecolor=ctc](2.21786,8.35439){4.59}{62.5}{102.5}

\pswedge*[linecolor=ctb](-2.21786,8.35439){4.59}{77.5}{117.5}

\pswedge*[linecolor=cta](-4.59792,6.52811){4.59}{137.5}{177.5}

\pswedge*[linecolor=ctb](7.59507,6.65897){1.7}{-22.5}{27.5}

\pswedge*[linecolor=ctc](6.80975,8.55488){1.7}{17.5}{67.5}

\pswedge*[linecolor=cta](3.60311,11.0154){1.7}{37.5}{87.5}

\pswedge*[linecolor=ctb](1.56854,11.2833){1.7}{77.5}{127.5}

\pswedge*[linecolor=ctc](-1.56854,11.2833){1.7}{52.5}{102.5}

\pswedge*[linecolor=cta](-3.60311,11.0154){1.7}{92.5}{142.5}

\pswedge*[linecolor=ctb](-6.80975,8.55488){1.7}{112.5}{162.5}

\pswedge*[linecolor=ctc](-7.59507,6.65897){1.7}{152.5}{202.5}


\pswedge*[linecolor=cta](0,0){10.95}{-150.}{-30.}

\pswedge*[linecolor=ctb](2.59808,-1.5){9.66}{-67.5}{7.5}

\pswedge*[linecolor=ctc](3.74613,-4.27164){7.35}{-97.5}{-37.5}

\pswedge*[linecolor=cta](5.57241,-1.10842){7.35}{-22.5}{37.5}

\pswedge*[linecolor=ctb](3.35455,-7.24597){4.59}{-117.5}{-77.5}

\pswedge*[linecolor=cta](6.12619,-6.09792){4.59}{-57.5}{-17.5}

\pswedge*[linecolor=ctc](8.34405,-2.25647){4.59}{-42.5}{-2.5}

\pswedge*[linecolor=ctb](7.95247,0.717863){4.59}{17.5}{57.5}

\pswedge*[linecolor=ctc](1.9693,-9.90701){1.7}{-142.5}{-92.5}

\pswedge*[linecolor=cta](4.00387,-10.1749){1.7}{-102.5}{-52.5}

\pswedge*[linecolor=ctb](7.73809,-8.6281){1.7}{-82.5}{-32.5}

\pswedge*[linecolor=ctc](8.98734,-7.00004){1.7}{-42.5}{7.5}

\pswedge*[linecolor=cta](10.5559,-4.28324){1.7}{-67.5}{-17.5}

\pswedge*[linecolor=ctb](11.3412,-2.38733){1.7}{-27.5}{22.5}

\pswedge*[linecolor=ctc](10.8136,1.61998){1.7}{-7.5}{42.5}

\pswedge*[linecolor=cta](9.56437,3.24804){1.7}{32.5}{82.5}


\pswedge*[linecolor=ctc](-2.59808,-1.5){9.66}{172.5}{247.5}

\pswedge*[linecolor=cta](-5.57241,-1.10842){7.35}{142.5}{202.5}

\pswedge*[linecolor=ctb](-3.74613,-4.27164){7.35}{217.5}{277.5}

\pswedge*[linecolor=ctc](-7.95247,0.717863){4.59}{122.5}{162.5}

\pswedge*[linecolor=ctb](-8.34405,-2.25647){4.59}{182.5}{222.5}

\pswedge*[linecolor=cta](-6.12619,-6.09792){4.59}{197.5}{237.5}

\pswedge*[linecolor=ctc](-3.35455,-7.24597){4.59}{257.5}{297.5}

\pswedge*[linecolor=cta](-9.56437,3.24804){1.7}{97.5}{147.5}

\pswedge*[linecolor=ctb](-10.8136,1.61998){1.7}{137.5}{187.5}

\pswedge*[linecolor=ctc](-11.3412,-2.38733){1.7}{157.5}{207.5}

\pswedge*[linecolor=cta](-10.5559,-4.28324){1.7}{197.5}{247.5}

\pswedge*[linecolor=ctb](-8.98734,-7.00004){1.7}{172.5}{222.5}

\pswedge*[linecolor=ctc](-7.73809,-8.6281){1.7}{212.5}{262.5}

\pswedge*[linecolor=cta](-4.00387,-10.1749){1.7}{232.5}{282.5}

\pswedge*[linecolor=ctb](-1.9693,-9.90701){1.7}{272.5}{322.5}


\psline(0,0)(0,3)

\psline(0,3)(1.82628,5.38006)(4.59792,6.52811)(7.59507,6.65897)(9.16566,6.00841)

\psline(7.59507,6.65897)(9.10299,7.44394)(9.10299,7.44394)(9.10299,7.44394)(9.10299,7.44394)

\psline(4.59792,6.52811)(6.80975,8.55488)(8.43107,9.06608)(8.43107,9.06608)(8.43107,9.06608)

\psline(6.80975,8.55488)(7.46032,10.1255)(7.46032,10.1255)(7.46032,10.1255)(7.46032,10.1255)

\psline(1.82628,5.38006)(2.21786,8.35439)(3.60311,11.0154)(4.95181,12.0503)(4.95181,12.0503)

\psline(3.60311,11.0154)(3.67726,12.7138)(3.67726,12.7138)(3.67726,12.7138)(3.67726,12.7138)

\psline(2.21786,8.35439)(1.56854,11.2833)(1.93649,12.943)(1.93649,12.943)(1.93649,12.943)

\psline(1.56854,11.2833)(0.53365,12.632)(0.53365,12.632)(0.53365,12.632)(0.53365,12.632)

\psline(0,3)(-1.82628,5.38006)(-2.21786,8.35439)(-1.56854,11.2833)(-0.53365,12.632)

\psline(-1.56854,11.2833)(-1.93649,12.943)(-1.93649,12.943)(-1.93649,12.943)(-1.93649,12.943)

\psline(-2.21786,8.35439)(-3.60311,11.0154)(-3.67726,12.7138)(-3.67726,12.7138)(-3.67726,12.7138)

\psline(-3.60311,11.0154)(-4.95181,12.0503)(-4.95181,12.0503)(-4.95181,12.0503)(-4.95181,12.0503)

\psline(-1.82628,5.38006)(-4.59792,6.52811)(-6.80975,8.55488)(-7.46032,10.1255)(-7.46032,10.1255)

\psline(-6.80975,8.55488)(-8.43107,9.06608)(-8.43107,9.06608)(-8.43107,9.06608)(-8.43107,9.06608)

\psline(-4.59792,6.52811)(-7.59507,6.65897)(-9.10299,7.44394)(-9.10299,7.44394)(-9.10299,7.44394)

\psline(-7.59507,6.65897)(-9.16566,6.00841)(-9.16566,6.00841)(-9.16566,6.00841)(-9.16566,6.00841)


\psline(0,0)(2.59808,-1.5)

\psline(2.59808,-1.5)(3.74613,-4.27164)(3.35455,-7.24597)(1.9693,-9.90701)(0.620601,-10.9419)

\psline(1.9693,-9.90701)(1.89515,-11.6054)(1.89515,-11.6054)(1.89515,-11.6054)(1.89515,-11.6054)

\psline(3.35455,-7.24597)(4.00387,-10.1749)(3.63592,-11.8346)(3.63592,-11.8346)(3.63592,-11.8346)

\psline(4.00387,-10.1749)(5.03876,-11.5236)(5.03876,-11.5236)(5.03876,-11.5236)(5.03876,-11.5236)

\psline(3.74613,-4.27164)(6.12619,-6.09792)(7.73809,-8.6281)(7.95998,-10.3136)(7.95998,-10.3136)

\psline(7.73809,-8.6281)(9.17185,-9.54151)(9.17185,-9.54151)(9.17185,-9.54151)(9.17185,-9.54151)

\psline(6.12619,-6.09792)(8.98734,-7.00004)(10.2407,-8.14854)(10.2407,-8.14854)(10.2407,-8.14854)

\psline(8.98734,-7.00004)(10.6728,-6.77815)(10.6728,-6.77815)(10.6728,-6.77815)(10.6728,-6.77815)

\psline(2.59808,-1.5)(5.57241,-1.10842)(8.34405,-2.25647)(10.5559,-4.28324)(11.2064,-5.85384)

\psline(10.5559,-4.28324)(12.1772,-4.79444)(12.1772,-4.79444)(12.1772,-4.79444)(12.1772,-4.79444)

\psline(8.34405,-2.25647)(11.3412,-2.38733)(12.8491,-3.1723)(12.8491,-3.1723)(12.8491,-3.1723)

\psline(11.3412,-2.38733)(12.9118,-1.73677)(12.9118,-1.73677)(12.9118,-1.73677)(12.9118,-1.73677)

\psline(5.57241,-1.10842)(7.95247,0.717863)(10.8136,1.61998)(12.4991,1.39809)(12.4991,1.39809)

\psline(10.8136,1.61998)(12.067,2.76848)(12.067,2.76848)(12.067,2.76848)(12.067,2.76848)

\psline(7.95247,0.717863)(9.56437,3.24804)(10.9981,4.16145)(10.9981,4.16145)(10.9981,4.16145)

\psline(9.56437,3.24804)(9.78626,4.93349)(9.78626,4.93349)(9.78626,4.93349)(9.78626,4.93349)


\psline(0,0)(-2.59808,-1.5)

\psline(-2.59808,-1.5)(-5.57241,-1.10842)(-7.95247,0.717863)(-9.56437,3.24804)(-9.78626,4.93349)

\psline(-9.56437,3.24804)(-10.9981,4.16145)(-10.9981,4.16145)(-10.9981,4.16145)(-10.9981,4.16145)

\psline(-7.95247,0.717863)(-10.8136,1.61998)(-12.067,2.76848)(-12.067,2.76848)(-12.067,2.76848)

\psline(-10.8136,1.61998)(-12.4991,1.39809)(-12.4991,1.39809)(-12.4991,1.39809)(-12.4991,1.39809)

\psline(-5.57241,-1.10842)(-8.34405,-2.25647)(-11.3412,-2.38733)(-12.9118,-1.73677)(-12.9118,-1.73677)

\psline(-11.3412,-2.38733)(-12.8491,-3.1723)(-12.8491,-3.1723)(-12.8491,-3.1723)(-12.8491,-3.1723)

\psline(-8.34405,-2.25647)(-10.5559,-4.28324)(-12.1772,-4.79444)(-12.1772,-4.79444)(-12.1772,-4.79444)

\psline(-10.5559,-4.28324)(-11.2064,-5.85384)(-11.2064,-5.85384)(-11.2064,-5.85384)(-11.2064,-5.85384)

\psline(-2.59808,-1.5)(-3.74613,-4.27164)(-6.12619,-6.09792)(-8.98734,-7.00004)(-10.6728,-6.77815)

\psline(-8.98734,-7.00004)(-10.2407,-8.14854)(-10.2407,-8.14854)(-10.2407,-8.14854)(-10.2407,-8.14854)

\psline(-6.12619,-6.09792)(-7.73809,-8.6281)(-9.17185,-9.54151)(-9.17185,-9.54151)(-9.17185,-9.54151)

\psline(-7.73809,-8.6281)(-7.95998,-10.3136)(-7.95998,-10.3136)(-7.95998,-10.3136)(-7.95998,-10.3136)

\psline(-3.74613,-4.27164)(-3.35455,-7.24597)(-4.00387,-10.1749)(-5.03876,-11.5236)(-5.03876,-11.5236)

\psline(-4.00387,-10.1749)(-3.63592,-11.8346)(-3.63592,-11.8346)(-3.63592,-11.8346)(-3.63592,-11.8346)

\psline(-3.35455,-7.24597)(-1.9693,-9.90701)(-1.89515,-11.6054)(-1.89515,-11.6054)(-1.89515,-11.6054)

\psline(-1.9693,-9.90701)(-0.620601,-10.9419)(-0.620601,-10.9419)(-0.620601,-10.9419)(-0.620601,-10.9419)

\end{pspicture}
\caption{Coloring a topograph with three colors.}
\label{topv}
\end{figure}
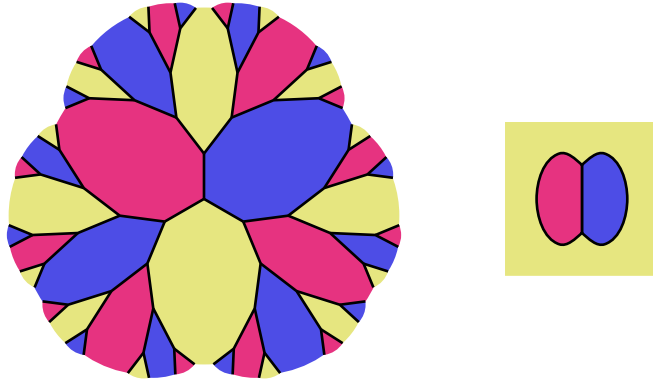
Then a {\em sign variant} of a Markov topograph consists of picking two colors and changing the signs of all regions with those colors. In this way, every topograph has  $3$  sign variants, (which may not all be distinct). Moving along the edges in the topograph in Figure \ref{topv}, we see that the adjacent region colors have a periodicity. This is captured by the simple graph on the right, with the colors appearing in the same order as we move along its edges. The graph on the right may also be closed up into a sphere with three stripes to be rolled along the topograph. Hence the topograph is a covering of this colored graph. The same patterns appear in \cite[Sect. 6.1]{nor} when considering quadratic form topographs modulo $2$. 

\begin{adef}\la{hei} \cite[p. 3]{hin}
{\rm Every topograph region  may be given a {\em height} from a fixed vertex $v$ as follows. The three regions meeting at $v$ have height $1$ and inductively, for an edge $e$ directed away from $v$,  the region pointed to by $e$  has height equal to the sum of the heights of the two  regions adjacent to $e$. For a region $m$ we may denote its height from $v$ by $h_v(m)$ or $h_{r,s,t}(m)$ if the regions $r,s,t$ meet at $v$.}
\end{adef}

\begin{lemma} \la{height}
Let $h_v$ and $h_w$ be the heights in a topograph from vertices $v$ and $w$. Then there exist $\alpha, \beta>0$ with $\alpha \beta=1$ so that, for all regions $m$ in the topograph,
\begin{equation} \la{ec}
  \alpha \cdot h_v(m) \lqs h_w(m) \lqs \beta \cdot h_v(m).
\end{equation}
\end{lemma}
\begin{proof}
Let $X$ be the set of regions adjacent to the simple path $P$ between $v$ and $w$, including the three regions surrounding each of  $v$ and $w$. Choose  $\alpha, \beta>0$ so that  \e{ec} is true for all $m \in X$. By symmetry the optimal  $\alpha, \beta$ will be reciprocals. An induction argument completes the proof, since all edges not on $P$ may be directed away from both $v$ and $w$.
\end{proof}


\subsection{Markov trees} \la{asc}


\SpecialCoor
\psset{griddots=5,subgriddiv=0,gridlabels=0pt}
\psset{xunit=0.28cm, yunit=0.28cm, runit=0.28cm}
\psset{linewidth=1pt}
\psset{dotsize=4pt 0,dotstyle=*}
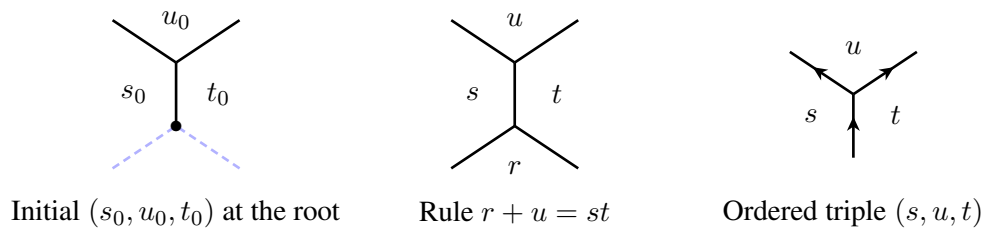
\begin{figure}[htb]
\centering
\begin{pspicture}(-8,-1.9)(35,8.5) 

\psset{arrowscale=1.4,arrowinset=0.3,arrowlength=1.1}
\newrgbcolor{light}{0.7 0.7 1.0}
\newrgbcolor{blue2}{0.3 0.3 0.8}

\rput(14,4.5){%
        \begin{pspicture}(1,-1)(7,8)

\psline(1,0)(4,2)
\psline(7,0)(4,2)
\psline(4,5)(7,7)
\psline(4,5)(1,7)

\psline(4,2)(4,5)

\rput(2,3.5){$s$}

\rput(6,3.5){$t$}

\rput(4,7){$u$}
\rput(4,0){$r$}

\rput(4,-2.1){Rule $r+u=st$}

\end{pspicture}}

\rput(30,4.5){%
        \begin{pspicture}(1,2)(7,8)

\psline[ArrowInside=->,ArrowInsidePos=0.6](4,5)(7,7)
\psline[ArrowInside=->,ArrowInsidePos=0.6](4,5)(1,7)

\psline[ArrowInside=->,ArrowInsidePos=0.6](4,2)(4,5)

\rput(6,4){$t$}
\rput(2,4){$s$}

\rput(4,7.2){$u$}

\rput(4,-0.6){Ordered triple $(s,u,t)$}

\end{pspicture}}

\rput(-2,4.5){%
        \begin{pspicture}(1,-1)(7,8)

\psline[linecolor=light,linestyle=dashed,dash=3pt 2pt](1,0)(4,2)(7,0)
\psline(4,5)(7,7)
\psline(4,5)(1,7)

\psline(4,2)(4,5)

\psdot(4,2)

\rput(2,3.5){$s_0$}
\rput(6,3.5){$t_0$}
\rput(4,7){$u_0$}

\rput(4,-2.1){Initial $(s_0,u_0,t_0)$ at the root}

\end{pspicture}}

\end{pspicture}
\caption{Region labels on a Markov tree.}
\label{root}
\end{figure}

\begin{adef}\la{mat}
{\rm A {\em Markov tree} is a rooted trivalent tree in the plane with all of its complementary regions given labels from a ring. The root vertex has degree $1$. Near the root, as on the left of Figure \ref{root}, is the {\em base triple} $(s_0,u_0,t_0)$.
Four regions surrounding an edge, as shown in the middle  of Figure \ref{root}, must satisfy the local rule $r+u=st$. }
\end{adef}

\SpecialCoor
\psset{griddots=5,subgriddiv=0,gridlabels=0pt}
\psset{xunit=0.23cm, yunit=0.23cm, runit=0.23cm}
\psset{linewidth=1pt}
\psset{dotsize=4pt 0,dotstyle=*}
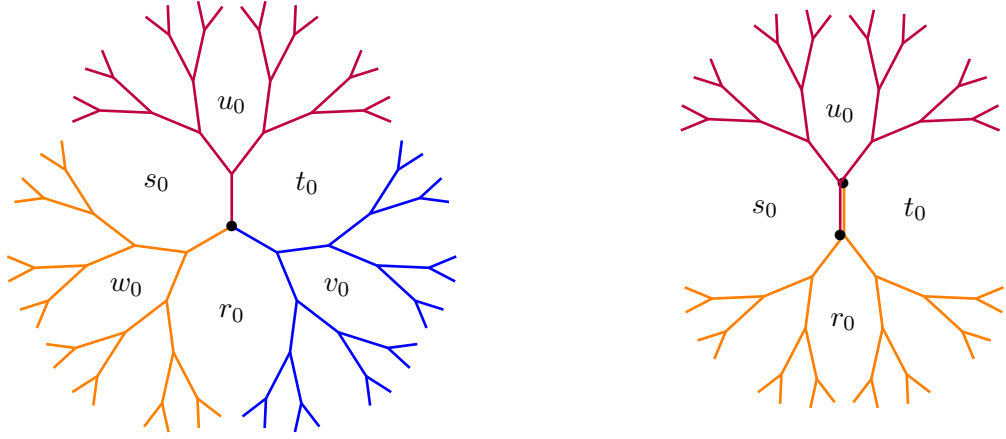
\begin{figure}[htb]
\centering

\psset{arrowscale=1.4,arrowinset=0.3,arrowlength=1.1}
\newrgbcolor{light}{0.8 0.8 1.0}
\newrgbcolor{darkbrown}{0.5 0.324219 0.257813}
\newrgbcolor{pale}{1 0.7 0.4}

\begin{pspicture}(-30,-12)(30,13)

\rput(-15,1){
\begin{pspicture}(-15,-12)(15,14) 

\psset{linecolor=purple}

\psline(0,0)(0,3)

\psline(0,3)(1.82628,5.38006)(4.59792,6.52811)(7.59507,6.65897)(9.16566,6.00841)

\psline(7.59507,6.65897)(9.10299,7.44394)(9.10299,7.44394)(9.10299,7.44394)(9.10299,7.44394)

\psline(4.59792,6.52811)(6.80975,8.55488)(8.43107,9.06608)(8.43107,9.06608)(8.43107,9.06608)

\psline(6.80975,8.55488)(7.46032,10.1255)(7.46032,10.1255)(7.46032,10.1255)(7.46032,10.1255)

\psline(1.82628,5.38006)(2.21786,8.35439)(3.60311,11.0154)(4.95181,12.0503)(4.95181,12.0503)

\psline(3.60311,11.0154)(3.67726,12.7138)(3.67726,12.7138)(3.67726,12.7138)(3.67726,12.7138)

\psline(2.21786,8.35439)(1.56854,11.2833)(1.93649,12.943)(1.93649,12.943)(1.93649,12.943)

\psline(1.56854,11.2833)(0.53365,12.632)(0.53365,12.632)(0.53365,12.632)(0.53365,12.632)

\psline(0,3)(-1.82628,5.38006)(-2.21786,8.35439)(-1.56854,11.2833)(-0.53365,12.632)

\psline(-1.56854,11.2833)(-1.93649,12.943)(-1.93649,12.943)(-1.93649,12.943)(-1.93649,12.943)

\psline(-2.21786,8.35439)(-3.60311,11.0154)(-3.67726,12.7138)(-3.67726,12.7138)(-3.67726,12.7138)

\psline(-3.60311,11.0154)(-4.95181,12.0503)(-4.95181,12.0503)(-4.95181,12.0503)(-4.95181,12.0503)

\psline(-1.82628,5.38006)(-4.59792,6.52811)(-6.80975,8.55488)(-7.46032,10.1255)(-7.46032,10.1255)

\psline(-6.80975,8.55488)(-8.43107,9.06608)(-8.43107,9.06608)(-8.43107,9.06608)(-8.43107,9.06608)

\psline(-4.59792,6.52811)(-7.59507,6.65897)(-9.10299,7.44394)(-9.10299,7.44394)(-9.10299,7.44394)

\psline(-7.59507,6.65897)(-9.16566,6.00841)(-9.16566,6.00841)(-9.16566,6.00841)(-9.16566,6.00841)


\psset{linecolor=blue}

\psline(0,0)(2.59808,-1.5)

\psline(2.59808,-1.5)(3.74613,-4.27164)(3.35455,-7.24597)(1.9693,-9.90701)(0.620601,-10.9419)

\psline(1.9693,-9.90701)(1.89515,-11.6054)(1.89515,-11.6054)(1.89515,-11.6054)(1.89515,-11.6054)

\psline(3.35455,-7.24597)(4.00387,-10.1749)(3.63592,-11.8346)(3.63592,-11.8346)(3.63592,-11.8346)

\psline(4.00387,-10.1749)(5.03876,-11.5236)(5.03876,-11.5236)(5.03876,-11.5236)(5.03876,-11.5236)

\psline(3.74613,-4.27164)(6.12619,-6.09792)(7.73809,-8.6281)(7.95998,-10.3136)(7.95998,-10.3136)

\psline(7.73809,-8.6281)(9.17185,-9.54151)(9.17185,-9.54151)(9.17185,-9.54151)(9.17185,-9.54151)

\psline(6.12619,-6.09792)(8.98734,-7.00004)(10.2407,-8.14854)(10.2407,-8.14854)(10.2407,-8.14854)

\psline(8.98734,-7.00004)(10.6728,-6.77815)(10.6728,-6.77815)(10.6728,-6.77815)(10.6728,-6.77815)

\psline(2.59808,-1.5)(5.57241,-1.10842)(8.34405,-2.25647)(10.5559,-4.28324)(11.2064,-5.85384)

\psline(10.5559,-4.28324)(12.1772,-4.79444)(12.1772,-4.79444)(12.1772,-4.79444)(12.1772,-4.79444)

\psline(8.34405,-2.25647)(11.3412,-2.38733)(12.8491,-3.1723)(12.8491,-3.1723)(12.8491,-3.1723)

\psline(11.3412,-2.38733)(12.9118,-1.73677)(12.9118,-1.73677)(12.9118,-1.73677)(12.9118,-1.73677)

\psline(5.57241,-1.10842)(7.95247,0.717863)(10.8136,1.61998)(12.4991,1.39809)(12.4991,1.39809)

\psline(10.8136,1.61998)(12.067,2.76848)(12.067,2.76848)(12.067,2.76848)(12.067,2.76848)

\psline(7.95247,0.717863)(9.56437,3.24804)(10.9981,4.16145)(10.9981,4.16145)(10.9981,4.16145)

\psline(9.56437,3.24804)(9.78626,4.93349)(9.78626,4.93349)(9.78626,4.93349)(9.78626,4.93349)


\psset{linecolor=orange}

\psline(0,0)(-2.59808,-1.5)

\psline(-2.59808,-1.5)(-5.57241,-1.10842)(-7.95247,0.717863)(-9.56437,3.24804)(-9.78626,4.93349)

\psline(-9.56437,3.24804)(-10.9981,4.16145)(-10.9981,4.16145)(-10.9981,4.16145)(-10.9981,4.16145)

\psline(-7.95247,0.717863)(-10.8136,1.61998)(-12.067,2.76848)(-12.067,2.76848)(-12.067,2.76848)

\psline(-10.8136,1.61998)(-12.4991,1.39809)(-12.4991,1.39809)(-12.4991,1.39809)(-12.4991,1.39809)

\psline(-5.57241,-1.10842)(-8.34405,-2.25647)(-11.3412,-2.38733)(-12.9118,-1.73677)(-12.9118,-1.73677)

\psline(-11.3412,-2.38733)(-12.8491,-3.1723)(-12.8491,-3.1723)(-12.8491,-3.1723)(-12.8491,-3.1723)

\psline(-8.34405,-2.25647)(-10.5559,-4.28324)(-12.1772,-4.79444)(-12.1772,-4.79444)(-12.1772,-4.79444)

\psline(-10.5559,-4.28324)(-11.2064,-5.85384)(-11.2064,-5.85384)(-11.2064,-5.85384)(-11.2064,-5.85384)

\psline(-2.59808,-1.5)(-3.74613,-4.27164)(-6.12619,-6.09792)(-8.98734,-7.00004)(-10.6728,-6.77815)

\psline(-8.98734,-7.00004)(-10.2407,-8.14854)(-10.2407,-8.14854)(-10.2407,-8.14854)(-10.2407,-8.14854)

\psline(-6.12619,-6.09792)(-7.73809,-8.6281)(-9.17185,-9.54151)(-9.17185,-9.54151)(-9.17185,-9.54151)

\psline(-7.73809,-8.6281)(-7.95998,-10.3136)(-7.95998,-10.3136)(-7.95998,-10.3136)(-7.95998,-10.3136)

\psline(-3.74613,-4.27164)(-3.35455,-7.24597)(-4.00387,-10.1749)(-5.03876,-11.5236)(-5.03876,-11.5236)

\psline(-4.00387,-10.1749)(-3.63592,-11.8346)(-3.63592,-11.8346)(-3.63592,-11.8346)(-3.63592,-11.8346)

\psline(-3.35455,-7.24597)(-1.9693,-9.90701)(-1.89515,-11.6054)(-1.89515,-11.6054)(-1.89515,-11.6054)

\psline(-1.9693,-9.90701)(-0.620601,-10.9419)(-0.620601,-10.9419)(-0.620601,-10.9419)(-0.620601,-10.9419)


\rput(4.33013,2.5){$t_0$}

\rput(0,7.){$u_0$}


\rput(0,-5.){$r_0$}

\rput(6.06218,-3.5){$v_0$}


\rput(-4.33013,2.5){$s_0$}

\rput(-6.06218,-3.5){$w_0$}

\psset{linecolor=black}
\psdot(0,0)

\end{pspicture}}

\rput(20,1){
\begin{pspicture}(-10,-11)(10,14) 


\psset{linecolor=orange}

\rput(0,-4){
\begin{pspicture}(-10,0)(10,-14)
\psline(0,0)(0,-3)

\psline(0,-3)(1.82628,-5.38006)(4.59792,-6.52811)(7.59507,-6.65897)(9.16566,-6.00841)

\psline(7.59507,-6.65897)(9.10299,-7.44394)(9.10299,-7.44394)(9.10299,-7.44394)(9.10299,-7.44394)

\psline(4.59792,-6.52811)(6.80975,-8.55488)(8.43107,-9.06608)(8.43107,-9.06608)(8.43107,-9.06608)

\psline(6.80975,-8.55488)(7.46032,-10.1255)(7.46032,-10.1255)(7.46032,-10.1255)(7.46032,-10.1255)

\psline(1.82628,-5.38006)(2.21786,-8.35439)(3.60311,-11.0154)(4.95181,-12.0503)(4.95181,-12.0503)

\psline(3.60311,-11.0154)(3.67726,-12.7138)(3.67726,-12.7138)(3.67726,-12.7138)(3.67726,-12.7138)

\psline(2.21786,-8.35439)(1.56854,-11.2833)(1.93649,-12.943)(1.93649,-12.943)(1.93649,-12.943)

\psline(1.56854,-11.2833)(0.53365,-12.632)(0.53365,-12.632)(0.53365,-12.632)(0.53365,-12.632)

\psline(0,-3)(-1.82628,-5.38006)(-2.21786,-8.35439)(-1.56854,-11.2833)(-0.53365,-12.632)

\psline(-1.56854,-11.2833)(-1.93649,-12.943)(-1.93649,-12.943)(-1.93649,-12.943)(-1.93649,-12.943)

\psline(-2.21786,-8.35439)(-3.60311,-11.0154)(-3.67726,-12.7138)(-3.67726,-12.7138)(-3.67726,-12.7138)

\psline(-3.60311,-11.0154)(-4.95181,-12.0503)(-4.95181,-12.0503)(-4.95181,-12.0503)(-4.95181,-12.0503)

\psline(-1.82628,-5.38006)(-4.59792,-6.52811)(-6.80975,-8.55488)(-7.46032,-10.1255)(-7.46032,-10.1255)

\psline(-6.80975,-8.55488)(-8.43107,-9.06608)(-8.43107,-9.06608)(-8.43107,-9.06608)(-8.43107,-9.06608)

\psline(-4.59792,-6.52811)(-7.59507,-6.65897)(-9.10299,-7.44394)(-9.10299,-7.44394)(-9.10299,-7.44394)

\psline(-7.59507,-6.65897)(-9.16566,-6.00841)(-9.16566,-6.00841)(-9.16566,-6.00841)(-9.16566,-6.00841)

\end{pspicture}}

\psdot[linecolor=black](0.17,3)

\psset{linecolor=purple}

\psline(0,0)(0,3)

\psline(0,3)(1.82628,5.38006)(4.59792,6.52811)(7.59507,6.65897)(9.16566,6.00841)

\psline(7.59507,6.65897)(9.10299,7.44394)(9.10299,7.44394)(9.10299,7.44394)(9.10299,7.44394)

\psline(4.59792,6.52811)(6.80975,8.55488)(8.43107,9.06608)(8.43107,9.06608)(8.43107,9.06608)

\psline(6.80975,8.55488)(7.46032,10.1255)(7.46032,10.1255)(7.46032,10.1255)(7.46032,10.1255)

\psline(1.82628,5.38006)(2.21786,8.35439)(3.60311,11.0154)(4.95181,12.0503)(4.95181,12.0503)

\psline(3.60311,11.0154)(3.67726,12.7138)(3.67726,12.7138)(3.67726,12.7138)(3.67726,12.7138)

\psline(2.21786,8.35439)(1.56854,11.2833)(1.93649,12.943)(1.93649,12.943)(1.93649,12.943)

\psline(1.56854,11.2833)(0.53365,12.632)(0.53365,12.632)(0.53365,12.632)(0.53365,12.632)

\psline(0,3)(-1.82628,5.38006)(-2.21786,8.35439)(-1.56854,11.2833)(-0.53365,12.632)

\psline(-1.56854,11.2833)(-1.93649,12.943)(-1.93649,12.943)(-1.93649,12.943)(-1.93649,12.943)

\psline(-2.21786,8.35439)(-3.60311,11.0154)(-3.67726,12.7138)(-3.67726,12.7138)(-3.67726,12.7138)

\psline(-3.60311,11.0154)(-4.95181,12.0503)(-4.95181,12.0503)(-4.95181,12.0503)(-4.95181,12.0503)

\psline(-1.82628,5.38006)(-4.59792,6.52811)(-6.80975,8.55488)(-7.46032,10.1255)(-7.46032,10.1255)

\psline(-6.80975,8.55488)(-8.43107,9.06608)(-8.43107,9.06608)(-8.43107,9.06608)(-8.43107,9.06608)

\psline(-4.59792,6.52811)(-7.59507,6.65897)(-9.10299,7.44394)(-9.10299,7.44394)(-9.10299,7.44394)

\psline(-7.59507,6.65897)(-9.16566,6.00841)(-9.16566,6.00841)(-9.16566,6.00841)(-9.16566,6.00841)


\rput(4.33013,1.5){$t_0$}

\rput(0,7.){$u_0$}


\rput(0.17,-5.){$r_0$}



\rput(-4.33013,1.5){$s_0$}


\psset{linecolor=black}
\psdot(0,0)

\end{pspicture}}

\end{pspicture}

\caption{Markov topographs made up of two or three Markov trees}
\label{topb}
\end{figure}

In other words, a Markov tree is a rooted subtree of a Markov topograph. It is uniquely specified by its base triple. 
Figure \ref{topb} shows how a topograph breaks into two or three Markov trees. A decomposition into four trees is seen in Figure \ref{fatop}.

Label the edges of a Markov tree
as we did for topographs. The initial edge from the root to $u_0$  has label $2u_0-s_0t_0$ to be consistent with the local rule. 
Our main object of study is the following, generalizing the original Markov tree in Figure \ref{mart}.

\begin{adef}\la{eatx} 
{\rm A {\em  rising tree} is a real Markov tree $T$ with all regions $>2$. Furthermore, when given the positive edge direction, all edges are directed away from the root except for those on a finite simple path starting at the root. If this {\em descending path} is empty then $T$ is {\em strictly rising}.}
\end{adef}


The material in Section \ref{mtop} lets us understand the simple structure of these trees. 
If a rising tree $T$ has base triple $(s_0,u_0,t_0)$, then it is strictly rising exactly when $u_0 >s_0 t_0/2$, by the climbing lemma, and regions strictly increase (rise) moving away from the root.   It is an exercise to show that there are uncountably many strictly rising trees for each  discriminator $D \in \R$.

\SpecialCoor
\psset{griddots=5,subgriddiv=0,gridlabels=0pt}
\psset{xunit=0.35cm, yunit=0.35cm, runit=0.5cm}
\psset{linewidth=1pt}
\psset{dotsize=4pt 0,dotstyle=*}
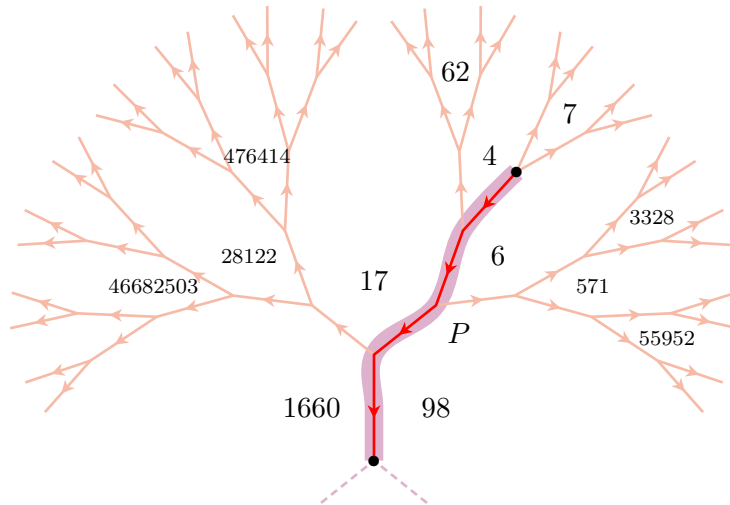
\begin{figure}[ht]
\centering

\psset{arrowscale=1.4,arrowinset=0.3,arrowlength=1.1}
\newrgbcolor{light}{0.8 0.8 1.0}
\newrgbcolor{pale}{1 0.7 1}
\newrgbcolor{pale}{1 0.7 0.4}
\newrgbcolor{pale}{1 0.85 0.65}

\newrgbcolor{lightorange}{0.964844 0.730469 0.65625}

\begin{pspicture}(-15,-1.6)(15,17.5) 

\psset{arrowscale=1.8,arrowinset=0.3,arrowlength=1.1}
\newrgbcolor{light}{0.9 0.7 1.0}
\newrgbcolor{blue2}{0.3 0.3 0.8}
\newrgbcolor{light}{0.7 0.7 1.0}

\newrgbcolor{light}{0.87 0.7 0.8}

\psset{arrowscale=1.4,arrowinset=0.3,arrowlength=1.0}



\psset{linecolor=lightorange} 


\pscurve[linewidth=7pt,linecolor=light](5.38512,10.9094)
(3.36415,8.69223)(2.34549,5.87047)
(0,4)(0,2)(0,0)

\psline[linecolor=light,linestyle=dashed,dash=3pt 2pt](-2,-1.6)(0,0)(2,-1.6)

\psline[ArrowInside=->,ArrowInsidePos=0.6](0,4)(2.34549,5.87047)(5.32318,6.23569)(8.21397,5.4336)(10.7154,3.77746)(12.438,1.82995)

\psline[ArrowInside=->,ArrowInsidePos=0.6](10.7154,3.77746)(13.1809,2.95215)

\psline[ArrowInside=->,ArrowInsidePos=0.6](8.21397,5.4336)(11.2111,5.56406)(13.7494,5.00099)

\psline[ArrowInside=->,ArrowInsidePos=0.6](11.2111,5.56406)(13.6909,6.34558)

\psline[ArrowInside=->,ArrowInsidePos=0.6](5.32318,6.23569)(7.93444,7.71262)(10.8743,8.31034)(13.4694,8.15128)

\psline[ArrowInside=->,ArrowInsidePos=0.6](10.8743,8.31034)(13.2013,9.47015)

\psline[ArrowInside=->,ArrowInsidePos=0.6](7.93444,7.71262)(9.9615,9.92419)(12.1545,11.3209)

\psline[ArrowInside=->,ArrowInsidePos=0.6](9.9615,9.92419)(11.1623,12.2303)

\psline[ArrowInside=->,ArrowInsidePos=0.6](2.34549,5.87047)(3.36415,8.69223)(5.38512,10.9094)(7.9923,12.3935)(10.5078,13.0511)

\psline[ArrowInside=->,ArrowInsidePos=0.6](7.9923,12.3935)(9.84196,14.2207)

\psline[ArrowInside=->,ArrowInsidePos=0.6](5.38512,10.9094)(6.62203,13.6425)(8.27056,15.6531)

\psline[ArrowInside=->,ArrowInsidePos=0.6](6.62203,13.6425)(7.04442,16.208)

\psline[ArrowInside=->,ArrowInsidePos=0.6](3.36415,8.69223)(3.22543,11.689)(4.01956,14.582)(5.33328,16.8257)

\psline[ArrowInside=->,ArrowInsidePos=0.6](4.01956,14.582)(4.03543,17.182)

\psline[ArrowInside=->,ArrowInsidePos=0.6](3.22543,11.689)(2.16744,14.4963)(1.91145,17.0836)

\psline[ArrowInside=->,ArrowInsidePos=0.6](2.16744,14.4963)(0.652066,16.609)

\psline[ArrowInside=->,ArrowInsidePos=0.6](-3.22543,11.689)(-2.16744,14.4963)(-0.652066,16.609)

\psline[ArrowInside=->,ArrowInsidePos=0.6](-2.16744,14.4963)(-1.91145,17.0836)

\psline[ArrowInside=->,ArrowInsidePos=0.6](-4.01956,14.582)(-5.33328,16.8257)

\psline[ArrowInside=->,ArrowInsidePos=0.6](-3.36415,8.69223)(-5.38512,10.9094)(-6.62203,13.6425)(-7.04442,16.208)

\psline[ArrowInside=->,ArrowInsidePos=0.6](-6.62203,13.6425)(-8.27056,15.6531)

\psline[ArrowInside=->,ArrowInsidePos=0.6](-5.38512,10.9094)(-7.9923,12.3935)(-9.84196,14.2207)

\psline[ArrowInside=->,ArrowInsidePos=0.6](-7.9923,12.3935)(-10.5078,13.0511)

\psline[ArrowInside=->,ArrowInsidePos=0.6](-2.34549,5.87047)(-5.32318,6.23569)(-7.93444,7.71262)(-9.9615,9.92419)(-11.1623,12.2303)

\psline[ArrowInside=->,ArrowInsidePos=0.6](-9.9615,9.92419)(-12.1545,11.3209)

\psline[ArrowInside=->,ArrowInsidePos=0.6](-7.93444,7.71262)(-10.8743,8.31034)(-13.2013,9.47015)

\psline[ArrowInside=->,ArrowInsidePos=0.6](-10.8743,8.31034)(-13.4694,8.15128)

\psline[ArrowInside=->,ArrowInsidePos=0.6](-5.32318,6.23569)(-8.21397,5.4336)(-11.2111,5.56406)(-13.6909,6.34558)

\psline[ArrowInside=->,ArrowInsidePos=0.6](-11.2111,5.56406)(-13.7494,5.00099)

\psline[ArrowInside=->,ArrowInsidePos=0.6](-8.21397,5.4336)(-10.7154,3.77746)(-13.1809,2.95215)

\psline[ArrowInside=->,ArrowInsidePos=0.6](-10.7154,3.77746)(-12.438,1.82995)

\psline[ArrowInside=->,ArrowInsidePos=0.6]
(0,4)(-2.34549,5.87047)(-3.36415,8.69223)(-3.22543,11.689)(-4.01956,14.582)(-4.03543,17.182)

\psline[ArrowInside=->,ArrowInsidePos=0.6,linecolor=red](5.38512,10.9094)
(3.36415,8.69223)(2.34549,5.87047)
(0,4)(0,0)

\rput(3.2,4.8){$P$}
\rput(-2.34549,2){$1660$}
\rput(2.34549,2){$98$}


\rput(-8.30087,6.6009){$_{46682503}$}


\rput(-4.38281,11.514){$_{476414}$}
\rput(-4.69099,7.74094){$_{28122}$} 

\rput(0.,6.8){$17$}

\rput(3.0867,14.6858){$62$}
\rput(4.69099,7.74094){$6$}
\rput(4.38281,11.514){$4$}

\rput(7.40607,13.1265){$7$}

\rput(8.30087,6.6009){$_{571}$}
\rput(10.5457,9.18954){$_{3328}$}

\rput(11.1048,4.6315){$_{55952}$}

\psset{linecolor=gray}

\psset{linecolor=black}

\psdot[linecolor=black](0,0)
\psdot[linecolor=black](5.38512,10.9094)

\end{pspicture}
\caption{The rising  tree $(1660,17,98)$ with $D=-67$ and descending path $P$. All edges are shown positively directed (pointing to the larger region at their endpoints).}
\label{rtree}
\end{figure}

Next suppose $T$ has a nonempty descending path $P$. By {\em descending} we will always mean moving along edges against the positive edge direction, so that neighboring region labels decrease. Then $P$ descends from  the root to a vertex $v$ at the end of the path. If there is a $0$-edge on $P$ then it is an edge well (in the topograph $T$ extends to) and so this edge must have endpoint $v$. 
Otherwise, $v$ must be a vertex well. 
From this point of view, rising trees that are not strictly rising correspond to parts of Markov topographs that include an edge or vertex well. The descending path is just the path from the root to the well. By Lemma \ref{wells} these trees have discriminator $D<4$, so we have established:

\begin{lemma} \la{tr4} 
A rising tree with a nonempty descending path must have discriminator $D<4$.
\end{lemma}

Figure \ref{rtree} shows the rising tree $(1660,17,98)$. For the example of a tree with base triple $(s_0,u_0,t_0)$ and $2<s_0, t_0 \lqs u_0$ from the introduction, check that this tree must be rising with at most one edge in its descending path. In general, regions along any path moving away from the root in a rising tree will eventually start strictly increasing. Will will see in Proposition \ref{sim2} that this growth is always exponential, and this property will characterize rising trees.

\begin{adef}\la{pr} 
{\rm A {\em  progressing tree} is a real Markov tree with base triple $(s_0,u_0,t_0)$ where  
\begin{equation*}
  2\lqs s_0<u_0, \qquad 2\lqs t_0<u_0,  \qquad (s_0-2)(t_0-2)=0.
\end{equation*}}
\end{adef}

The next lemma is an easy exercise.

\begin{lemma} \la{pr4} 
A progressing tree has all regions $\gqs 2$, discriminator $D>4$, and all edges positively directed away from the root. The regions bordering a $2$-region increase (progress) in an arithmetic progression.
\end{lemma}

Apart from the constant tree, with all regions equal, we will see there is only one other type of tree in this context. We call it a {\em leveling tree} and it is specific to $D=4$.  This type is introduced and studied in Section \ref{d4}. Theorem \ref{class2} shows that every  
Markov tree  with 
labels $\gqs 2$  is either  rising, progressing, leveling or constant. 

A triple $(s,u,t)$ is {\em oriented} if $u$ is further from the root than $s$ and $t$. A triple $(s,u,t)$ is {\em ordered}  if it is oriented and $u > st/2$. This means that the edge between $s$ and $t$ is positively directed towards $u$.
By the climbing lemma, ordered triples $(s,u,t)$ with $s,t \gqs 2$ appear as on the right of Figure \ref{root}, with the arrows indicating both the positive edge direction and the direction away from the root. 
It may also be seen with the climbing lemma that every oriented triple in a progressing tree is ordered. The same is true for rising trees except that triples with central vertex on the descending path are not ordered -- see Figure \ref{rtree}.

For two adjacent regions $s$ and $t$, as in the middle of Figure \ref{root}, the other two regions $r$ and $u$ may be given by the following formulas in terms of $D$:
\begin{align}
  r, u & = \tfrac 12 \left(st \pm \sqrt{s^2t^2-4(s^2+t^2)+4D}\right)\notag \\
   & = \tfrac 12 \left(st \pm \sqrt{(s^2-4)(t^2-4)+4(D-4)}\right). \label{iv}
\end{align}
So one of $r$ and $u$ is $\gqs st/2$ and one is $\lqs st/2$.


Every Markov tree may be {\em Farey-indexed} by overlaying the Farey tree in Figure \ref{farey}. Each region of the tree corresponds to a unique fraction $q$ in $[0,1]$ and we write its label as $m_q$, beginning with
\begin{equation*}
  m_0 = s_0, \quad m_{1/2}=u_0, \quad m_1=t_0.
\end{equation*}

\subsection{Farey neighbors} \la{rey}
We recall here some facts about Farey fractions, see \cite[Chap. III]{hawr}. The Farey fractions of order $n$, denoted $\mathcal F_n$, are the reduced rationals in the interval $[0,1]$
with denominator at most $n$.
Two fractions $a/b < c/d$ are Farey neighbors if they appear consecutively in some $\mathcal F_n$. This is equivalent to having $bc-ad=1$. The mediant of two Farey neighbors lies between them, and all three are Farey neighbors. Every  fraction in $(0,1)$ is produced uniquely in this way as a mediant, starting with $0/1$ and $1/1$.

\SpecialCoor
\psset{griddots=5,subgriddiv=0,gridlabels=0pt}
\psset{xunit=0.4cm, yunit=0.4cm, runit=0.5cm}
\psset{linewidth=1pt}
\psset{dotsize=4pt 0,dotstyle=*}
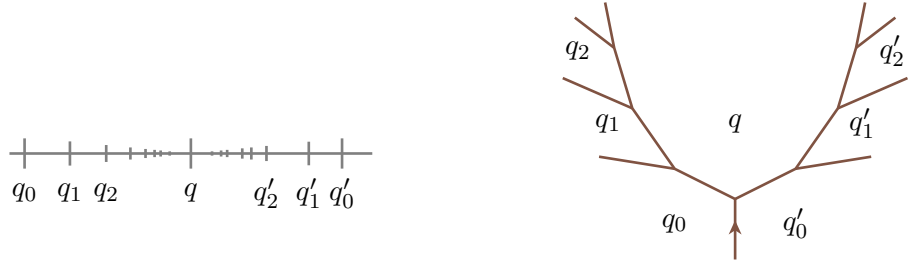
\begin{figure}[!htb]
\centering

\psset{arrowscale=1.4,arrowinset=0.3,arrowlength=1.1}
\newrgbcolor{light}{0.8 0.8 1.0}
\newrgbcolor{pale}{1 0.7 1}
\newrgbcolor{pale}{1 0.7 0.4}
\newrgbcolor{darkbrown}{0.5 0.324219 0.257813}

\begin{pspicture}(-15,0)(15,8.5) 

\rput(9,4.5){%
        \begin{pspicture}(-6,0)(6,9)

\psset{linecolor=darkbrown}

\psline[ArrowInside=->,ArrowInsidePos=0.6](0,0)(0,2)
\psline(0,2)(2,3)(3.4,5)(4,7)(4.3,8.5)
\psline(2,3)(4.5,3.4)
\psline(3.4,5)(5.7,6)
\psline(4,7)(5.3,8)

\rput(0,4.5){$q$}
\rput(2,1.2){$q_0'$}
\rput(4.2,4.5){$q_1'$}
\rput(5.2,6.9){$q_2'$}

\psline(0,2)(-2,3)(-3.4,5)(-4,7)(-4.3,8.5)
\psline(-2,3)(-4.5,3.4)
\psline(-3.4,5)(-5.7,6)
\psline(-4,7)(-5.3,8)

\rput(-2,1.2){$q_0$}
\rput(-4.2,4.5){$q_1$}
\rput(-5.2,6.9){$q_2$}
        \end{pspicture}}

\rput(-9,3){%
        \begin{pspicture}(-6,-2)(6,1)
        
\psset{linecolor=gray}
        
\psline(-6,0)(6,0)
\psline(-5.5,0.5)(-5.5,-0.5)
\psline(-4,0.4)(-4,-0.4)
\psline(-2.8,0.28)(-2.8,-0.28)

\psline(-2,0.2)(-2,-0.2)
\psline(-1.5,0.15)(-1.5,-0.15)
\psline(-1.2,0.12)(-1.2,-0.12)
\psline(-1,0.1)(-1,-0.1)
\psline(-0.7,0.07)(-0.7,-0.07)

\psline(5,0.5)(5,-0.5)
\psline(0,0.5)(0,-0.5)
\psline(3.9,0.39)(3.9,-0.39)
\psline(2.5,0.25)(2.5,-0.25)

\psline(2,0.2)(2,-0.2)
\psline(1.7,0.17)(1.7,-0.17)
\psline(1.2,0.12)(1.2,-0.12)
\psline(1,0.1)(1,-0.1)
\psline(0.7,0.07)(0.7,-0.07)

\rput(0,-1.3){$q$}
\rput(-5.5,-1.3){$q_0$}
\rput(-4,-1.3){$q_1$}
\rput(-2.8,-1.3){$q_2$}

\rput(5,-1.3){$q_0'$}
\rput(3.9,-1.3){$q_1'$}
\rput(2.5,-1.3){$q_2'$}

        \end{pspicture}}

\psset{linecolor=black}

\end{pspicture}
\caption{The Farey neighbors of $q$, on the number line and on the Farey tree.}
\label{fqs}
\end{figure}

Given a reduced $q=k/n$ in the interval $(0,1)$, (fractions are always assumed to be reduced from this point), we will need  its least Farey neighbor $q_0=k_0/n_0$ and its greatest Farey neighbor $q_0'=k_0'/n_0'$. They may be found using the extended Euclidean algorithm, or by noting that they are adjacent to $k/n$ in $\mathcal F_n$; see  \cite[Sect. 4.5]{Knu}. We have the elementary results:

\begin{lemma} \la{qq} For rational $q=k/n$ with $0<q<1$, there exist $k_0, n_0, k_0', n_0'$, found as above, so that the complete set of Farey neighbors of $q$ is given by $q_i$, $q'_j$ with
\begin{equation*}
 0\lqs  q_i=\frac{k_i}{n_i}:=\frac{k_0 + ik}{n_0 +in} < q < \frac{k_0' + jk}{n_0' +jn} =: \frac{k'_j}{n'_j} =q_j' \lqs 1, \qquad \quad (i, j\in \Z_{\gqs 0}).
\end{equation*}
Then $q_i \to q$ and $q_j' \to q$.
Also, if $p<p'$ are Farey neighbors  with mediant $q$, then we must have $p=q_0$ and $p'=q_0'$. 
\end{lemma}

From now on we will use the notation $q_0, q_0'$  for the unique Farey neighbors with mediant $q$,  as shown in Figure \ref{fqs}. 
Note  that $0$  only has Farey neighbors $1/(j+1)$ for $j \gqs 0$, while  $1$ only has  $j/(j+1)$ for $j \gqs 0$. We also record
\begin{equation}\label{nb}
  k n_i - n k_i =1 = n k_j'-k n_j'.
\end{equation}

Farey sequences may be divided into rows, as in \cite[Sect. 3.2]{aig}. Let the $0$th row be $\{0/1, 1/1\}$ and  row $1$ is $\{1/2\}$. Row $2$ is $\{1/3, 2/3\}$ and, in general, row $\ell$ is produced by taking mediants of every pair of consecutive fractions in the union of all lower rows. The fractions in row $\ell$ have denominators at least $\ell+1$ and there are $2^{\ell-1}$ of them for $\ell \gqs 1$. Fractions in row $\ell$ are those appearing  $\ell$ edges away from the root in the Farey tree, shown in Figure \ref{farey}.

\subsection{Tree numbers at Farey neighbors}

For $x,s,u,t \in \R$ with $x \gqs 2$ and $u>2$, set
\begin{alignat}{2}\label{lrh}
  \lambda_x & := (x+\sqrt{x^2-4})/2, &\qquad \rho(x) & :=\log(\lambda_x), \\
  f(s,u,t) &:= \frac{\lambda_u s-t}{\sqrt{u^2-4}}, &\qquad g(s,u,t) & := \frac{t-\lambda_u^{-1} s}{\sqrt{u^2-4}}. \label{fg}
\end{alignat}
 The entire theory rests on the following explicit formulas for tree numbers at Farey neighbors, see \cite[p. 715]{z82}.
We  focus on the real cases we will need, but of course this all works for complex Markov trees as well.

\begin{prop} \la{start}
Let $T$ be a real, Farey-indexed Markov tree  and let $(s,u,t)$ be an oriented  triple on it with $u=m_q>2$.  Denote the Farey neighbors of $q$ by $q_j$, $q_j'$ as in Lemma \ref{qq}. 
Then for $j \gqs 0$,
\begin{align}
  m_{q_j} & = f(s,u,t) \cdot e^{\rho(u) j} + g(s,u,t) \cdot e^{-\rho(u) j}, \la{m1}\\
  m_{q'_j} & =  f(t,u,s) \cdot  e^{\rho(u) j} + g(t,u,s)  \cdot e^{-\rho(u) j}. \la{m2}
\end{align}
\end{prop}
\begin{proof}
By the local rule $m_{q_1'}=ut-s$, $m_{q_2'}=u(ut-s)-t$, and in general
\begin{equation*}
  \begin{pmatrix} u & -1 \\ 1 & 0 \end{pmatrix}^j \begin{pmatrix} t \\ s \end{pmatrix} = \begin{pmatrix} m_{q'_j} \\ m_{q'_{j-1}} \end{pmatrix}.
\end{equation*}
The eigenvalues of this matrix are the roots $\lambda_u$, $\lambda^{-1}_u$  of $x^2-ux+1=0$, and we obtain \e{m2} by diagonalizing. Swapping $s$ and $t$ gives \e{m1}.
\end{proof}

It is similar for $q$ at the endpoints of the interval $[0,1]$:

\begin{prop} \la{stx}
Let the above tree $T$ have base triple $(s_0, u_0, t_0)$.  
For $q=0$, so that $q_j'=1/(j+1)$,
\begin{equation*}
   m_{q'_{j+1}}  =  \begin{cases}
   f(u_0,s_0,t_0) \cdot e^{\rho({s_0}) j} + g(u_0,s_0,t_0) \cdot e^{-\rho({s_0}) j} & \text{if} \quad s_0 > 2,\\
   (u_0-t_0)j +t_0 & \text{if}  \quad s_0 = 2.
   \end{cases}
\end{equation*}
 For $q=1$, so that $q_j=j/(j+1)$, we have 
\begin{equation*}
   m_{q_{j+1}}  = \begin{cases}
   f(u_0,t_0,s_0) \cdot e^{\rho({t_0}) j} + g(u_0,t_0,s_0) \cdot e^{-\rho({t_0}) j} & \text{if} \quad t_0 > 2,\\
   (u_0-s_0)j +s_0 & \text{if}  \quad t_0 = 2.
   \end{cases}
\end{equation*}
\end{prop}

\subsection{Some properties of $f$, $g$ and $\rho$}

Recall \e{lrh}. These useful identities are noted for $x\gqs 2$:
\begin{equation}\label{use}
  \lambda_x + \lambda_x^{-1} = x, \qquad \lambda_x - \lambda_x^{-1} = \sqrt{x^2-4}, \qquad \rho(x) = \cosh^{-1}(x/2).
\end{equation}

\begin{lemma}
Let $s, u, t$ be any real numbers with $2< s, t\lqs u$. For $f$ and $g$ defined in \e{fg},
\begin{alignat}{2}
 s-1 & < f(s,u,t) < s, &\qquad \quad 0  < g(s,u,t) & <1, \la{wa}\\
  u/2 & < f(u,s,t),  &\qquad \quad   g(u,s,t) & <u/2.  \la{wb}
\end{alignat}
\end{lemma}
\begin{proof}
We have
\begin{equation} \la{fsu}
   f(s,u,t) =  \frac s2 + \frac{us -2t}{2\sqrt{u^2-4}}> \frac s2 + \frac{us -2u}{2\sqrt{u^2}} = s-1.
\end{equation}
Hence
\begin{equation} \la{ib}
   g(s,u,t) =  f(t,u,s)/\lambda_u > (t-1)/\lambda_u >0.
\end{equation}
Then \e{wa} follows from \e{fsu}, \e{ib} since $f(s,u,t)+g(s,u,t)=s$.
The bounds \e{wb} are shown similarly. Note that $g(u,s,t)$ can be negative.
\end{proof}

The function $\rho(x):[2,\infty) \to [0,\infty)$ is strictly increasing, has the power series expansion
\begin{equation} \la{rs}
  \rho(x) = \log x - \sum_{m=1}^\infty \frac 1{2m} \binom{2m}{m} x^{-2m}, \qquad (x \gqs 2),
\end{equation} and easily satisfies 
\begin{equation} \la{cru}
  \rho(x) = \log(x) -\rho^*(x) \qquad \text{for} \qquad 0<\rho^*(x) \lqs 4 \log 2/x^2, \qquad (x\gqs 2).
\end{equation}

\begin{lemma} \la{rox}
For 
 $x\gqs x_0>2$ and $\delta>0$ we have
\begin{equation} \la{rr}
  \rho(x) < \rho(x+\delta)<\rho(x)+  \frac{x_0^2 \cdot \delta}{(x_0^2-4) x}.
\end{equation}
For 
$x_0>2$, $0<\delta \lqs 1$   and $x \gqs x_0+\delta$ we also have
\begin{equation} \la{rrb}
  \rho(x) -  \frac{2x_0^2 \cdot \delta}{(x_0^2-4) x} < \rho(x-\delta)<\rho(x).
\end{equation}
\end{lemma}
\begin{proof}
With \e{rs},
\begin{equation} \la{hg}
   \rho(x+\delta)-\rho(x) = \log(1+\delta/x)+  \sum_{m=1}^\infty \frac 1{2m \cdot x^{2m}} \binom{2m}{m} \left(1-\frac 1{(1+\delta/x)^{2m}} \right).
\end{equation}
The inequalities
\begin{equation*}
  \log(1+y) \lqs y \quad \text{for} \quad y>-1, \qquad 1-(1+y)^{-r} \lqs r y \quad \text{for} \quad r, y \gqs 0,
\end{equation*}
may be verified by calculus. Using the bound $\binom{n}{k} < 2^n$ also, \e{hg} implies
\begin{equation*}
   \rho(x+\delta)-\rho(x) < \frac{\delta}{x} +  \sum_{m=1}^\infty \left(\frac {4}{x_0^2} \right)^m \frac{\delta}{x}. 
\end{equation*}
This gives \e{rr}. Replacing $x$ by $x-\delta$ then gives \e{rrb}.
\end{proof}

\section{Tree number bounds} \la{bou}

\subsection{Estimates for $m_{k/n}$} \la{uni}

\begin{lemma} \la{s}
Let $T$ be a rising tree with base triple $(s_0,u_0,t_0)$ and a descending path that is either empty or consists of a single $0$-edge. Set 
$
  \omega_T:= \log(\min(s_0,t_0)/2)$ and $\omega'_T:= \log(u_0)$.
Then 
\begin{equation*}
  2e^{\omega_T n} \lqs m(T)_{k/n} < e^{\omega'_T n} \qquad \text{for all} \qquad k/n \in \Q \cap [0,1].
\end{equation*}

\end{lemma}
\begin{proof}
We have $s_0, t_0>2$ and, by our assumptions, $u_0\gqs s_0 t_0/2$.
Then the cases with $n=1, 2$ 
can be verified directly. Since the descending path has at most one edge, every oriented triple $(s,u,t)$, outside of the base triple, is ordered: $u>st/2$.   Suppose $k/n$ has $n\gqs 3$ and is the mediant of Farey neighbors $k_0/n_0$ and $k_0'/n_0'$. Then by induction, 
\begin{equation*}
  m_{k/n} > \tfrac 12 m_{k_0/n_0} m_{k_0'/n_0'} \gqs \tfrac 12 \cdot 2 e^{\omega_T n_0} \cdot 2 e^{\omega_T n_0'} = 2 e^{\omega_T n}, 
\end{equation*}
giving the lower bound. Similarly for the upper bound, using the fact that every oriented triple $(s,u,t)$, outside of the base triple, satisfies $ st > u$ by the local rule.
\end{proof}

\begin{prop} \la{sim2}
Let $T$ be any rising tree. There exist $\omega_T, \omega'_T >0$ so that
\begin{equation} \la{kk}
  e^{\omega_T n} < m(T)_{k/n} < e^{\omega'_T n}  \qquad \text{for all} \qquad k/n \in \Q \cap [0,1].
\end{equation}
\end{prop}
\begin{proof}
If $T$ is strictly rising then we are in the case of Lemma \ref{s}. Otherwise let $v$ be the vertex at the  end of the descending path. It is a vertex well or an endpoint of an edge well. The three trees with root $v$ satisfy the conditions of Lemma \ref{s} and so, written in terms of height (see Definition \ref{hei}),
\begin{equation*}
  e^{C \cdot h_v(m)} < m  <  e^{C'  \cdot h_v(m)},
\end{equation*}
for some $C, C'>0$ and all regions $m$ of $T$. The last step is to change the base point to the root $r$ of $T$. By Lemma \ref{height}, $\alpha  \cdot h_r(m) \lqs h_v(m) \lqs \beta \cdot  h_r(m)$ and the required constants are $\omega_T = C\alpha$, $\omega'_T = C'\beta$.
\end{proof}

Progressing trees have an upper bound with the same form as \e{kk}, but necessarily a much smaller lower bound since they contain $2$-regions. 
The following result will be needed in Lemma \ref{yk} and Theorem \ref{mc}.

\begin{prop} \la{sih}
Let $T$ be any progressing tree.  There exist $\alpha_T, \beta_T >0$ so that
\begin{equation} \la{wek}
  (2+\alpha_T \cdot n)e^{\beta_T \min(k,n-k)}  < m(T)_{k/n}  \qquad \text{for all} \qquad k/n \in \Q \cap (0,1).
\end{equation}
\end{prop}
\begin{proof}
Suppose $T$ has base $(s_0,u_0,t_0)$ and $s_0=2$. Then $m_{1/n}=u_0-2\delta +\delta n$, an arithmetic progression, for $\delta=u_0-t_0>0$. We first seek $\alpha, \beta>0$ so that
\begin{equation}\label{sek}
  (2+\alpha \cdot n)e^{\beta} <m_{1/n} \qquad \text{for all} \qquad n\gqs 2.
\end{equation}
Then \e{sek} is equivalent to
\begin{equation} \la{22}
  2e^{\beta} +2\delta-u_0 < (\delta - \alpha e^{\beta})n.
\end{equation}
We will make $\delta - \alpha e^{\beta}>0$, and so just need to ensure \e{22} is true for $n=2$. This boils down to
\begin{equation*}
   \alpha e^{\beta} < \delta, \qquad (1+\alpha) e^{\beta} < u_0/2.
\end{equation*}
If $e^{\beta} =\sqrt{u_0/2}$ then we require $\alpha< \sqrt{u_0/2} -1, \delta/\sqrt{u_0/2}$. So we may take the explicit choice
\begin{equation}\label{ey}
  \alpha =  \alpha_T = \tfrac 12 \min\left(\sqrt{u_0/2}-1, (u_0-t_0)/\sqrt{u_0/2}\right), \qquad \beta = \beta_T = \tfrac 12 \log\left( u_0/2\right).
\end{equation}
 The inequality \e{sek} is true with this choice, and it remains true for $s_0>2$ as well. 

We prove \e{wek} for $1\lqs k \lqs n/2$ next by induction on $k$, using that all oriented triples in the tree are ordered. Suppose $k/n$  is the mediant of Farey neighbors $k_0/n_0, k_0'/n_0' \lqs 1/2$  with $n_0, n_0'\gqs 2$. Then  
\begin{align*}
   m_{k/n} & > \tfrac 12 m_{k_0/n_0} m_{k_0'/n_0'} \\
   & > \tfrac 12 (2+\alpha \cdot n_0)e^{\beta k_0} (2+\alpha \cdot n_0')e^{\beta k_0'}\\
   & =  \left(2+\alpha (n_0+n_0') + \tfrac 12 \alpha^2 n_0 n_0'\right) e^{\beta (k_0+k_0')}\\
   &> (2+\alpha \cdot n)e^{\beta k}.
\end{align*}
This covers all numbers on the left side of the tree; the same bounds are valid on the right  by symmetry.
\end{proof}

\begin{cor} \la{fin}
There are only finitely many numbers below any given bound on each rising or progressing tree. In particular, a rising tree $T$ has a minimal region label $\xi_T>2$.
\end{cor}

The minimal regions of a rising tree are easy to find: they occur in the base triple if it is strictly rising, or otherwise at the well at the end of the descending path. The three smallest regions in a progressing tree are $2\lqs t_0<u_0$ if $s_0=2$.


\subsection{Key inequalities for $\rho$ on triples}
\begin{lemma} \la{kyw}
Let $(s,u,t)$ be an oriented triple  in a rising tree $T$ with base triple $(s_0,u_0,t_0)$ and minimal region $\xi_T$. Define $A_T := (1-4/\xi_T^2) \cdot \min(1,s_0 t_0/u_0)>0$. Then
\begin{gather}
st  \gqs \min(1,s_0 t_0/u_0) \cdot  u, \la{wus}\\
  \sqrt{(s^2-4)(t^2-4)}  > A_T \cdot u. \la{wus2}
\end{gather}
\end{lemma}
\begin{proof}
 The inequality $st>u$ is true for all oriented triples, except possibly the base triple, by the local rule. Then \e{wus} makes sure we can include the base triple.
Since $s\gqs \xi_T>2$, we have 
\begin{equation*}
  s^2-4 \gqs s^2 (1-4/\xi_T^2).
\end{equation*}
The same is true for $t$, and with \e{wus} we obtain \e{wus2}.
\end{proof}

Zagier's lemma \cite[Lemma 2]{z82} for the original Markov tree, may be generalized as follows. This gives the fundamental inequalities we will need to show the convexity, linearity and concavity of the encoding functions.

\begin{prop} \la{zag2}
For each rising tree $T$ there exists $\kappa_T>0$ so that the following holds. Let $(s,u,t)$ be  any oriented triple  in $T$. If $D=4$ then $\rho(u)=\rho(s)+\rho(t)$. Otherwise, 
\begin{alignat}{2}
  \rho(u) & < \rho(s)+\rho(t)  < \rho(u)+ \frac{\kappa_T}{u^2}  && \qquad \text{if} \qquad D<4,  \la{d4a}\\
  \rho(u)- \frac{\kappa_T}{u^2} & < \rho(s)+\rho(t)  < \rho(u) && \qquad \text{if} \qquad D>4. \la{d4b}
\end{alignat}
\end{prop}
\begin{proof}
There is a unique $u_1 >2$ so that $\rho(u_1)=\rho(s)+\rho(t)$. As in \cite[(13)]{z82},
\begin{align} 
  u_1 & =\rho^{-1}(\rho(s)+\rho(t)) = 2\cosh (\rho(s)+\rho(t)) \notag\\
  & = e^{\rho(s)+\rho(t)}+  e^{-\rho(s)-\rho(t)} = \lambda_s \lambda_t + \lambda^{-1}_s \lambda^{-1}_t  \notag\\
  & =\left(s t+\sqrt{(s^2-4)(t^2-4)} \right)/2. \la{u1}
\end{align}
It follows that
$
  u_1^2 +s^2 + t^2 -s t u_1 = 4.
$
Subtracting the equation $u^2+s^2 + t^2 -s t u = D$  from this finds
\begin{equation} \la{zg}
  (u_1-u)(u_1+u- st) = 4-D.
\end{equation}

Now assume that $(s,u,t)$ is ordered, meaning that $u>st/2$. Hence by \e{iv},
\begin{equation}\label{us}
  u = \left(st+\sqrt{(s^2-4)(t^2-4)+4(D-4)} \right)/2.
\end{equation}
If $D=4$ then comparing \e{u1} and \e{us} finds $u_1=u$ and so $\rho(u)=\rho(s)+\rho(t)$.
If  $D<4$ then $u < u_1$ and \e{zg} implies that
\begin{equation*}
  u_1 = u + \frac{4-D}{u_1+u- st} < u+ \frac{2(4-D)}{\sqrt{(s^2-4)(t^2-4)}},
\end{equation*} 
and so 
\begin{equation} \label{im}
  u<u_1  < u+ \frac{2|D-4|}{\sqrt{(s^2-4)(t^2-4)}}.
\end{equation}
Apply $\rho$ to \e{im}, and use Lemma \ref{rox} with $x_0= \xi_T$, the minimal region number, to see that
\begin{equation*}
  \rho(u)  < \rho(u_1) < \rho\left(u+ \frac{2|D-4|}{\sqrt{(s^2-4)(t^2-4)}}\right) < \rho(u) +  \frac{2 \xi_T^2 |D-4|}{u(\xi_T^2 -4)\sqrt{(s^2-4)(t^2-4)}}.
\end{equation*}
With \e{wus2}, this  establishes \e{d4a}. 

Similarly, when  $D>4$,
\begin{equation} \label{im2}
  u_1<u  < u_1+ \frac{2|D-4|}{\sqrt{(s^2-4)(t^2-4)}}.
\end{equation}
Apply $\rho$ to \e{im2} and use Lemma \ref{rox} as before. From the second inequality above we obtain
\begin{equation*}
  \rho(u)   < \rho(u_1) +  \frac{2 \xi_T^2 |D-4|}{u_1(\xi_T^2 -4)\sqrt{(s^2-4)(t^2-4)}}.
\end{equation*}
The $u_1$ in the denominator   may be replaced by the smaller $st/2$. Then apply \e{wus}, \e{wus2} to produce \e{d4b}.

We have assumed that  $(s,u,t)$ is an  ordered triple. This will always be the case if $T$ is strictly rising. Otherwise,  the finite set $S$ of oriented triples $(s,u,t)$ with central vertex on the descending path of $T$  are the only possible unordered ones.  Let $(s,u,t)$ be one of these with $u\lqs s t/2$. Lemma \ref{tr4} informs us that $D<4$ since the descending path is nonempty. By \e{u1}, $u_1> s t/2$, and hence $u <u_1$, giving $\rho(u)< \rho(s)+\rho(t)$ as desired. Finally, increase  $\kappa_T$ if necessary to ensure that \e{d4a} is true for all triples in $S$.
\end{proof}

The version of the last proposition for progressing trees is shown similarly. These have $D>4$ by Lemma \ref{pr4}, and a slightly weaker inequality:

\begin{prop} \la{zag2b}
For each progressing tree $T$ there exists $\kappa_T>0$ so that, for all of its oriented triples $(s,u,t)$,
\begin{equation*}
   \rho(u)- \frac{\kappa_T}{u^\delta}  < \rho(s)+\rho(t)  < \rho(u),
\end{equation*}
for $\delta=2$, except that if $s=2$ or $t=2$ we must take $\delta=1$.
\end{prop}

\section{The encoding function $\phi_T$} \la{y4}

\begin{adef} \la{phid}
{\rm Let $T$ be a Farey-indexed, real Markov tree. 
The {\em encoding function over the rationals} for $T$ is $\phi = \phi_T$, defined by
\begin{equation} \la{phi}
  \phi(k/n):= \rho(m_{k/n})/n,
\end{equation}
for $\rho(x) = \cosh^{-1}(x/2)$. 
Then $\phi:\Q \cap [0,1] \to \R_{\gqs 0}$ if regions of $T$ are at least $2$.
}
\end{adef}

Our main goal in this section is to prove:

\begin{theorem} \la{cont}
The function $\phi_T$ is continuous when $T$ is a rising or progressing tree. For $D<4$, $D=4$ and $D>4$ its graph is strictly convex, a straight line and strictly concave, respectively.
\end{theorem}

\subsection{$\phi_T$ at Farey neighbors}

Recall the Farey neighbor notation of Lemma \ref{qq}.

\begin{prop} \la{shutx}
Let $T$ be a rising tree or a progressing tree and fix $q$  in $\Q \cap (0,1)$. Write the oriented triple $(m_{q_0}, m_q, m_{q'_0})$ as $(s, u, t)$ 
and define $\varepsilon_j$, $\varepsilon_j'$ by means of
\begin{equation} \la{pha}
  \rho(m_{q_j})  = \rho(u) j  +\log f(s,u,t) + \varepsilon_j, \qquad
  \rho(m_{q'_j})  = \rho(u) j +\log f(t,u,s) + \varepsilon'_j. 
\end{equation}
Then $\varepsilon_j, \varepsilon'_j \to 0$ as $j \to \infty$.
\end{prop}
\begin{proof} 
 We have $u>2$ and by Proposition \ref{start},
\begin{equation} \la{opk}
  \rho(m_{q_j}) =  \rho\left(f(s,u,t) \cdot e^{\rho(u) j} + g(s,u,t) \cdot e^{-\rho(u) j}\right).
\end{equation}
Note that $f(s,u,t)>0$ necessarily, since $m_{q_j}\gqs 2$. Then Lemma \ref{rox} and \e{cru} imply, as $j\to \infty$,
\begin{align*} 
  \rho(m_{q_j}) & = \rho\left(f(s,u,t) \cdot e^{\rho(u) j}\right) + O\left(e^{-2\rho(u) j}\right) \\
   & = \log\left(f(s,u,t) \cdot e^{\rho(u) j}\right) + O\left(e^{-2\rho(u) j}\right).
\end{align*}
The arguments  for $\varepsilon_j'$ are the same, with $s$ and $t$ switching places.
\end{proof}

The   endpoint version of the last result has a similar proof, based on Proposition \ref{stx}.
\begin{prop} \la{shut2}
Let $T$ be a rising or progressing tree with base triple $(s_0,u_0,t_0)$.
When $q=0$, so that $q_j'=1/(j+1)$, write 
\begin{equation*}
   \rho(m_{q'_{j+1}})  = \begin{cases}
   \rho(s_0) j +\log f(u_0,s_0,t_0) + \varepsilon'_j & \text{if} \quad s_0 \neq 2,\\
   \log((u_0-t_0)j +t_0) + \varepsilon'_j & \text{if}  \quad s_0 = 2.
   \end{cases}
\end{equation*}
When $q=1$, so that $q_j=j/(j+1)$, write 
\begin{equation*}
   \rho(m_{q_{j+1}})  = \begin{cases}
   \rho({t_0}) j  +\log f(u_0,t_0,s_0) + \varepsilon_j & \text{if} \quad t_0 \neq 2,\\
   \log((u_0-s_0)j +s_0) + \varepsilon_j & \text{if}  \quad t_0 = 2.
   \end{cases}. 
\end{equation*}
Then $\varepsilon_j, \varepsilon'_j \to 0$ as $j \to \infty$.
\end{prop}


It follows directly from Propositions \ref{shutx} and \ref{shut2} that, for all $q  \in \Q \cap [0,1]$ as $j \to \infty$,
\begin{equation} 
 \label{to}
  \phi(q_j)  \to \phi(q) \qquad(q\neq 0),  \qquad \qquad
  \phi(q'_j)  \to \phi(q)  \qquad (q \neq 1),
\end{equation}
for the encoding functions of all rising and progressing trees.

\subsection{Slopes between Farey neighbors}

For unequal rationals $p, q \in [0,1]$, let $\g(p,q)$ denote the slope of the secant line  meeting the graph of $\phi$ at $(p,\phi(p))$ and $(q,\phi(q))$. 
If $q \in \Q \cap (0,1)$ has Farey neighbors $q_i$, $q_j'$, as in Lemma \ref{qq}, it is convenient to write $S_i$ and $S'_j$ for the slopes between $q$ and these neighbors:
\begin{equation}\label{sde}
  \g(p,q): = \frac{\phi(q)-\phi(p)}{q-p}, \qquad S_i=S_i(q):=\g(q_i,q), \qquad S'_j=S'_j(q):=\g(q,q'_j).
\end{equation}
 Figure \ref{slo} illustrates $S_i$ and $S'_j$ in an example with $D<4$.

\SpecialCoor
\psset{griddots=5,subgriddiv=0,gridlabels=0pt}
\psset{xunit=0.6cm, yunit=0.5cm, runit=0.5cm}
\psset{linewidth=1pt}
\psset{dotsize=4pt 0,dotstyle=*}
\begin{figure}[!htb]
\centering

\psset{arrowscale=1.4,arrowinset=0.3,arrowlength=1.1}
\newrgbcolor{light}{0.8 0.8 1.0}
\newrgbcolor{light}{0.7 0.7 1.0}
\newrgbcolor{pale}{1 0.7 1}
\newrgbcolor{pale}{1 0.7 0.4}
\newrgbcolor{darkbrown}{0.5 0.324219 0.257813}

\begin{pspicture}(-7,-1.7)(4,6) 

\psset{linecolor=light}
\psline(0,0)(2.59808, 1.5)

\psline(0,0)(-2.89778, -0.776457)

\psline(0,0)(-5.47668, 1.12072)
\psline(0,0)(-4.27781, 0.510205)
\psline(0,0)(-3.13071, 0.0928761)

\psline(0,0)(2.14808, 2.27942)
\psline(0,0)(2.6641, 3.38564)
\psline(0,0)(3.08013, 4.66506)

\psset{linecolor=red}

\psline(-6.7273, 1.92442)(-5.47668, 1.12072)(-4.27781, 0.510205)
(-3.13071, 0.0928761)(-2.03538, -0.131268)

\psline(-2.03538, -0.131268)(-1.82252, -0.152914)(-1.61174, -0.166833)
(-1.40302, -0.173025)(-1.19638, -0.17149)(-0.991808, -0.162226)(-0.789305, -0.145236)
(-0.588873, -0.120518)(-0.390511, -0.0880728)(-0.19422, -0.0479001)(0., 0.)

\psline(1.53205, 1.34641)(2.14808, 2.27942)(2.6641, 3.38564)(3.08013, 4.66506)(3.39615, 6.11769)

\psline(0., 0.)(0.171205, 0.103464)(0.33841, 0.213856)(0.501615, 
0.331177)(0.66082, 0.455426)(0.816025, 0.586603)(0.96723, 
0.724708)(1.11444, 0.869741)(1.25764, 1.0217)(1.39685, 
1.18059)(1.53205, 1.34641)

\psdot[linecolor=black](0,0)

\psdots[linecolor=black](-5.47668, 1.12072)(-4.27781, 0.510205)
(-3.13071, 0.0928761)(-2.03538, -0.131268)

\psdots[linecolor=black](1.53205, 1.34641)(2.14808, 2.27942)(2.6641, 3.38564)(3.08013, 4.66506)

\rput(0.5,-0.9){$(q,\phi(q))$}

\rput(-6.5,1.1){$S_0$}
\rput(-5.2,0.3){$S_1$}
\rput(-4,-0.4){$S_2$}
\rput(-2,-1.1){$S$}

\rput(2.2,0.7){$S'$}
\rput(3.75,5.4){$S_0'$}
\rput(3.5,4){$S_1'$}
\rput(3,2.7){$S_2'$}

\end{pspicture}
\caption{The slopes $S_i$ and $S_j'$ of secant lines for the graph of $\phi$.}
\label{slo}
\end{figure}
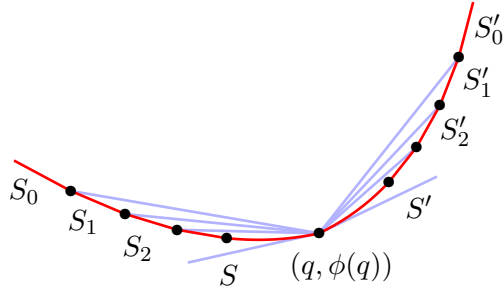

\begin{prop} \la{sjs0}
Let $T$ be a rising or progressing tree and fix $q=k/n \in \Q \cap (0,1)$. In the above notation, for all $i, j \gqs 0$ we have  $S_i=S'_j$ if $D=4$. If $D<4$ then the sequence $S_i$ is strictly increasing and $S'_j$ is strictly decreasing.  If $D>4$ then this is reversed: the sequence $S_i$ is strictly decreasing and $S'_j$ is strictly increasing.
There exists $\kappa_T>0$ so that
\begin{equation}
\begin{aligned} \la{dw4}
   0 & < S'_0 - S_0 < \kappa_T n/m_q^2 && \qquad  \text{if} \quad D<4,  \\
  -\kappa_T n/m_q^\delta & < S'_0 - S_0 < 0 && \qquad  \text{if}  \quad D>4,
\end{aligned}
\end{equation}
for  $\delta=2$, except that if $m_{q_0}=2$ or $m_{q'_0}=2$ we must take $\delta=1$.
\end{prop}
\begin{proof}
Unwinding definitions and using \e{nb} shows
\begin{equation} \la{zu}
  S_{i+1}-S_i = n\left(\rho(m_q)+\rho(m_{q_i}) - \rho(m_{q_{i+1}}) \right).
\end{equation}
Since $(m_{q_i},m_{q_{i+1}},m_q)$ is an oriented triple, it follows from Proposition \ref{zag2} that the right side of \e{zu} is positive, zero or negative according as $D<4$, $D=0$ or $D>4$, respectively.
Similarly for $S_j'$. We also have 
\begin{equation} \la{zu2}
  S'_0-S_0 = n\left(\rho(m_{q_0})+\rho(m_{q_0'}) - \rho(m_q) \right),
\end{equation}
 which implies $S_0=S_0'$ when $D=4$ and implies \e{dw4} using Propositions \ref{zag2}, \ref{zag2b}.
\end{proof}

Going further, with the same notation in place and recalling $f$ from \e{fg}:

\begin{prop} \la{sjs}
Let $T$ be a rising or progressing tree and fix $q \in \Q \cap (0,1)$.  Put
\begin{equation} \la{ss}
  S =S(q):=  \rho(u) n_0   -n \log f(s,u,t), \qquad  S'=S'(q) :=  -\rho(u) n_0'   +n \log f(t,u,s).
\end{equation}
Then 
\begin{equation} \la{sps}
  S'-S = n \log\left(1-\frac{D-4}{u^2-4}\right).
\end{equation}
For all $i, j \gqs 0$ we have  $S_i=S=S'=S'_j$ if $D=4$, and
\begin{align}
   S_i<S<S'<S'_j & \qquad \text{if} \qquad D<4,  \la{<}\\
  S_i>S>S'>S'_j & \qquad \text{if} \qquad D>4. \la{>}
\end{align}
For $D \neq 4$, we have  $S_i \to S$ and $S_j' \to S'$ strictly monotonically as $i, j \to \infty$. 
\end{prop}
\begin{proof}
We have 
\begin{align*}
  S'-S & =  -\rho(u) n_0'   +n \log f(t,u,s)  -\rho(u) n_0   +n \log f(s,u,t)\\
   & = -n \rho(u) +n \log(f(s,u,t)f(t,u,s)) = n \log (f(s,u,t)f(t,u,s)/\lambda_u).
\end{align*}
Then
\begin{equation} \la{thu}
  \frac{f(s,u,t)f(t,u,s)}{\lambda_u} =  \frac{(\lambda_u s-t)(\lambda_u t-s)}{\lambda_u (u^2-4)}
  = \frac{\lambda_u st-s^2 -t^2 +\lambda_u^{-1} s t}{u^2-4} = \frac{u^2-D}{u^2-4},
\end{equation} 
and we have established $S'-S = n \log((u^2-D)/(u^2-4))$, which is equivalent to \e{sps}. 
Recall from the proof of Proposition \ref{shutx} that $f(s,u,t), f(t,u,s)>0$. It follows that $u^2-D>0$.
Hence we also obtain from \e{sps} that $D<4 \iff S<S'$ and $D=4 \iff S=S'$.

A calculation using \e{nb} and \e{pha} finds 
\begin{align*} 
  S_i & = -n \rho(m_{q_i}) + n_i \rho(u) \\
  & = \rho(u) n_0   -n \log f(s,u,t) - n \varepsilon_i \\
  & = S - n \varepsilon_i.
\end{align*}
For $D \neq 4$ we know from Proposition \ref{sjs0} that $S_i$ is strictly increasing or decreasing, and from Proposition \ref{shutx} that $\varepsilon_i \to 0$. Hence $S_i \to S$  strictly monotonically. Similarly for $S_j'$ where
$
  S'_j =   S' + n \varepsilon'_j$.
For $D = 4$ all the $S_i, S_j'$ are equal and therefore equal to the limits $S, S'$. (It may also be shown directly with \e{ejx}, \e{ejxb} that $\varepsilon_j = \varepsilon_j'=0$ when $D=4$.) 
\end{proof}

Similar work, identifying $S$, $S'$ in \e{ss} and the difference $S'-S$ in \e{sps}, is found in \cite[pp. 6-8]{hin} for the $D=0$ case. The quantities $S$ and $S'$ will become the left and right derivatives at $q$. See also \cite[Sect. 3]{ga22}.

\begin{proof}[Proof of Theorem \ref{cont}]
Assume first that $T$ is a rising tree with discriminator $D=4$. Consider the points $(q,\phi(q))$ for $q\in \mathcal F_n$, the Farey fractions of order $n$ for large $n$. Proposition \ref{sjs0} shows that the slopes between every pair of these points is always the same. So the graph of $\phi$ is a continuous straight line.

Suppose next that $D<4$. Strict convexity of the graph is equivalent to showing $\g(p_1,p_2)<\g(p_2,p_3)$ for all rationals $0\lqs p_1<p_2<p_3\lqs 1$. To prove this, choose $n$ large enough  that the $p_i$ we are checking are all in $\mathcal F_n$. Draw the line segments connecting all of these consecutive Farey neighbor points  on the graph. By \e{<}, the slopes of these segments strictly increase moving right from $p_1$ to $p_3$.
Viewed from $p_2$, each segment to its left must make a triangle with one corner at $p_2$. Hence  $\g(p_1,p_2)\lqs m$ for $m$ the slope of the segment meeting $p_2$ on the left. Similarly,  $m' \lqs \g(p_2,p_3)$ for $m'$ the slope of the segment meeting $p_2$ on the right. Since $m<m'$, this confirms that $\phi$ is a strictly convex function. Similarly, $\phi$ is strictly concave when $D>4$.

When $D<4$, we claim that there exists $\alpha \in [0,1]$ so that $\phi$ is strictly decreasing on $\Q\cap [0,\alpha]$ and strictly increasing on $\Q\cap [\alpha,1]$. To see this, take the Farey fraction points of order $n$ on the graph for some large $n$ and draw the line segments between consecutive points. These slopes strictly increase, moving from left to right. As $n \to \infty$, we converge on a minimum at  $\alpha \in [0,1]$ which may be rational or irrational. Similarly, for $D>4$ there exists $\alpha \in [0,1]$ so that $\phi$ is strictly increasing on $\Q\cap [0,\alpha]$ and strictly decreasing on $\Q\cap [\alpha,1]$. 

Next let $q \in \Q\cap [0,1)$ and we will show that $\phi$ is continuous at $q$ from the right. Take any $\varepsilon>0$. By \e{to} there exist right Farey neighbors $q'_j$ of $q$ with $|\phi(q)-\phi(q'_j)|<\varepsilon$. Choose $j$ large enough that $\phi$ is strictly increasing or strictly decreasing between $q$ and $q'_j$. In other words, $\alpha$ above should not be in the interval $(q,q'_j)$. Hence $|\phi(q)-\phi(p)|<\varepsilon$
for all $p \in \Q\cap [q,q'_j]$ as desired. Continuity on the left at $q \in \Q\cap (0,1]$ is proven in the same way.
\end{proof}

We have demonstrated that $\phi$ is continuous on $\Q\cap [0,1]$. It does not automatically follow that $\phi$ extends to a continuous function on all of $[0,1]$. 
We show that 
it does 
in the next section using its concavity/convexity.
The differentiability properties of $\phi$ will require the next endpoint version of Proposition \ref{sjs}.

\begin{prop} \la{sjs2}
Suppose a rising or progressing tree has base triple $(s_0, u_0, t_0)$.
For $q=0$ put
\begin{equation}\label{pu}
  S'=S'(0) :=  -2\rho(s_0)    + \log f(u_0,s_0,t_0).
\end{equation}
Then for all $j \gqs 0$ we have, writing $S'_j$ for $S'_j(0)$,
\begin{equation*}
   S'<S'_j  \quad \text{if} \quad D<4,  \qquad S'=S'_j \quad \text{if} \quad D=4, \qquad S'>S'_j  \quad \text{if} \quad D>4.
\end{equation*}
For $q=1$ put
\begin{equation}\label{pu2}
  S=S(1) :=  2\rho(t_0)    - \log f(u_0,t_0,s_0).
\end{equation}
Then for all $i \gqs 0$ we have, writing $S_i$ for $S_i(1)$,
\begin{equation*}
   S_i<S  \quad \text{if} \quad D<4,  \qquad S=S_i \quad \text{if} \quad D=4, \qquad S_i>S  \quad \text{if} \quad D>4.
\end{equation*}
The formulas are undefined in two cases we must clarify:  $S'(0)=\infty$ if $s_0=2$ and $S(1)=-\infty$ if $t_0=2$.
For $D \neq 4$, we have  $S_i \to S$ and $S_j' \to S'$ strictly monotonically as $i, j \to \infty$.
\end{prop}
\begin{proof}
Its proof is similar to that of Proposition \ref{sjs}, but based on Proposition \ref{shut2} now. If the tree is progressing, with $s_0=2$, then $D>4$ by Lemma \ref{pr4}, and a calculation finds
\begin{equation*}
  S_{j+1}'(0) = \frac{\phi(q_{j+1}') -\phi(0)}{q_{j+1}' -0} = \frac{\rho(m_{q_{j+1}'})/(j+2) -\rho(m_0)/1}{1/(j+2)}= \rho(m_{q_{j+1}'}).
\end{equation*}
Similarly for $S_{i+1}(1)$ and therefore, as $j \to \infty$,
\begin{align*}
  S_{j+1}'(0) & = \log((u_0-t_0)j +t_0) + \varepsilon'_j \to \infty, \\
   S_{i+1}(1) & = 
 -\log((u_0-s_0)i +s_0) - \varepsilon_i \to -\infty,
\end{align*}
by Proposition \ref{shut2}. This requires the limits  $S'(0)=\infty$ if $s_0=2$ and $S(1)=-\infty$ if $t_0=2$.
\end{proof}

One further required result  quantifies the curvature: 

\begin{cor} \la{sldi}
Let $T$ be a rising tree with $D\neq 4$ or a progressing tree. 
Suppose $p, q, p'$ are in $\Q\cap[0,1]$ with $p<q<p'$ and let   $n$ be the denominator of $q$. Then the difference of the slopes on the graph of $\phi_T$ between these points can be estimated precisely with
\begin{equation} \la{cq}
 \frac{\kappa'_T n}{m_q^2} < |\g(p,q) -\g(q,p')| < \frac{\kappa_T n}{m_q^\delta},
\end{equation}
where the upper bound requires that  $p$ and $p'$ are each Farey neighbors of $q$.
Here  $\kappa_T'$ and $\kappa_T$ are positive constants depending only on $T$. Also $\delta=2$, except that if $m_{q_0}=2$ or $m_{q'_0}=2$ we must take $\delta=1$.
\end{cor}
\begin{proof} First assume that  $p$ and $p'$ are each Farey neighbors of $q$. By Lemma \ref{qq} this means $p=q_i$ and $p'=q'_j$ for some $i,j \gqs 0$. In the notation of Propositions \ref{sjs0}, \ref{sjs} we have $\g(p,q) -\g(q,p') = S_i-S_j'$ and these results show that
\begin{equation*}
  |S-S'|< |S_i-S_j'| \lqs |S_0-S_0'|.
\end{equation*}
See Figure \ref{slo} for an illustration. The upper bound in \e{cq} now follows from \e{dw4}. For the lower bound use \e{sps} and that $|\log(1+x)|\gqs \log(2) \cdot \min(1,  |x|) $ for $1+x>0$. Hence
\begin{align*}
   |S-S'| = \left| n \log\left(1+\frac{4-D}{m_q^2-4}\right)\right| & > n  \log(2) \cdot \min \left(1,  \frac{|4-D|}{m_q^2}\right) \\
   & =  n \cdot m_q^{-2}  \cdot \log(2) \min \left(m_q^2,  |4-D|\right).
\end{align*}
This implies the lower bound in \e{cq} for $\kappa_T' = \log(2) 
  \min(4, |D-4|)$.

To obtain the lower bound for general 
$p, p'$  with $p<q<p'$, use that there exist Farey neighbors $q_i$ and $q'_j$ of $q$ with  $p<q_i<q<q_j'<p'$. The strict concavity or convexity of $\phi$ then implies $|\g(p,q) -\g(q,p')| > |S_i-S_j'|$ and the previous work applies.
\end{proof}

\section{Extending $\phi$ to a function $\Phi$ on $\R\cap [0,1]$} \la{y5}

\subsection{Differentiability}
Recall Propositions \ref{sjs}, \ref{sjs2} and their notation as well as $\g(p,q)$ from \e{sde}. 

\begin{prop} \la{der}
 Let a rising or progressing tree have encoding function $\phi$. Take $q \in \Q\cap [0,1]$. Then $\phi$ has left derivative $S(q)$ at $q\neq 0$,
  and   right derivative $S'(q)$ at $q\neq 1$. 
\end{prop}
\begin{proof}
Assume $D<4$. 
We show next that the left derivative of $\phi$ at $q \in \Q\cap (0,1]$ exists and is $S = S(q)$.
 By Propositions \ref{sjs}, \ref{sjs2} we know $S_i$ is strictly increasing with limit $S$.
For any $\varepsilon>0$, choose $i$ large enough that $S-S_i<\varepsilon$.  We claim  that $S-\varepsilon< \g(p,q) <S$ for all rational $p$ with $q_i < p<q$. By convexity, $\g(p,q)$ is a strictly increasing function of $p$. Since $q_j \to q$ monotonically by Lemma \ref{qq}, there exists $j$ so that $p<q_j$. Hence
\begin{equation*}
  S - \varepsilon <S_i < \g(p,q) < S_j <S,
\end{equation*}
confirming the claim.
Similarly for the right derivative and for $D>4$. When $D=4$ then $S(q)=S'(q)$, the slope of the line.
\end{proof}

Therefore  $\phi$ is  semi-differentiable. 
 For $D \neq 4$
 the left and right derivatives are different at every $q\in \Q\cap(0,1)$ by \e{sps}.
 So $\phi$ is differentiable exactly  when $D=4$.


\begin{lemma} \la{yk}
  Let a rising or progressing tree have encoding function $\phi$. For each irrational $\alpha \in(0,1)$ there exists $\g_\alpha \in \R$ with the following property. Given any $\varepsilon >0$,  there is an open neighborhood $\mathcal N$ of $\alpha$ where $|\g(q,q')- \g_\alpha|<\varepsilon$ for all unequal rationals $q, q' \in \mathcal N$.
\end{lemma}
\begin{proof}
 First we fix any $\varepsilon >0$. Let $C_j=u_j/v_j$ be the sequence of continued fraction convergents to $\alpha$. Consecutive members of this sequence are Farey neighbors,  alternating  above and below $\alpha$, with $v_j$ strictly increasing to $\infty$. Choose $j$ so that $C_j<\alpha< C_{j+1}$, and large enough that the slope differences to points on the graph of $\phi$ at Farey neighbors on each side of $C_j$, $C_{j+1}$ are at most $\varepsilon/2$. This is possible since, for example on the left, Corollary \ref{sldi} shows the difference is $O(m^{-1}_{C_j})$, and Propositions \ref{sim2}, \ref{sih} find this is $O(e^{-\omega_T v_j})$ and $O(1/v_j)$, respectively. Fix $p$ as any left Farey neighbor of $C_j$ and $p'$ any right Farey neighbor of $C_{j+1}$. Let
\begin{equation*}
  \g_1=\g(p,C_{j}), \qquad \g_2=\g(C_{j},C_{j+1}), \qquad \g_3=\g(C_{j+1},p'),
\end{equation*}
be the slopes of the line segments. For definiteness, we may assume $D<4$, ($D=4$ is trivial). 
 Then
\begin{equation*}
  \g_1<\g_2<\g_1+\varepsilon/2, \qquad \g_2<\g_3<\g_2+\varepsilon/2.
\end{equation*}
As $\phi$ is strictly convex by Theorem \ref{cont}, any unequal rationals $q, q' \in (C_j, C_{j+1})$ must have 
\begin{equation*} 
  \g_1< \g(q,q') < \g_3 < \g_1+\varepsilon.
\end{equation*}

Consider a positive sequence $\varepsilon_i \to 0$. The above argument for $\varepsilon = \varepsilon_i$ produces a $t_i$ $(=\g_1)$ and a neighborhood $\mathcal N_i$ $(=(C_j, C_{j+1}))$ of $\alpha$ so that $|\g(q,q')-t_i|< \varepsilon_i$ for all unequal rationals $q, q' \in \mathcal N_i$. Then $t_i$ is a Cauchy sequence, converging to $\g_\alpha$ with the desired property.
\end{proof}

\begin{lemma} \la{yu}
  Let a rising or progressing tree have encoding function $\phi$. Fix any irrational $\alpha \in(0,1)$. For every $\varepsilon >0$,  there exists an open neighborhood $\mathcal N$ of $\alpha$ so that $|\phi(q')-\phi(q)|<\varepsilon$ for all rational $q, q' \in \mathcal N$.
\end{lemma}
\begin{proof}
By Lemma \ref{yk},  there exist $\g_\alpha$ and a neighborhood $\mathcal N$ of $\alpha$ so that $|\g(q,q')- \g_\alpha|<1$ for all unequal rationals $q, q' \in \mathcal N$. Hence
\begin{equation*}
  |\phi(q')-\phi(q)| < (|\g_\alpha| +1)|q'-q|.
\end{equation*}
Shrinking $\mathcal N$ if necessary to have width $\varepsilon/(|\g_\alpha|+1)$ gives the result.
\end{proof}

It follows from Lemma \ref{yu} that if a sequence of rationals $x_j$  converges to an irrational $x \in(0,1)$, then  $\phi(x_j)$ is a Cauchy sequence and must have a limit.  Define $\Phi(x)$ to be this limit, necessarily the same for all rational sequences  converging to $x$.

\begin{adef} \la{phid2}
{\rm Let $T$ be a rising tree or a progressing tree.
The {\em encoding function over the reals} for $T$ is $\Phi = \Phi_T$, well defined by
\begin{equation}\label{pph}
  \Phi(x) := \lim_{x_j \in \Q \to x}\phi(x_j) \qquad \qquad x \in \R \cap [0,1].
\end{equation}
}
\end{adef}


\begin{proof}[Proof of Theorem \ref{mthm}]
If $D=4$ then it is clear that points $(x,\Phi(x))$, as limits of points on a line by Theorem \ref{cont}, must be on this same line. So the graph of $\Phi$ is a continuous straight line.

For $D<4$ we claim now that the graph of $\Phi$ is strictly convex. For any real $u_1<u_2<u_3$ in $[0,1]$, with  two rational and one irrational,  by taking  $x_j \in \Q$ approaching the  irrational $u_i$, and using the strict convexity of $\phi$, it is straightforward to show that $\g(u_1,u_2) \lqs \g(u_2,u_3)$. Next take any real $r_1 <r_2<r_3$ in $[0,1]$. Strict convexity follows if it can be shown that $\g(r_1,r_2) < \g(r_2,r_3)$. Include any six rational numbers $p_i$ between them as follows,
\begin{equation*}
  \bm{r_1} \ < p_1<p_2<p_3< \ \bm{r_2} \ <p_4<p_5<p_6<\ \bm{r_3}.
\end{equation*}
Then 
$\g(r_1,p_1)\lqs \g(p_1,p_2) < \g(p_2,p_3)\lqs \g(p_3,r_2)$ and as a consequence $\g(r_1,r_2)<\g(p_3,r_2)$. Similarly on the right  $\g(r_2,p_4)<\g(r_2,r_3)$. Combined with $\g(p_3,r_2) \lqs \g(r_2,p_4)$, we have verified  $\g(r_1,r_2) < \g(r_2,r_3)$ and that $\Phi$ is strictly convex.

In the same way, $\Phi$ is strictly concave for $D>4$. Now the arguments in the proof of Theorem \ref{cont} go through to show that, for all $D$, $\Phi$ is continuous on its domain $\R \cap [0,1]$. The arguments for Proposition \ref{der} and Lemma \ref{yk} also go through, including all reals in the constructed intervals instead of only rationals. So the left and right derivatives of $\Phi(x)$ at rational $x=q$ are $S(q)$ and $S'(q)$, respectively. At irrational $x=\alpha$ the derivative of $\Phi(x)$ exists and is $\g_\alpha$.
\end{proof}


\subsection{Summary of how to compute left and right derivatives of $\Phi$ at rationals.}
Let $T$ be a rising or progressing tree with base triple $(s_0,u_0,t_0)$. For $q=k/n \in \Q \cap (0,1)$, let $q_0=k_0/n_0<q_0'=k_0'/n_0'$ be the Farey neighbors with mediant $q$. Write the oriented triple $(m_{q_0},m_q,m_{q_0'})$ as $(s,u,t)$ for convenience, and
recall our notation
\begin{equation*}
   \lambda_u  := (u+\sqrt{u^2-4})/2, \qquad \rho(u)  :=\log(\lambda_u), \qquad f(s,u,t) := (\lambda_u s-t)/\sqrt{u^2-4}.
\end{equation*}
The left and right derivatives of $\Phi_T(x)$ (and $\phi_T(x)$) at $x = q$ are then given by
\begin{align*}
  \left.\tfrac{d}{dx}\right|_{x=q^-}\Phi(x) & = S(q) =  \rho(u) n_0   -n \log f(s,u,t), \\
  \left.\tfrac{d}{dx}\right|_{x=q^+}\Phi(x) & = S'(q) =  -\rho(u) n_0'   +n \log f(t,u,s),
\end{align*}
respectively, as in  Proposition \ref{sjs}. At $x=0, 1$, see Proposition \ref{sjs2},
\begin{align*}
 \left.\tfrac{d}{dx}\right|_{x=0^+}\Phi(x) & = S'(0) =   -2\rho(s_0)    + \log f(u_0,s_0,t_0),  \\
 \left.\tfrac{d}{dx}\right|_{x=1^-}\Phi(x) & = S(1) =  2\rho(t_0)    - \log f(u_0,t_0,s_0),
\end{align*}
where we understand $S'(0)=\infty$ if $s_0=2$ and $S(1)=-\infty$ if $t_0=2$.

For example, with $T=(2.1,5,2.4)$ and $q=2/3$ use that $q_0=1/2, q_0'=1/1$ and $(s,u,t)=(5,9.9,2.4)$ to compute the left and right derivatives of $\Phi_T(x)$ at $x=2/3$ as approximately $-0.1447, -0.3415$, respectively. See the right of Figure \ref{phplots}.
The right derivative at $0$ and the left derivative at $1$ are $1.3090, -0.4490$, respectively.
Derivatives  of $\Phi(x)$ at irrational $x$ are computed in Section \ref{seri}.

\subsection{Infinite flatness} \la{ite}
 Suppose that $f(x)$ is defined on an open interval of $\R$ containing $a$, and $f'(a)$ exists. Similarly to \cite[Sect. 2]{mr95b}, we may say $f$ is {\em infinitely flat} 
 at $a$ if for every  integer $n \gqs 2$ there is a $\delta_n>0$ so that 
 \begin{equation}\label{hy}
 f(x+a)-f(a)-f'(a) x =O\left(|x|^n\right) \quad \text{for all} \quad x \in (-\delta_n, \delta_n).
 \end{equation}
 A well-known example is $e^{-1/x^2}$ at $0$. We show next that $\Phi$ for a rising tree is infinitely flat at
  irrationals $\alpha$ with bounded continued fraction coefficients. This means that $\alpha = [a_0, a_1, a_2,  \dots]$ and $a_i \leq K$ for some $K$ and all $i$. The set of such irrationals is dense in $\R$ since it contains the quadratic irrationals. Note that all the continued fractions in this paper are standard ones with integer coefficients $a_i\gqs 1$ for $i\gqs 1$. 

\begin{theorem} \la{et}
The encoding function $\Phi$ for a rising tree $T$ is infinitely flat at any irrational $\alpha$ that has  continued fraction coefficients that are bounded by some $K$.
In fact, there exist $A, B >0$ depending only on $T$ and  $K$, 
so that for $x$ in a neighborhood of $0$,
\begin{equation} \la{abd}
  \left|\Phi(x+\alpha)-\Phi(\alpha)-\Phi'(\alpha)x \right| \lqs A \sqrt{|x|} \cdot e^{-B/\sqrt{|x|}}.
\end{equation}
\end{theorem}
\begin{proof}
Let $C_j=u_j/v_j=[a_0,a_1, \dots, a_j]$ be the sequence of continued fraction convergents to $\alpha$, each a Farey neighbor of the next. Suppose $p<C_{j+1}<\alpha<C_{j}<p'$ with $p$ a left Farey neighbor of $C_{j+1}$ and $p'$ a right neighbor of $C_{j}$.  Let
\begin{equation*}
  \g_1=\g(p,C_{j+1}), \qquad \g_2=\g(C_{j+1},C_{j}), \qquad \g_3=\g(C_{j},p'),
\end{equation*}
be the slopes of the line segments. Assume for now that $D<4$, with $\Phi$  strictly convex by Theorem \ref{mthm}. Then  Corollary \ref{sldi} gives an upper bound for slope differences, and with Proposition \ref{sim2},
\begin{equation*}
  \g_2<\g_1+ \kappa_T v_{j+1} e^{-2 \omega_T v_{j+1}}, \qquad \g_3<\g_2+ \kappa_T v_{j}  e^{-2 \omega_T  v_{j}}.
\end{equation*}
We have $\g_1 < \Phi'(\alpha)$, and hence
\begin{equation} \la{ky}
  \g_3 < \Phi'(\alpha) + \kappa_T (v_j+v_{j+1}) e^{-2 \omega_T  v_{j}}.
\end{equation}
Now let $\g_4 = \g(\alpha,C_{j})$. By convexity,
\begin{equation} \la{kyb}
 \g_4 < \g_3, \qquad \Phi(x+\alpha) \lqs \Phi(\alpha) + \g_4 x \quad \text{for} \quad 0 \lqs x \lqs C_j-\alpha.
\end{equation}
Together \e{ky}, \e{kyb} give
\begin{equation}\label{cv}
  \Phi(x+\alpha) \lqs \Phi(\alpha) + \Phi'(\alpha) x +  \kappa_T (v_j+v_{j+1}) e^{-2 \omega_T  v_{j}}x  \quad \text{for} \quad 0 \lqs x \lqs C_j-\alpha.
\end{equation}

The convergents $C_j$ approach $\alpha$ with, as in \cite[Chap. 1]{khi},
\begin{equation} \la{khh}
  \frac{1}{v_j(v_j+v_{j+1})} < C_j -\alpha < \frac{1}{v_j v_{j+1}}.
\end{equation}
The recurrence $v_{j+1} =v_j a_{j+1} +v_{j-1}$ implies that
\begin{equation}\label{bc}
  v_{j+1} <(K+1) v_j, \qquad v_{j+2} < (K+1)^2 v_j.
\end{equation}
Hence
\begin{equation} \label{bc2}
  \frac{1}{(K+2)v_j^2} < C_j -\alpha < \frac{1}{v_j^2}.
\end{equation}
The bound on $\Phi(x+\alpha)$ in \e{cv} will be used only for $x$ in the range $C_{j+2}-\alpha \lqs x \lqs C_j-\alpha$.
Then by \e{bc} and \e{bc2}, these $x$ values satisfy
\begin{equation*}
  \frac{1}{(K+2)^5v_{j}^2} <  x  < \frac{1}{v_j^2} \quad  \implies  \quad  \frac{x^{-1/2}}{(K+2)^{5/2}} <  v_j  < x^{-1/2}.
\end{equation*}
Set the positive constants $A :=\kappa_T(K+2)$ and $B:=2\omega_T (K+2)^{-5/2}$. With \e{cv}, we have shown that
\begin{equation}\label{cv2}
  \Phi(x+\alpha) \lqs \Phi(\alpha) + \Phi'(\alpha) x + A \sqrt{x} \cdot e^{-B/\sqrt{x}} \qquad \text{for} \qquad C_{j+2} \lqs x +\alpha \lqs C_j.
\end{equation}
The same arguments go through on the left of $\alpha$ with $C_{j+1} \lqs x +\alpha \lqs C_{j+3}$. Also the case where $D>4$ is similar, reversing inequalities. This proves \e{abd} for all $x$ in the interval $(C_{2}-\alpha,  C_3-\alpha)$ containing $\alpha$, and it follows that $\Phi$ is infinitely flat at $\alpha$.  
(For $D=4$ these results are trivially true.)
\end{proof}

As discussed in the introduction, McShane and Rivin claimed in \cite[Thm. 2.1]{mr95b} that the  boundary for the stable norm ball is  infinitely flat in all irrational directions. As we will see in Section \ref{norm}, there is a direct connection between this boundary and the graph of $\Phi$ when $D=0$. It follows that Theorem \ref{et} implies the boundary is infinitely flat in the dense set of directions that are irrational with bounded continued fraction coefficients.
 The preprint \cite{nxv} of  Doan,  Li, and  Nguyen appeared while this section was being written, giving an elegant exact criterion for infinite flatness on the McShane-Rivin norm ball boundary.   Here we translate the work in \cite[Sect. 5]{nxv} to our more general context and give an example of a point where $\Phi$ is not infinitely flat, and does not even have a second derivative.  
 
 Define the irrational number
\begin{equation} \la{bet}
  \beta=[b_0, b_1, b_2, \dots] = [0, 1, 2^{1\cdot 1}, 2^{2\cdot 3},2^{3\cdot 193}, \dots],
\end{equation}
recursively with $b_{j+1}= 2^{j \cdot v_j}$ where  $u_j/v_j = [b_0,b_1, \dots, b_j]$ is the $j$th convergent $C_j$. This number lies at the other extreme to the irrationals $\alpha$ in Theorem \ref{et}. From its construction it follows  that $v_{j+1}>2^{j \cdot v_j}$ and hence, as in \e{khh}, the convergents  approach $\beta$ extremely rapidly  with 
\begin{equation}\label{bh}
  |\beta-C_j| < e^{-A \cdot v_j},
\end{equation}
for an $A>0$ that can be taken arbitrarily large as $j \to \infty$.

\begin{theorem} (Based on \cite[Sect. 5]{nxv}) \la{etq}
Assume  that $T$ is a rising tree with $D\neq 4$ or a progressing tree. Let $\beta$ be given by \e{bet}. Then there exists  a sequence $x_j \to 0$ and an arbitrarily large $B >0$  such that
\begin{equation} \la{abd2}
  \left|\Phi(x_j+\beta)-\Phi(\beta)-\Phi'(\beta)x_j \right| > B |x_j|^2 \qquad \text{as} \qquad j \to \infty.
\end{equation}
\end{theorem}
\begin{proof}
Without losing generality, suppose again  that $D<4$, making $\Phi$ strictly convex. Let $\delta= C_j-\beta$, which we may assume is positive. 
The slope differences between $(C_j,\Phi(C_j))$ and any adjacent points on the graph can be bounded from below by Corollary \ref{sldi}. This corollary was proved for $\phi$, but using the strict concavity/convexity of $\Phi$ in the last paragraph of its proof shows its lower bound is also valid for $\Phi$. Combining this with the upper bound from Proposition \ref{sim2}, (valid for progressing trees), finds
\begin{equation}\label{ew}
  |\g(\beta, \beta+\delta) -\g(\beta+\delta, \beta+2\delta)| > c'  e^{-c  \cdot v_j},
\end{equation}
for some $c, c'>0$. 
This is equivalent to
\begin{equation*}
  \Phi(\beta +2\delta)   > 2(\Phi(\beta+\delta)-\Phi(\beta))+ \Phi(\beta) + c' \delta   e^{-c \cdot v_j}.
\end{equation*}
By convexity $(\Phi(\beta +\delta) - \Phi(\beta))/\delta >\Phi'(\beta)$ and this means
\begin{equation*}
  \Phi(\beta +2\delta)   -  \Phi(\beta) -2\delta \Phi'(\beta) >  c' \delta  e^{-c \cdot v_j}.
\end{equation*}
Take $x_j = 2\delta$ and include \e{bh} to produce
\begin{equation} \la{t80}
  \frac{ \left|\Phi(x_j+\beta)-\Phi(\beta)-\Phi'(\beta)x_j \right|}{|x_j|^2} > \frac{c'  e^{(A-c) v_j}}4.
\end{equation}
Choose $A>c$ and then the right side of \e{t80} grows without bound.
\end{proof}

Hence \e{hy} is not true even for $n=2$ at some irrationals that are very well approximated by rationals. They further show in \cite[Sect. 5]{nxv} that there are uncountably many such irrationals in any interval. In summary, the existence of higher derivatives of $\Phi(x)$ depends  on the arithmetic properties of  the number $x$. For $x=\alpha$ in Theorem \ref{et} all higher derivatives exist and are $0$, while for $x=\beta$ in Theorem \ref{etq} the second derivative does not exist or is infinite.

\subsection{Additional results}

For a Farey-indexed Markov tree $T$ and any $q \in \Q \cap (0,1)$ define the difference
\begin{equation} \la{de}
  \Delta(q) = \Delta_T(q):= \rho(m_q)-\rho(m_{q_0})-\rho(m_{q_0'}),
\end{equation}
for $q_0$ and $q_0'$ the Farey neighbors with mediant $q$. 
As we move up the Farey tree, $\Delta(q)$ becomes exponentially small for rising trees. Combining Propositions \ref{sim2} and \ref{zag2} finds:

\begin{cor} \la{fgh}
For every rising tree $T$ there exist $\kappa_T, \omega_T>0$ so that
\begin{equation} \la{deb}
  \left|  \Delta \left(\frac kn \right)\right| < \frac{\kappa_T}{m^2_{k/n}} < \kappa_T e^{-2\omega_T  n}  \qquad \text{for all} \qquad k/n \in \Q \cap (0,1),
\end{equation}
when $D\neq 4$. For $D = 4$ we have $\Delta(k/n) =0$.
\end{cor}

A stronger form of Proposition \ref{shutx} is given next for rising trees. Its uniformity  will be required for Theorems \ref{hnn} and \ref{hn}.

\begin{prop} \la{shut}
Let $T$ be a rising tree.
For $q$  in $\Q \cap (0,1)$ write the oriented triple $(m_{q_0}, m_q, m_{q'_0})$ as $(s, u, t)$, and 
define $\varepsilon_j$, $\varepsilon_j'$ as before with \e{pha}.
Then there exists $C_T$, depending only on $T$, so that
\begin{equation} \la{phu}
  |\varepsilon_j|, |\varepsilon'_j| \lqs  C_T e^{-2\rho(u) j} \qquad \text{for all} \qquad j \gqs 0.
\end{equation}
\end{prop}
\begin{proof} 
We focus on $\varepsilon_j$ and the arguments  for $\varepsilon_j'$ are the same, with $s$ and $t$ switching places. First assume that $(s,u,t)$ is ordered. By  \e{wa} this implies that $f(s,u,t)>s-1>1$ and $0<g(s,u,t)<1$.  For $j$ large enough that $f(s,u,t) \cdot e^{\rho(u) j} \gqs \sqrt{5}$, say, Lemma \ref{rox} may be applied to \e{opk} to obtain
\begin{equation} \la{ao}
  \rho(m_{q_j}) = \rho\left(f(s,u,t) \cdot e^{\rho(u) j}\right) + E_1, \qquad |E_1| \lqs 5 \left|\frac{g(s,u,t)}{f(s,u,t)}\right| e^{-2\rho(u) j}.
\end{equation}
 Also, with \e{cru},
\begin{equation} \la{ao2}
  \rho\left(f(s,u,t) \cdot e^{\rho(u) j}\right) = \log\left(f(s,u,t) \cdot e^{\rho(u) j}\right)  + E_2, \qquad |E_2| \lqs \frac{4\log 2}{f(s,u,t)^2} e^{-2\rho(u) j}.
\end{equation}
Set $J := \log(\sqrt{5})/\rho(u)$.
Since $\varepsilon_j =E_1+E_2$, this confirms \e{phu} with $C_T=5+4\log 2$ for every $j\gqs J$.

For  $0 \lqs j < J$ use
\begin{equation*}
   \varepsilon_j = \rho(m_{q_j})  - \rho(u) j  -\log f(s,u,t).
\end{equation*}
Then
\begin{equation} \la{ejx}
   \varepsilon_0 = \rho(s)  -\log f(s,u,t) = \log\left( \frac{\lambda_s}{f(s,u,t)}\right).
\end{equation}
The bounds $s/2\lqs \lambda_s\lqs s$ and \e{wa}  show that  $|\varepsilon_0| < \log 2$. Recalling \e{de}, 
\begin{equation} \la{ejxb}
   \varepsilon_j -   \varepsilon_0 = \Delta(q_j) +  \Delta(q_{j-1}) + \cdots +  \Delta(q_1).
\end{equation}
With \e{deb} the right side of \e{ejxb} is bounded by some $B_T$, depending only on $T$. Therefore
\begin{equation*}
  |\varepsilon_j| \lqs B_T  e^{2\rho(u) j} e^{-2\rho(u) j} < B_T  e^{2\rho(u) J} e^{-2\rho(u) j} 
  \lqs B_T  e^{2\log \sqrt{5}} e^{-2\rho(u) j},
\end{equation*}
confirming \e{phu} with $C_T=B_T  e^{2\log \sqrt{5}}$ for every $j < J$.

For the finite number of unordered oriented triples $(s,u,t)$ on the descending path of $T$, we argue similarly. The difference is that now we know less about $f(s,u,t)$ and $g(s,u,t)$. It must be true that  $f(s,u,t)>0$ for regions to remain $>2$, while $g(s,u,t)$ can be positive or negative.
We have $E_1, E_2, \varepsilon_0$ in \e{ao}, \e{ao2} and \e{ejx} depending on $s,u,t$. Increasing $C_T$, if necessary,  accommodates these cases.
\end{proof}

\section{Rising tree numbers below a given bound} \la{num}
The original motivation for this paper was to generalize Zagier's work in \cite{z82}, and that is achieved in this section. For a rising or progressing tree $T$ let 
\begin{equation*}
  M(X)=M_T(X) := \#\{ q\in \Q \cap [0,1] \, : \, m_q \lqs X \},
\end{equation*}
count (with their multiplicities) all numbers on $T$ that are at most $X$. Then $m_{k/n} = 2 \cosh(n \cdot \phi(k/n))$
means
\begin{equation} \la{eqmx}
  m_{k/n} \lqs X \iff n \cdot \phi(k/n) \lqs \rho(X).
\end{equation}
To convert from the coordinates of the graph $(k/n,\phi(k/n))$ to coordinates $(n,k)$, note that
\begin{equation*}
  (n,k)= \rho(X)\left(\frac 1{\phi(k/n)}, \frac{k/n}{\phi(k/n)}\right)
\end{equation*}
when there is equality in \e{eqmx}.
So we require the transformation
\begin{equation}\label{f}
  F:(x,y) \mapsto (1/y,x/y),
\end{equation}
 sending the vertical strip with $0\lqs x\lqs 1$ in the first quadrant to the triangular region $y\lqs x$ in the first quadrant. Note that $F$ maps the line $\ell$ given by $y=ax+b$ to another line:
  \begin{equation}\label{f2}
    F: \ell \mapsto y=-bx/a+1/a \quad (a\neq 0), \qquad F: \ell \mapsto x=1/b \quad (a=0, b\neq 0).
  \end{equation}

\begin{adef}\la{enc} 
{\rm The {\em  lattice  curve} for a Markov tree $T$ is denoted $\Gamma_T$. It is the image of the graph of $\Phi_T$ under the above transformation $F$.}
\end{adef}

 Considering dilates of lattice curves gives a direct way to picture the numbers $m(T)_{k/n}$, and the next lemma is an easy exercise; examples of the curves $\Gamma_T$ appear in Figures \ref{wplot}, \ref{wplot2ee}.

\begin{lemma} \la{latt}
Let $T$ be a rising or progressing tree. The dilated set $t \cdot \Gamma_T$ for $t>0$ intersects a lattice point $(n,k) \in \Z^2$ with $\gcd(n,k)=1$ exactly when $m(T)_{k/n}=2\cosh(t)$.
\end{lemma}

In polar coordinates, $(r\cos \theta, r \sin \theta) = (1/y,x/y)$ means $x=\tan \theta$ and $y=1/(r \cos \theta)$. Hence  $\Gamma_T$ is the polar graph $r=f(\theta)$ for  
\begin{equation} \la{polar}
  f(\theta)=\frac 1{\cos(\theta) \cdot \Phi(\tan \theta)} \qquad \quad (0\lqs \theta \lqs \pi/4).
\end{equation}
By \e{eqmx}, we have
\begin{equation} \la{av}
   m_{k/n} \lqs X  \iff \sqrt{n^2+k^2} \lqs f(\tan^{-1}(k/n)) \cdot \rho(X).
\end{equation}

It is easy to see that the arc length of the graph of $\Phi_T$ for a rising tree $T$ is well-defined and finite as follows. Taking any points on the graph to get a  polygonal approximation, we see that the slopes are bounded in absolute value by $m=\max(|S'(0)|, |S(1)|)$ by the concavity/convexity of Theorem \ref{mthm}. Hence the length of such an approximation is at most $m$. Including more points gives a strictly increasing arc length that must have a limit $L \lqs m$. Thus, the graph of $\Phi$ is rectifiable.
Let $\mathcal W = \mathcal W_T$ be the wedge shaped domain bounded by $\theta=0$, $\theta=\pi/4$ and $r=f(\theta)$. This boundary is a closed, rectifiable Jordan curve and so $\mathcal W$ has a well-defined area.

\SpecialCoor
\psset{griddots=5,subgriddiv=0,gridlabels=0pt}
\psset{xunit=7cm, yunit=7cm, runit=7cm}
\psset{linewidth=1pt}
\psset{dotsize=1.4pt,dotstyle=*}
\begin{figure}[!htb]
\centering

\psset{arrowscale=1.4,arrowinset=0.3,arrowlength=1.1}
\newrgbcolor{light}{0.8 0.8 1.0}
\newrgbcolor{pale}{1 0.7 1}
\newrgbcolor{pale}{1 0.7 0.4}

\newrgbcolor{markov}{0.902344 0.710938 0.945313}
\psset{linecolor=markov}

\begin{pspicture}(-0.0,-0.04)(1.45,0.79) 

\psset{linecolor=lightgray}

\psline{->}(0,0)(1.45,0)
\psline{->}(0,0)(0,0.72)
\psline(0,0)(0.72,0.72)

\rput(1.5,0){$n$}
\rput(-0.06,0.7){$k$}

\psset{dotsize=1.3pt,dotstyle=*}
\psset{linewidth=0.9pt}


\pspolygon[fillstyle=solid,
fillcolor=markov,linecolor=orange](0,0)(0.567296, 0.567296)
(0.625684, 0.500547)(0.642208, 0.481656)(0.654555, 0.46754)(0.671776, 0.44785)
(0.687573, 0.429733)(0.697413, 0.418448)(0.70901, 0.405148)
(0.739761, 0.369881)(0.771771, 0.330759)(0.785365, 0.314146)
(0.797657, 0.299121)(0.819022, 0.273007)(0.844617, 0.241319)
(0.851877, 0.23233)(0.864888, 0.216222)(0.881306, 0.195846)
(0.894897, 0.178979)(0.91608, 0.15268)
(1.03904, 0.)(0,0)

\psline[linecolor=white,linewidth=0.3pt,linestyle=dashed](0.567296, 0.567296)(1.03904, 0.)

\pspolygon[linecolor=orange](0,0)(0.567296, 0.567296)
(0.625684, 0.500547)(0.642208, 0.481656)(0.654555, 0.46754)(0.671776, 0.44785)
(0.687573, 0.429733)(0.697413, 0.418448)(0.70901, 0.405148)
(0.739761, 0.369881)(0.771771, 0.330759)(0.785365, 0.314146)
(0.797657, 0.299121)(0.819022, 0.273007)(0.844617, 0.241319)
(0.851877, 0.23233)(0.864888, 0.216222)(0.881306, 0.195846)
(0.894897, 0.178979)(0.91608, 0.15268)
(1.03904, 0.)(0,0)


\savedata{\mydatab}[
{{1.03904, 0.}, {0.567296, 0.567296}, {0.739761, 0.369881}, {0.819022,
   0.273007}, {0.671776, 0.44785}, {0.864888, 0.216222}, {0.785365, 
  0.314146}, {0.697413, 0.418448}, {0.642208, 0.481656}, {0.894897, 
  0.178979}, {0.844617, 0.241319}, {0.797657, 0.299121}, {0.771771, 
  0.330759}, {0.70901, 0.405148}, {0.687573, 0.429733}, {0.654555, 
  0.46754}, {0.625684, 0.500547}, {0.91608, 0.15268}, {0.881306, 
  0.195846}, {0.851877, 0.23233}, {0.836772, 0.251032}, {0.803373, 
  0.292136}, {0.792884, 0.304955}, {0.777378, 0.323907}, {0.764421, 
  0.339743}, {0.71562, 0.397567}, {0.704131, 0.410743}, {0.691325, 
  0.425431}, {0.683192, 0.434758}, {0.659628, 0.46174}, {0.650011, 
  0.472735}, {0.632922, 0.492273}, {0.615132, 0.51261}, {0.931834, 
  0.133119}, {0.906328, 0.164787}, {0.886113, 0.189881}, {0.876189, 
  0.202197}, {0.855308, 0.228082}, {0.849039, 0.235844}, {0.839985, 
  0.247054}, {0.832608, 0.256187}, {0.806676, 0.288098}, {0.800956, 
  0.295089}, {0.794696, 0.302741}, {0.790781, 0.307526}, {0.77971, 
  0.321057}, {0.775303, 0.326443}, {0.767619, 0.335833}, {0.759816, 
  0.345371}, {0.719892, 0.392668}, {0.712713, 0.400901}, {0.705921, 
  0.408691}, {0.702142, 0.413025}, {0.693005, 0.423503}, {0.689891, 
  0.427075}, {0.68503, 0.43265}, {0.680713, 0.437601}, {0.662392, 
  0.458579}, {0.65753, 0.464139}, {0.651771, 0.470723}, {0.647912, 
  0.475135}, {0.635751, 0.489039}, {0.630318, 0.49525}, {0.619884, 
  0.507178}, {0.60781, 0.52098}, {0.944009, 0.118001}, {0.924496, 
  0.14223}, {0.909746, 0.160543}, {0.902725, 0.169261}, {0.888408, 
  0.187033}, {0.884226, 0.192223}, {0.878275, 0.199608}, {0.873503, 
  0.20553}, {0.857307, 0.225607}, {0.853853, 0.229884}, {0.850113, 
  0.234514}, {0.847796, 0.237383}, {0.841331, 0.245388}, {0.838792, 
  0.248531}, {0.834414, 0.253952}, {0.830027, 0.259383}, {0.808827, 
  0.285468}, {0.805219, 0.289879}, {0.801841, 0.294008}, {0.799976, 
  0.296287}, {0.79551, 0.301745}, {0.794002, 0.303589}, {0.791662, 
  0.30645}, {0.789597, 0.308973}, {0.780988, 0.319495}, {0.778743, 
  0.322239}, {0.776105, 0.325463}, {0.774349, 0.327609}, {0.768878, 
  0.334295}, {0.766465, 0.337245}, {0.761881, 0.342847}, {0.75666, 
  0.349228}, {0.722879, 0.389242}, {0.717963, 0.39488}, {0.713757, 
  0.399704}, {0.711582, 0.402198}, {0.70675, 0.40774}, {0.705227, 
  0.409487}, {0.702964, 0.412082}, {0.701062, 0.414264}, {0.693959, 
  0.42241}, {0.6923, 0.424313}, {0.690438, 0.426447}, {0.68925, 
  0.42781}, {0.685781, 0.431788}, {0.684355, 0.433425}, {0.681801, 
  0.436353}, {0.679118, 0.43943}, {0.664132, 0.45659}, {0.661188, 
  0.459957}, {0.658305, 0.463252}, {0.656659, 0.465134}, {0.652548, 
  0.469835}, {0.651102, 0.471488}, {0.648798, 0.474122}, {0.646703, 
  0.476518}, {0.637258, 0.487315}, {0.63459, 0.490365}, {0.631334, 
  0.494088}, {0.629092, 0.496651}, {0.621685, 0.505119}, {0.618199, 
  0.509105}, {0.611168, 0.517142}, {0.602433, 0.527128}, {0.953701, 
  0.105967}, {0.938288, 0.125105}, {0.927051, 0.139058}, {0.921821, 
  0.145551}, {0.91139, 0.158503}, {0.9084, 0.162214}, {0.904189, 
  0.167442}, {0.900849, 0.17159}, {0.889752, 0.185365}, {0.887433, 
  0.188243}, {0.884939, 0.191338}, {0.883403, 0.193244}, {0.879153, 
  0.198518}, {0.877499, 0.200571}, {0.874665, 0.204088}, {0.871849, 
  0.207583}, {0.858616, 0.223987}, {0.856424, 0.226701}, {0.854385, 
  0.229225}, {0.853265, 0.230612}, {0.850598, 0.233914}, {0.849702, 
  0.235024}, {0.848316, 0.236739}, {0.847099, 0.238247}, {0.84207, 
  0.244472}, {0.840772, 0.24608}, {0.839252, 0.247961}, {0.838245, 
  0.249208}, {0.835127, 0.253069}, {0.833761, 0.25476}, {0.831182, 
  0.257953}, {0.828269, 0.261559}, {0.81034, 0.283619}, {0.807854, 
  0.286658}, {0.805741, 0.28924}, {0.804654, 0.290569}, {0.802251, 
  0.293506}, {0.801497, 0.294428}, {0.800381, 0.295793}, {0.799445, 
  0.296937}, {0.795974, 0.301179}, {0.795168, 0.302164}, {0.794267, 
  0.303265}, {0.793693, 0.303967}, {0.792022, 0.306009}, {0.791338, 
  0.306845}, {0.790117, 0.308338}, {0.788838, 0.309901}, {0.781795, 
  0.318509}, {0.78043, 0.320177}, {0.7791, 0.321802}, {0.778343, 
  0.322728}, {0.77646, 0.32503}, {0.7758, 0.325836}, {0.774751, 
  0.327117}, {0.773801, 0.328279}, {0.769551, 0.333472}, {0.768361, 
  0.334927}, {0.766915, 0.336694}, {0.765923, 0.337907}, {0.762668, 
  0.341885}, {0.761147, 0.343744}, {0.758103, 0.347464}, {0.754362, 
  0.352036}, {0.725085, 0.386712}, {0.721507, 0.390816}, {0.718646, 
  0.394096}, {0.717235, 0.395716}, {0.714249, 0.399139}, {0.713349, 
  0.400172}, {0.712046, 0.401667}, {0.71098, 0.402889}, {0.707228, 
  0.407192}, {0.7064, 0.408142}, {0.705491, 0.409185}, {0.704921, 
  0.409838}, {0.703305, 0.411691}, {0.70266, 0.412431}, {0.701532, 
  0.413724}, {0.700383, 0.415042}, {0.694573, 0.421705}, {0.69354, 
  0.42289}, {0.692559, 0.424016}, {0.692011, 0.424643}, {0.690683, 
  0.426166}, {0.690229, 0.426687}, {0.689519, 0.427502}, {0.688887, 
  0.428227}, {0.68619, 0.431319}, {0.685471, 0.432145}, {0.684616, 
  0.433124}, {0.684042, 0.433783}, {0.682225, 0.435866}, {0.68141, 
  0.436801}, {0.679838, 0.438605}, {0.678007, 0.440704}, {0.665327, 
  0.455224}, {0.663351, 0.457483}, {0.661622, 0.45946}, {0.660714, 
  0.460498}, {0.658662, 0.462844}, {0.658005, 0.463595}, {0.65702, 
  0.464721}, {0.656183, 0.465678}, {0.652986, 0.469334}, {0.652222, 
  0.470207}, {0.651358, 0.471195}, {0.650802, 0.471831}, {0.649158, 
  0.47371}, {0.648474, 0.474493}, {0.647236, 0.475909}, {0.645917, 
  0.477417}, {0.638195, 0.486244}, {0.636604, 0.488063}, {0.635021, 
  0.489873}, {0.634105, 0.49092}, {0.63178, 0.493578}, {0.630949, 
  0.494528}, {0.629611, 0.496057}, {0.628379, 0.497467}, {0.622632, 
  0.504036}, {0.62095, 0.505959}, {0.61886, 0.508349}, {0.617396, 
  0.510023}, {0.612414, 0.515717}, {0.609989, 0.51849}, {0.60493, 
  0.524273}, {0.598315, 0.531836}, {0.961599, 0.0961599}, {0.949116, 
  0.111661}, {0.94027, 0.122644}, {0.936225, 0.127667}, {0.928286, 
  0.137524}, {0.926043, 0.140309}, {0.922906, 0.144204}, {0.920437, 
  0.14727}, {0.912356, 0.157303}, {0.910691, 0.159371}, {0.908908, 
  0.161584}, {0.907815, 0.162941}, {0.904807, 0.166675}, {0.903644, 
  0.16812}, {0.901659, 0.170584}, {0.899698, 0.173019}, {0.890635, 
  0.184269}, {0.889157, 0.186103}, {0.887789, 0.187801}, {0.887039, 
  0.188732}, {0.885261, 0.190939}, {0.884665, 0.191677}, {0.883747, 
  0.192817}, {0.882942, 0.193816}, {0.879637, 0.197918}, {0.878788, 
  0.198971}, {0.877798, 0.2002}, {0.877144, 0.201012}, {0.875125, 
  0.203517}, {0.874244, 0.20461}, {0.872588, 0.206666}, {0.870728, 
  0.208975}, {0.85954, 0.222844}, {0.858024, 0.22472}, {0.856741, 
  0.226309}, {0.856082, 0.227124}, {0.854632, 0.228919}, {0.854179, 
  0.229481}, {0.853508, 0.230312}, {0.852947, 0.231006}, {0.850873, 
  0.233573}, {0.850394, 0.234167}, {0.849859, 0.234829}, {0.849518, 
  0.235251}, {0.848529, 0.236475}, {0.848125, 0.236976}, {0.847405, 
  0.237868}, {0.846652, 0.238799}, {0.842538, 0.243893}, {0.841747, 
  0.244872}, {0.840978, 0.245824}, {0.840541, 0.246365}, {0.839456, 
  0.247708}, {0.839077, 0.248178}, {0.838476, 0.248923}, {0.837931, 
  0.249597}, {0.835509, 0.252596}, {0.834834, 0.253432}, {0.834015, 
  0.254445}, {0.833455, 0.255139}, {0.831623, 0.257407}, {0.830771, 
  0.258462}, {0.829072, 0.260565}, {0.826996, 0.263135}, {0.811462, 
  0.282248}, {0.809645, 0.28447}, {0.808199, 0.286237}, {0.807487, 
  0.287107}, {0.805988, 0.288939}, {0.805537, 0.28949}, {0.804885, 
  0.290287}, {0.804354, 0.290936}, {0.802488, 0.293217}, {0.802078, 
  0.293719}, {0.801628, 0.294268}, {0.801346, 0.294613}, {0.800549, 
  0.295587}, {0.800231, 0.295976}, {0.799676, 0.296654}, {0.799112, 
  0.297344}, {0.796272, 0.300814}, {0.79577, 0.301428}, {0.795294, 
  0.30201}, {0.795028, 0.302335}, {0.794385, 0.303121}, {0.794166, 
  0.303389}, {0.793822, 0.303809}, {0.793517, 0.304182}, {0.792219, 
  0.305769}, {0.791873, 0.306191}, {0.791463, 0.306692}, {0.791188, 
  0.307028}, {0.790319, 0.308091}, {0.78993, 0.308566}, {0.78918, 
  0.309482}, {0.78831, 0.310546}, {0.782351, 0.31783}, {0.781432, 
  0.318952}, {0.780631, 0.319931}, {0.780212, 0.320444}, {0.779265, 
  0.321601}, {0.778962, 0.321971}, {0.778509, 0.322525}, {0.778124, 
  0.322995}, {0.77666, 0.324785}, {0.776311, 0.325211}, {0.775916, 
  0.325693}, {0.775663, 0.326003}, {0.774915, 0.326917}, {0.774604, 
  0.327297}, {0.774042, 0.327984}, {0.773445, 0.328714}, {0.76997, 
  0.33296}, {0.769259, 0.333829}, {0.768553, 0.334692}, {0.768145, 
  0.335191}, {0.767113, 0.336453}, {0.766744, 0.336903}, {0.766152, 
  0.337626}, {0.765608, 0.338292}, {0.763082, 0.341379}, {0.762346, 
  0.342278}, {0.761435, 0.343392}, {0.760798, 0.34417}, {0.75864, 
  0.346807}, {0.757595, 0.348084}, {0.755427, 0.350734}, {0.752615, 
  0.354172}, {0.726781, 0.384767}, {0.724059, 0.387889}, {0.721988, 
  0.390264}, {0.720998, 0.391399}, {0.718972, 0.393723}, {0.718378, 
  0.394404}, {0.717532, 0.395375}, {0.716852, 0.396155}, {0.714536, 
  0.398811}, {0.714041, 0.399379}, {0.713504, 0.399994}, {0.713171, 
  0.400376}, {0.71224, 0.401444}, {0.711873, 0.401864}, {0.711241, 
  0.402589}, {0.710606, 0.403317}, {0.707539, 0.406835}, {0.707017, 
  0.407434}, {0.706528, 0.407995}, {0.706257, 0.408305}, {0.705609, 
  0.409049}, {0.70539, 0.4093}, {0.705049, 0.409691}, {0.704749, 
  0.410035}, {0.703492, 0.411476}, {0.703164, 0.411853}, {0.702777, 
  0.412296}, {0.70252, 0.412591}, {0.701717, 0.413512}, {0.701362, 
  0.413919}, {0.700688, 0.414693}, {0.699918, 0.415576}, {0.695002, 
  0.421213}, {0.694296, 0.422023}, {0.69369, 0.422718}, {0.693377, 
  0.423077}, {0.692679, 0.423878}, {0.692458, 0.424131}, {0.69213, 
  0.424507}, {0.691855, 0.424823}, {0.690822, 0.426007}, {0.690581, 
  0.426284}, {0.690309, 0.426596}, {0.690136, 0.426794}, {0.689629, 
  0.427376}, {0.68942, 0.427615}, {0.689046, 0.428044}, {0.688653, 
  0.428495}, {0.686447, 0.431025}, {0.686012, 0.431524}, {0.685586, 
  0.432013}, {0.685342, 0.432292}, {0.684732, 0.432992}, {0.684517, 
  0.433238}, {0.684174, 0.433631}, {0.683862, 0.433989}, {0.682451, 
  0.435607}, {0.682052, 0.436066}, {0.681563, 0.436626}, {0.681226, 
  0.437013}, {0.68011, 0.438293}, {0.679583, 0.438897}, {0.678517, 
  0.440119}, {0.677188, 0.441644}, {0.666199, 0.454227}, {0.66478, 
  0.455849}, {0.663628, 0.457166}, {0.663054, 0.457823}, {0.661826, 
  0.459226}, {0.661453, 0.459654}, {0.660909, 0.460276}, {0.660461, 
  0.460787}, {0.658868, 0.462609}, {0.658512, 0.463016}, {0.658119, 
  0.463464}, {0.657873, 0.463746}, {0.657169, 0.464551}, {0.656887, 
  0.464874}, {0.656391, 0.465441}, {0.655883, 0.466022}, {0.653267, 
  0.469012}, {0.652794, 0.469553}, {0.652342, 0.47007}, {0.652089, 
  0.470359}, {0.651472, 0.471065}, {0.65126, 0.471307}, {0.650928, 
  0.471687}, {0.650631, 0.472026}, {0.649354, 0.473487}, {0.64901, 
  0.47388}, {0.6486, 0.474349}, {0.648323, 0.474665}, {0.647442, 
  0.475672}, {0.647044, 0.476127}, {0.646272, 0.47701}, {0.645365, 
  0.478048}, {0.638834, 0.485514}, {0.637776, 0.486724}}
]

\psset{dotsize=0.9pt,dotstyle=*}

\dataplot[plotstyle=dots,linecolor=blue]{\mydatab}

\savedata{\mydatab}[
{{0.05, 0.05}, {0.05, 0.}, {0.1, 0.05}, {0.15, 0.05}, {0.15, 0.1}, {0.2, 
  0.05}, {0.2, 0.15}, {0.25, 0.05}, {0.25, 0.1}, {0.25, 0.15}, {0.25, 
  0.2}, {0.3, 0.05}, {0.3, 0.25}, {0.35, 0.05}, {0.35, 0.1}, {0.35, 
  0.15}, {0.35, 0.2}, {0.35, 0.25}, {0.35, 0.3}, {0.4, 0.05}, {0.4, 
  0.15}, {0.4, 0.25}, {0.4, 0.35}, {0.45, 0.05}, {0.45, 0.1}, {0.45, 
  0.2}, {0.45, 0.25}, {0.45, 0.35}, {0.45, 0.4}, {0.5, 0.05}, {0.5, 
  0.15}, {0.5, 0.35}, {0.5, 0.45}, {0.55, 0.05}, {0.55, 0.1}, {0.55, 
  0.15}, {0.55, 0.2}, {0.55, 0.25}, {0.55, 0.3}, {0.55, 0.35}, {0.55, 
  0.4}, {0.55, 0.45}, {0.55, 0.5}, {0.6, 0.05}, {0.6, 0.25}, {0.6, 
  0.35}, {0.6, 0.55}, {0.65, 0.05}, {0.65, 0.1}, {0.65, 0.15}, {0.65, 
  0.2}, {0.65, 0.25}, {0.65, 0.3}, {0.65, 0.35}, {0.65, 0.4}, {0.65, 
  0.45}, {0.65, 0.5}, {0.65, 0.55}, {0.65, 0.6}, {0.7, 0.05}, {0.7, 
  0.15}, {0.7, 0.25}, {0.7, 0.45}, {0.7, 0.55}, {0.7, 0.65}, {0.75, 
  0.05}, {0.75, 0.1}, {0.75, 0.2}, {0.75, 0.35}, {0.75, 0.4}, {0.75, 
  0.55}, {0.75, 0.65}, {0.75, 0.7}, {0.8, 0.05}, {0.8, 0.15}, {0.8, 
  0.25}, {0.8, 0.35}, {0.8, 0.45}, {0.8, 0.55}, {0.8, 0.65}, {0.85, 
  0.05}, {0.85, 0.1}, {0.85, 0.15}, {0.85, 0.2}, {0.85, 0.25}, {0.85, 
  0.3}, {0.85, 0.35}, {0.85, 0.4}, {0.85, 0.45}, {0.85, 0.5}, {0.85, 
  0.55}, {0.85, 0.6}, {0.85, 0.65}, {0.85, 0.7}, {0.9, 0.05}, {0.9, 
  0.25}, {0.9, 0.35}, {0.9, 0.55}, {0.9, 0.65}, {0.95, 0.05}, {0.95, 
  0.1}, {0.95, 0.15}, {0.95, 0.2}, {0.95, 0.25}, {0.95, 0.3}, {0.95, 
  0.35}, {0.95, 0.4}, {0.95, 0.45}, {0.95, 0.5}, {0.95, 0.55}, {0.95, 
  0.6}, {0.95, 0.65}, {0.95, 0.7}, {1., 0.05}, {1., 0.15}, {1., 
  0.35}, {1., 0.45}, {1., 0.55}, {1., 0.65}, {1.05, 0.05}, {1.05, 
  0.1}, {1.05, 0.2}, {1.05, 0.25}, {1.05, 0.4}, {1.05, 0.5}, {1.05, 
  0.55}, {1.05, 0.65}, {1.1, 0.05}, {1.1, 0.15}, {1.1, 0.25}, {1.1, 
  0.35}, {1.1, 0.45}, {1.1, 0.65}, {1.15, 0.05}, {1.15, 0.1}, {1.15, 
  0.15}, {1.15, 0.2}, {1.15, 0.25}, {1.15, 0.3}, {1.15, 0.35}, {1.15, 
  0.4}, {1.15, 0.45}, {1.15, 0.5}, {1.15, 0.55}, {1.15, 0.6}, {1.15, 
  0.65}, {1.15, 0.7}, {1.2, 0.05}, {1.2, 0.25}, {1.2, 0.35}, {1.2, 
  0.55}, {1.2, 0.65}, {1.25, 0.05}, {1.25, 0.1}, {1.25, 0.15}, {1.25, 
  0.2}, {1.25, 0.3}, {1.25, 0.35}, {1.25, 0.4}, {1.25, 0.45}, {1.25, 
  0.55}, {1.25, 0.6}, {1.25, 0.65}, {1.25, 0.7}, {1.3, 0.05}, {1.3, 
  0.15}, {1.3, 0.25}, {1.3, 0.35}, {1.3, 0.45}, {1.3, 0.55}, {1.35, 
  0.05}, {1.35, 0.1}, {1.35, 0.2}, {1.35, 0.25}, {1.35, 0.35}, {1.35, 
  0.4}, {1.35, 0.5}, {1.35, 0.55}, {1.35, 0.65}, {1.35, 0.7}, {1.4, 
  0.05}, {1.4, 0.15}, {1.4, 0.25}, {1.4, 0.45}, {1.4, 0.55}, {1.4, 
  0.65}}
]

\psset{dotsize=1.2pt,dotstyle=*}

\dataplot[plotstyle=dots,linecolor=black]{\mydatab}

\rput(0.3,0.5){$20\mathcal W$}

\end{pspicture}
\caption{The wedge $\mathcal W$  for the Markov tree $(3,15,6)$. Here it has been dilated by a factor of $20$. More Markov numbers are included as the dilation factor increases. The right  side of the wedge is the  strictly convex lattice curve $\G_T$, going from slope $\approx -1.143$ at the top to slope  $\approx -1.242$ at the bottom.}
\label{wplot}
\end{figure}
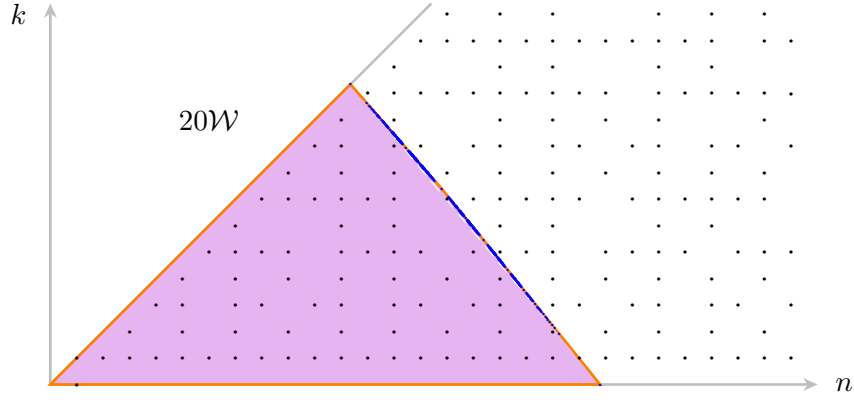

\begin{theorem} \la{zzz}
Let $T$ be a rising tree. Then
\begin{equation} \la{mzx}
  M_T(X) = \frac{6\text{\rm Area}(\mathcal W_T)}{\pi^2}  (\log X)^2 +O\left(\log X \log \log X \right).
\end{equation}
\end{theorem}
\begin{proof}

As we have seen, $M(X)$ is given by the number of lattice points inside dilates of the wedge $\mathcal W = \mathcal W_T$. Precisely, \e{av} implies $M(X)=N(\rho(X))$ for
\begin{equation} \la{mx}
  N(t) :=  \#\{ (n,k)\in \Z^2 \, : \, \gcd(n,k)=1, (n,k)\in t \mathcal W \}.
\end{equation}
Defining $N^*(t)$ as in \e{mx}, but with the $\gcd$ condition removed, we have
\begin{equation} \la{mxs}
  N^*(t)= \text{Area}(\mathcal W) \cdot t^2  +O(t).
\end{equation}
This follows from a classical 
result of Jarnik \cite{ste}: the absolute difference between the area inside a closed, rectifiable Jordan curve and the number of lattice points it contains is at most the length of the curve, (this length must be at least $1$). 

The domain $\mathcal W$ is star-shaped, meaning that if the point $(x, y)$ is in $\mathcal W$ then so is $(\lambda x, \lambda y)$ for all $\lambda \in [0,1]$. In this case there is a standard way, using M\"obius inversion, to adjust $N^*(t)$ to include the gcd condition. See  \cite[Lemma 4]{z82} for this in the example of a triangle, or \cite{mrz85} more generally. The result here is
 \begin{equation} \la{mxt}
  N(t)= (6/\pi^2)\text{Area}(\mathcal W) \cdot t^2  +O(t \log t).
\end{equation}
Noting that $\rho(X)=\log(X)+O(1/X^2)$ by \e{cru} completes the proof.
\end{proof}

\begin{lemma} \la{cfv}
As a polar curve, the set $\Gamma_T$ is strictly convex  for $D<4$ (bulges out from the origin) and strictly concave  for $D>4$ (bulges towards the origin).
\end{lemma}
\begin{proof}
We may argue geometrically, noting that $F$ in \e{f} sends triangles to triangles and is invertible.
The effect of $F$  is to map horizontal lines $y=c>0$ to vertical lines $x=1/c$, and to map vertical lines $x=c$ for $0\lqs c\lqs 1$ to lines $y=cx$ of slope $c$. It follows that the strictly convex graph of $\Phi$ for $D<4$, with its curve below any secant line, get mapped to a convex polar curve $\Gamma_T$, with any secant cutting off a piece of $\Gamma_T$ that is on the opposite side to the origin. The strictly concave graph of $\Phi$ for $D>4$, with its curve above any secant, get mapped to a concave polar $\Gamma_T$ with any secant cutting off a piece of it that is on the same side as the origin.
\end{proof}

Zagier's formula \cite[(3)]{z82} for the coefficient of the main term in \e{mzx} for the Markov numbers may be given in general, reinterpreted as a sum of triangle areas.

\begin{prop} \la{wedg}
For every rising tree $T$,
\begin{equation} \la{2w}
  2\text{Area}(\mathcal W_T) = \frac 1{\rho(m_0) \rho(m_1)} +\sum_{q \in \Q \cap (0,1)} \frac{\rho(m_{q_0}) + \rho(m_{q_0'}) - \rho(m_{q})}{\rho(m_{q_0})  \rho(m_{q_0'})  \rho(m_{q})}.
\end{equation}
\end{prop}
\begin{proof}
The main triangle in $\mathcal W$ has vertices at $(0,0)$, $(1/\rho(m_0),0)$ and $(1/\rho(m_1),1/\rho(m_1))$. In general, points on the side $\Gamma_T$ are
\begin{equation} \la{ve}
  \left(\frac 1{\phi(k/n)}, \frac{k/n}{\phi(k/n)}\right) = \left(\frac n{\rho(m_{k/n})}, \frac k{\rho(m_{k/n})}\right).
\end{equation}
Recall that for three points $(x_i,y_i)$ for $i=0,1,2 \bmod 3$, listed in a positive orientation, 
the Shoelace formula gives twice the area of the triangle they make as
\begin{equation} \la{lace}
   \sum_{i=0,1,2}(x_i y_{i+1}-x_{i+1}y_i).
\end{equation}

After the main triangle, the next has vertices \e{ve} for $k/n=0/1, 1/2$ and $1$, with $1/2$ the mediant of $0, 1$. Applying the Shoelace formula gives plus or minus its area. When $D<4$ these points correspond to points on the convex graph of $\Phi$, with this ordering giving a positive orientation. Since $F$ preserves this orientation,  the area value is positive. 
Continuing, the triangle corresponding to $q=k/n$ in the interval $(0,1)$ has vertices at
\begin{equation} \la{oo}
  \left(\frac {n_0}{\rho(m_{q_0})}, \frac {k_0}{\rho(m_{q_0})}\right), \quad \left(\frac n{\rho(m_{q})}, \frac k{\rho(m_{q})}\right), \quad 
  \left(\frac {n_0'}{\rho(m_{q_0'})}, \frac {k_0'}{\rho(m_{q_0'})}\right),
\end{equation}
where, in our usual notation, $q_0=k_0/n_0 < q_0'=k_0'/n_0'$ are the Farey neighbors with mediant $q$.
 By \e{lace} and \e{nb}, we therefore obtain \e{2w} when $D<4$. The series converges rapidly since the numerator in \e{2w} is $\ll e^{-\omega_T n}$ by \e{deb} and the denominator is $\gg n\cdot n_0 \cdot n_0'$ by \e{cru} and Proposition \ref{sim2}.
 
For $D>4$ the points \e{oo} are now listed with a negative orientation. This gives \e{2w} correctly as these areas must be subtracted from the main triangle to give the convex polar wedge from Lemma \ref{cfv}, as on the right of Figure \ref{wplot2ee}. When $D=4$ the numerators in the series are all zero and the wedge is just the main triangle.
\end{proof}

\SpecialCoor
\psset{griddots=5,subgriddiv=0,gridlabels=0pt}
\psset{xunit=1cm, yunit=1cm, runit=1cm}
\psset{linewidth=1pt}
\psset{dotsize=1.4pt,dotstyle=*}
\begin{figure}[!htb]
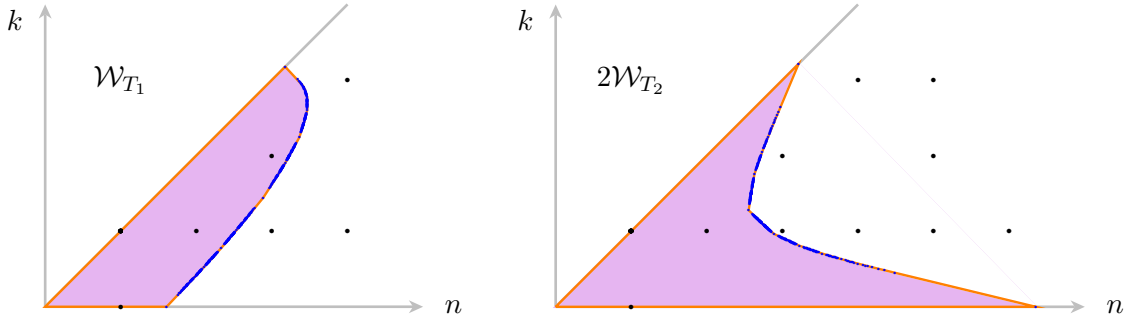

\centering

\psset{arrowscale=1.4,arrowinset=0.3,arrowlength=1.1}
\newrgbcolor{light}{0.8 0.8 1.0}
\newrgbcolor{pale}{1 0.7 1}
\newrgbcolor{pale}{1 0.7 0.4}

\newrgbcolor{markov}{0.902344 0.710938 0.945313}
\psset{linecolor=markov}

}

\end{pspicture}
\caption{Wedges for the rising trees $T_1=(2.4, 2.5, 2.1)$ and $T_2=(2.1, 5, 2.4)$ from Figure \ref{phplots}.}
\label{wplot2ee}
\end{figure}


The wedge $\mathcal W$  for the Markov tree $(3,15,6)$ in Figure \ref{wplot} has area $\approx 0.297267722$, while the wedges for $T_1$ and $T_2$ in Figure \ref{wplot2ee} have  approximate areas $3.9055$, $1.3626$, respectively.

\begin{proof}[Proof of Corollary \ref{zm}]
This follows from Theorem \ref{zzz} applied to the tree $(3,15,6)$. Note that  $M(\lambda X)$ and $M(X)$ have the same asymptotic expansion \e{mzx} for fixed $\lambda>0$, so multiplying the Markov numbers by $3$ has no effect. 
\end{proof}

We could also have used the rising trees $(3,6,3)$ or $(3,3,3)$ to obtain Corollary \ref{zm}. Their wedges have areas twice and $4$ times that of $(3,15,6)$, reflecting that numbers appear with multiplicity $2$ and $4$, respectively, on these trees.

For another example, take the tree $T=(3,25,16)$ with $D=-310$. Its numbers are 
\begin{equation} \la{16}
  3, 16, 25, 59, 152, 397, 397, 1039, 1472, 2720, 6327, 7121, \dots
\end{equation}
with $397$ appearing in two places on the tree, in $m_{1/5}$ and $m_{2/3}$, so that its multiplicity is $2$. (By the climbing lemma there can be no further instances of $397$.) The number of these numbers \e{16} less than $X$ is given by \e{mzx} with $\text{\rm Area}(\mathcal W_T) \approx 0.2221$.

It would be interesting to extend the results in this section to progressing trees $T$. However, their lattice curves have infinite length and their wedges have infinite area. By Proposition \ref{stx}, 
$M_T(X) \gg X$.

\section{Uniqueness} \la{niqu}

\subsection{The uniqueness conjecture}
For $X$ a real Markov topograph or a real Markov tree, we are interested in two related notions.  At each non-root vertex of $X$ the labels of the three regions that meet there form a triple. For each region, its label occurs in the infinitely many triples  along its border. Consider the set of all triples of $X$ with region labels listed in increasing order. 
We say $X$ has {\em uniqueness} if the maximal element in each of these triples 
always  determines the other two. In other words, if there are two increasing triples  $a\lqs b\lqs c$ and  $a'\lqs b'\lqs c$ then $a=a'$ and $b=b'$. The second notion is simpler to state: if the same label appears in two different regions of $X$ then we say $X$ has {\em repeated numbers}.

For example, the strictly rising tree with base $(3,5,3)$ has uniqueness among its labels that are $<10^{250}$, while having many repeated numbers by symmetry. The rising tree $(15,6,3)$ also seems to have uniqueness (see below) and does have repeated numbers since the descending path leads to $(3,3,3)$.

\begin{lemma} \la{uniq}
Let $T$ be a 
strictly rising tree or a progressing tree. Suppose its base triple $(s_0,u_0,t_0)$ has $s_0, u_0, t_0$ all distinct. Then $T$ has no repeated numbers exactly when it has uniqueness.
\end{lemma}
\begin{proof}
First note that every oriented triple $(s,u,t)$ in $T$ is ordered, with $u>st/2$. Also, all triples of $T$ have distinct components; this follows by induction since, moving away from the root, an oriented triple $(s,u,t)$ has $u>s,t$ with  $s \neq t$  known already.

Suppose $T$ does not have uniqueness, with different triples $a\lqs b\lqs c$ and  $a'\lqs b'\lqs c$. The first must correspond to either of  the ordered triples $(a,c,b)$ or $(b,c,a)$.  Similarly for the second. This means that the edge leading from region $c$ towards the root has different adjacent regions in each case. Hence $c$ is a repeated number.

Next suppose $c$ is a repeated number on $T$. It cannot be $s_0$, $t_0$ or $u_0$ since that would have $c$ appearing twice on a strictly increasing sequence of regions, according to the climbing lemma, Lemma \ref{climb}. Hence we must have two ordered triples $(a,c,b)$ and $(a',c,b')$. Assume that $(a',c,b')$ equals  $(a,c,b)$ or $(b,c,a)$; otherwise $T$ does not have uniqueness and we are done. 
Now consider the path of edges from each region $c$ to the root. In the first case, where $(a',c,b')=(a,c,b)$, following the local rule along the two paths shows they must look the same locally and have the same length. Suppose the paths meet first at a vertex $v$ and write $(s,u,t)$ for the ordered triple  surrounding $v$. Then we must have $(s,u,t)=(u,t,s)$ so that $s=u=t$, contradicting  distinctness. 
  The second case has $(a',c,b')=(b,c,a)$ and the paths are now mirror images of each other. This implies that $(s,u,t)=(t,u,s)$ where they meet, and so $s=t$ which again contradicts   distinctness. Therefore the repeated number $c$ implies that $T$ does not have uniqueness.
\end{proof}

The Markov uniqueness conjecture was first raised as a question by Frobenius in  1913,  \cite[p. 461]{fro}. The conjecture states  that the original Markov tree $(3,15,6)$, shown in Figure \ref{mart}, has uniqueness. By Lemma \ref{uniq} this is equivalent to it  having no repeated numbers. Though progress has been made, as described in \cite[Chap. 10]{aig} for example, this conjecture  remains open. 

We saw in \e{16} that strictly rising trees can have nontrivial repeated numbers and exceptions to  uniqueness in general. These exceptions seem to be rare. For example, computer searches have not located any other repeated numbers in the tree $(3,25,16)$.  Trees that do have uniqueness are found in Section \ref{d4}.

\subsection{Aigner's conjectures}
Following Aigner, we next examine rising tree numbers $m_{k/n}$ along `index lines' as described before Corollary \ref{zak}. Given a line $\ell$, let $(n_i,k_i) \in \Z^2$ be all the points on it satisfying $0\lqs k_i\lqs n_i$ and $\gcd(n_i,k_i)=1$. If we list the points in order, say with increasing $n_i$ or $k_i$, then how are the numbers $m_{k_i/n_i}$ ordered?

Our work in Section \ref{num} can be used to provide a clear answer. Consider how $\Gamma_T$ dilated by increasing $t>0$ intersects $\ell$. If $D<4$ for example, then by Lemma \ref{cfv} the polar curve $\Gamma_T$ is strictly concave, bulging outward from the origin. Then, if they intersect at all, $t \cdot \Gamma_T$ initially meets $\ell$ at one point and, as $t$ increases, further intersects the line in one or two places. By Lemma \ref{latt}, whenever $t \cdot \Gamma_T$ meets a point $(n_i,k_i)$, then that implies  $m_{k_i/n_i} = 2\cosh(t)$. So as $t$ increases, $t \cdot \Gamma_T$ meets lattice points corresponding to tree numbers that are getting bigger.

To make this discussion precise requires some notation. Define $I_T \subset \R \cup\{\pm\infty\}$ to be the set of all   slopes between pairs of distinct points of $\Gamma_T$. We may call this set the {\em slope interval} of the rising tree $T$. Define a sequence to be {\em monotone} if it is strictly increasing or decreasing. Empty and singleton sequences  are trivially monotone. 
Call a sequence {\em modal} if it satisfies
\begin{equation} \la{pui}
  m_1< m_2 < \dots < m_r \lqs m_{r+1} > \dots > m_{t-1} > m_{t},
\end{equation}
  and {\em  antimodal} for \e{pui} with the inequalities reversed.
In this
 non-standard language we have  one or two maximal (or minimal) numbers in the middle.
Monotone sequences are trivially modal and antimodal.

The proof of the next theorem has a similar method to that used in  \cite[Sect. 4]{ga22}.

\begin{theorem} \la{ag}
Let $T$ be a rising tree with slope interval $I_T$. Let $\ell$ be an index line of slope $s$, with $s$ possibly infinite. If $s \not\in I_T$ then  tree numbers are monotone along the line. If $s \in I_T$ then the numbers are antimodal along the line if $D<4$,  constant if $D=4$, and modal if $D>4$. For each $s \in I_T \cap (\Q \cup \{\pm \infty\})$ there exist index lines of that slope that are not monotone.
\end{theorem}
\begin{proof}
If $s \not\in I_T$ then $\ell$ and $t \cdot \Gamma_T$ can intersect at most once. If they do meet, then this point of intersection moves continuously and in one direction along $\ell$ with $t$. Hence Lemma \ref{latt} implies that the tree numbers are monotone along this line.

Next suppose  $s \in I_T$. There exists a line with this slope intersecting $\Gamma_T$ in two places. Assume $D\neq 4$. By continuity and strict concavity/convexity there is a parallel line meeting $\Gamma_T$ at the single `tangent' point $z_s$. Let $L$ be the line from the origin to $z_s$, having slope strictly between $0$ and $1$. Now for any line $\ell$ with the slope $s$, if there is a dilate of $\Gamma_T$ that ever meets $\ell$ then there is a $t_0>0$ so that $t_0 \cdot \Gamma_T$ meets $\ell$ at exactly one point: $\ell \cap L$. If $D<4$ then by Lemma \ref{cfv}, for increasing $t>t_0$, $\ell \cap t \cdot \Gamma_T$ consists of two points moving away from $\ell \cap L$ in each direction. If $D>4$ then the two intersection points move away from $\ell \cap L$  for decreasing $t<t_0$. Lemma \ref{latt} then informs us that the tree numbers must be antimodal if $D<4$ and modal if $D>4$. If $D=4$ then $\Gamma_T$ is a line, $I_T =\{s\}$ for $s$ the slope of $\Gamma_T$, and tree numbers on index lines with slope $s$ are constant by Lemma  \ref{latt}.

We next  show that the situation described in the last paragraph always has nontrivial examples for rational or infinite $s\in I_T$. 
Suppose  $s=u/v \in \Q$ with $v>0$. Choose a large $n^* \equiv 1 \bmod v$ and let $k^*\in \Z$ give a close lattice point $(n^*,k^*)$ to the line $L$. Let $\ell$ be the line through $(n^*,k^*)$ of slope $u/v$ and containing the lattice points $(n^*+v j,k^*+u j)$ for $j \in \Z$. Such points satisfying the two conditions 
$$
(a) \quad 0\lqs k^*+u j \lqs n^*+v j  \quad \qquad \text{and}  \quad \qquad (b) \quad \gcd(n^*+v j,k^*+u j)=1,
$$
 correspond to tree numbers. Finding at least two of these on each side of $L$  will show the numbers are not monotone. Condition (a) can be simplified by choosing $\delta, \delta'>0$ and independent of $n^*$ so that  $1\lqs j \lqs \delta n^*$ and $1\lqs -j \lqs \delta' n^*$ imply it. If we can find $n^*+v j$ that are prime numbers, then that will give  condition (b). They are supplied by the classical theorem of de La Vall\'ee Poussin. To state it, fix integers $d, a$ with $d\gqs 1$ and $\gcd(a,d)=1$. Then $\pi(x;d,a)$, the number of primes $\lqs x$ and $\equiv a \bmod d$, satisfies
 \begin{equation*}
   \pi(x;d,a) \sim \frac {x}{\phi_{\rm Euler}(d) \log x} \qquad \text{as} \qquad x \to \infty.
 \end{equation*}
Hence the number of primes that are $\equiv 1 \bmod v$ in $[n^*, n^*(1+ \delta)v]$ and in $[n^*(1- \delta')v, n^*]$  can be made as large as we please by  choosing $n^*$ large enough. This demonstrates that the tree numbers on this line are not monotone. If $\ell$ is vertical then the arguments are easier; just choose $n^*$ large and prime.
\end{proof}

Our modal (and antimodal) sequences allow for two equal maximal (minimal) numbers in the middle. This could occur if $T$ has repeated numbers. By strict convexity/concavity there cannot be  more than two repeated numbers on an index line. Two minimal numbers should also be allowed in \cite[Rem. 1.3]{ga22}.


To determine $I_T$, recall that the right derivative at $0$ and the left derivative at $1$ of $\Phi_T$
are given by \e{pu}, \e{pu2}, as
\begin{align}
  S'(0)   & =  -2\rho(s_0)    + \log f(u_0,s_0,t_0) 
   = \log\left( \frac{\lambda_{s_0} u_0 -t_0}{\lambda^2_{s_0} \sqrt{s_0^2 -4}}\right), \la{sg}\\
   S(1)   &  = 2\rho(t_0)    - \log f(u_0,t_0,s_0) 
     = -\log\left( \frac{\lambda_{t_0} u_0 -s_0}{\lambda^2_{t_0} \sqrt{t_0^2 -4}}\right). \la{sgb}
\end{align}
As a check on our work, when $D=4$  we know by Lemma \ref{tr4} that $T$ is strictly rising with $(s_0,u_0,t_0)$  an ordered triple. Hence  $\lambda_{u_0} = \lambda_{s_0} \lambda_{t_0}$ by Proposition \ref{zag2}, and therefore with \e{use},
\begin{equation}\label{u0}
  u_0 = \lambda_{u_0} + \lambda_{u_0}^{-1} = \lambda_{s_0} \lambda_{t_0} + \lambda_{s_0}^{-1} \lambda_{t_0}^{-1}.
\end{equation}
Then for $S'(0)$,
\begin{equation*}
  \frac{\lambda_{s_0} u_0 -t_0}{\lambda^2_{s_0} \sqrt{s_0^2 -4}} = \frac{\lambda_{s_0} (\lambda_{s_0} \lambda_{t_0} + \lambda_{s_0}^{-1} \lambda_{t_0}^{-1}) -(\lambda_{t_0}+\lambda_{t_0}^{-1})}{\lambda^2_{s_0} (\lambda_{s_0} - \lambda_{s_0}^{-1})} =  \frac{\lambda_{t_0}}{\lambda_{s_0}}.
\end{equation*}
Similarly for $S(1)$ and we obtain, as expected,
\begin{equation*}
  S'(0) = S(1) = \log(\lambda_{t_0}/\lambda_{s_0}) = \rho(t_0)-\rho(s_0) \qquad \text{for} \qquad D=4.
\end{equation*}

Next set
\begin{equation} \la{psi}
  \sigma(0) := -\frac{\rho(s_0)}{S'(0)},  \qquad \sigma(1) := 1-\frac{\rho(t_0)}{S(1)}.
\end{equation}

\begin{prop} \la{tan}
Let $T$ be a rising tree with base triple $(s_0,u_0,t_0)$. Define $\sigma(0)$, $\sigma(1)$ be means of \e{psi} and \e{sg}, \e{sgb}. Then $\Gamma_T$, the image of the graph of $\Phi_T$ under $F$,  has  tangents with slope $\sigma(0)$ where it meets the line $y=0$ and slope $\sigma(1)$ where it meets the line $y=x$.
\end{prop}
\begin{proof}
The transformation $F$ from \e{f} maps the slopes $S'(0)$, $S(1)$ to $\sigma(0)$ and $\sigma(1)$, respectively.
\end{proof}


The slope interval $I_T$ may then be expressed in terms of $\sigma(0)$ and $\sigma(1)$. This depends on how $\Gamma_T$ curves and is also complicated by its slopes possibly becoming infinite.  If $D=4$ then $\sigma(0) = \sigma(1)$ and $I_T=\{\sigma(0)\}$.

\begin{ex}
{\rm For the Markov tree generated by $(s_0,u_0,t_0)=(3,15,6)$ we obtain \e{sigs} from \e{sg}, \e{sgb} and \e{psi}. Then the slope interval  is the open interval $(\sigma(0),\sigma(1)) \approx (-1.242, -1.143)$; see Figure \ref{wplot}. Corollary \ref{zak} now follows from Theorem \ref{ag}, with the monotone sequences of Markov numbers easily determined to be increasing as described. 
This shows that any repeated Markov numbers must have indices on index lines in this narrow slope range.}
\end{ex}

\begin{ex}
{\rm For the rising tree $T_1=(2.4, 2.5, 2.1)$ we find $\sigma(0)\approx 1.058$, and $\sigma(1)\approx -0.981$. Then its slope interval is $I_{T_1} = (\sigma(0),\infty]\cup [-\infty,\sigma(1))$ as can be seen in Figure \ref{wplot2ee}.
Hence $m_{k/n}$ and $m_{k'/n'}$ are distinct for $T_1$ if $\sigma(1) \lqs (k-k')/(n-n') \lqs \sigma(0)$.}
\end{ex}

\section{Approximations to $\Phi$} \la{seri}

Let $\mathcal R_{\ell}$ be the increasing sequence that includes the Farey rows up to row $\ell$, as in Section \ref{rey}, with
\begin{equation} \la{bl}
  \mathcal R_{\ell} =\{b_0=0, b_1=\tfrac 1{\ell+1}, \dots, b_i, \dots, b_{2^\ell-1}=\tfrac {\ell}{\ell+1}, b_{2^\ell}=1\}.
\end{equation}
 For example, Figure \ref{farey} displays $\mathcal R_5$, ordered by angle.
For a rising tree with encoding function $\Phi$, define $\Phi(\ell,x)$ to be the piecewise linear function connecting the $2^\ell+1$ points $(b_i,\Phi(b_i))$. For increasing $\ell$  it converges uniformly to $\Phi(x)$ as shown next. Recall the constants $\omega_T$, $\kappa_T$ from Propositions \ref{sim2}, \ref{zag2}, respectively. They may be found explicitly with $\omega_T=0.4$, $\kappa_T =52$ for $T=(3,15,6)$, for example.

\begin{prop} \la{pah}
For a rising tree $T$ and $\ell \gqs 1$ 
we have
\begin{equation} \la{fh}
  \left| \Phi(x) - \Phi(\ell,x)\right| \lqs \kappa_T (\ell+1) e^{-2 \omega_T(\ell+1)} \qquad \text{for} \qquad x \in [0,1].
\end{equation}
\end{prop}
\begin{proof}
The result is trivially true for $D=4$ by Theorem \ref{mthm}. Assume $D<4$, making the graph of $\Phi$ strictly convex. Let $q=b_i$ for $i$ odd, so that $q$ is in Farey row $\ell$ and has denominator at least $\ell+1$. Let $m=\g(b_{i-1},q)$ and $m'=\g(q,b_{i+1})$ be the slopes of the segments of  $\Phi(\ell,x)$ for $b_{i-1} \lqs x \lqs b_{i+1}$. Let $S$ and $S'$ be the left and right derivatives of $\Phi$ at $q$. By \e{<}, Corollary \ref{sldi} and Proposition \ref{sim2} we know
\begin{equation*}
  0< S'-S, \qquad m'-m < \kappa_T (\ell+1) e^{-2 \omega_T(\ell+1)}, \qquad m<S.
\end{equation*}
Adding these inequalities implies $m'-S'<\kappa_T (\ell+1) e^{-2 \omega_T(\ell+1)}$ which will be needed for \e{boz}. Now
\begin{equation*}
  S'(x-q)+\Phi(q) \lqs \Phi(x) \lqs m'(x-q)+\Phi(q)=\Phi(\ell,x) \qquad \text{for} \qquad q\lqs x \lqs b_{i+1},
\end{equation*}
by convexity. Therefore, for these $x$ values,
\begin{align}
 0\lqs \Phi(\ell,x) -  \Phi(x)  & \lqs m'(x-q)-\Phi(q)-(S'(x-q)+\Phi(q)) \notag\\
   & \lqs (m'-S')(x-q) < (m'-S'). \la{boz}
\end{align}
Similarly on the left of $q$ and we obtain \e{fh}. The case $D>4$ is handled the same way.
\end{proof}

Proposition \ref{pah} may be  used to find values of $\Phi$  at irrationals to within a prescribed error.
We give another more explicit method for this next that also applies to $\Phi'$. 
First, a simple example gives the idea.
For $q_0$ and $q_0'$ the Farey neighbors with mediant $q$, we have
\begin{equation} \la{ppd}
  \rho(m_q) = \rho(m_{q_0})+\rho(m_{q_0'})+\Delta(q),
\end{equation}
in the notation of \e{de}, and as we move up the Farey tree $\Delta(q)$ becomes exponentially small as in \e{deb}.
The identity  \e{ppd} may  be used repeatedly to express $\rho(m_q)$ ultimately in terms of $\rho(m_0)$ and $\rho(m_1)$. To see this for $\phi(3/7)$, 
start with
\begin{equation} \la{del}
\begin{aligned}
  \rho(\tfrac 12) & = \rho(0)+\rho(1) +  \Delta(\tfrac 12), \\
  \rho(\tfrac 13) & = \rho(0)+\rho(\tfrac 12) +  \Delta(\tfrac 13) = 2\rho(0)+\rho(1) +  \Delta(\tfrac 12) +  \Delta(\tfrac 13),
\end{aligned}
\end{equation}
using the shorthand $\rho(q)$ for $\rho(m_q)$. Following the sequence $1/2, 1/3, 2/5, 3/7$ on the Farey tree:
\begin{equation} \la{delb}
 \begin{aligned}
  \rho(\tfrac 25) & = 3\rho(0)+2\rho(1) +  2\Delta(\tfrac 12)+  \Delta(\tfrac 13)+  \Delta(\tfrac 25), \\
  \rho(\tfrac 37) & = 4\rho(0)+3\rho(1) +  3\Delta(\tfrac 12)+  \Delta(\tfrac 13)+  \Delta(\tfrac 25)+  \Delta(\tfrac 37).
\end{aligned}
\end{equation}
Hence
\begin{equation}\label{512}
  \phi(\tfrac 3{7})  = \tfrac 47 \rho(0)+ \tfrac 37 \rho(1) +   \tfrac 37 \Delta(\tfrac 1{2})+   \tfrac 17\Delta(\tfrac 1{3})+   \tfrac 17\Delta(\tfrac 2{5})+   \tfrac 1{7}\Delta(\tfrac 3{7}).
\end{equation}

Some notation is required for the general case.
Let $C_j=u_j/v_j=[a_0, a_1,  \dots, a_j]$ be the sequence of convergents to some $\alpha =[a_0, a_1, a_2, \dots ]\in \R \cap (0,1)$. For $k \gqs 1$ set 
\begin{equation} \la{yuk}
  C_{j,k}:=[a_0,a_1, \dots, a_j, k] =\frac{k u_j+u_{j-1}}{k v_j+v_{j-1}},
\end{equation}
 and we will later need the remainders $r_j =[a_j,a_{j+1}, \dots]$. 
 Note that any discussion of $C_j$ or $C_{j,k}$ assumes that all the coefficients $a_1, \dots, a_j$ are  non-zero.
 Define the linear function of $x$
\begin{equation} \la{cjk}
  \Lambda(x,C_{n, k}):= (1-x)\rho(m_0)+ x \cdot \rho(m_1) +\sum_{j=0}^n \sideset{}{'}\sum_{\ell=1}^{a_{j+1}}
  (-1)^j (v_j x-u_j) \Delta(C_{j,\ell})
\end{equation}
where the prime signifies omitting  the $\ell =1$ term when $j=0$ and summing over $1\lqs \ell \lqs k$ when $j=n$. It can be seen that the fractions $C_{j,\ell}$ in the sum start at $1/2$ and move along the Farey tree towards $\alpha$ one edge at a time, going left for $j$ even and right for $j$ odd. Each $C_{j,\ell}$ is the mediant of  $C_{j}$ and $C_{j,{\ell-1}}$, (or $C_{j}$ and $C_{j-1}$ if $\ell=1$).

\begin{prop} \la{awok}
For a rising tree and $n, k\gqs 1$, the functions $\Phi(x)$ and $\Lambda(x,C_{n,k})$ agree at $x=C_n$ and  $x=C_{n,k}$.
\end{prop}
\begin{proof}
Write the Farey neighbors $C_n$ and $C_{n,k}$ as $b<b'$ with $b=u/v$ and $b'=u'/v'$.
We claim that the right side of \e{cjk} reduces to
\begin{equation} \la{fx}
  f(x) := -(v' x-u')\rho(m_b) + (v x-u)\rho(m_{b'}).
\end{equation}
If this is true then
\begin{align*}
f(b) & = f(u/v) = -(v' u/v -u')\rho(m_b) = \rho(m_b)/v = \Phi(b),\\
f(b') & = f(u'/v') = (v u'/v'-u)\rho(m_{b'}) = \rho(m_{b'})/v' = \Phi(b'),
\end{align*}
as we wanted. The claim may be established by using induction on the length of the sum. Let $q$ be the mediant of $b$ and $b'$ and assume the sum so far equals \e{fx}, giving the desired linear function on the interval $[b,b']$. To go to the interval $[b,q]$ use
\begin{align*}
  f(x) + (v x-u)\Delta(q) & = f(x) + (v x-u)(\rho(m_q) -\rho(m_b)-\rho(m_{b'})), \\
   &  = -((v+v') x-(u+u'))\rho(m_b) + (v x-u)\rho(m_{q}).
\end{align*}
This has the correct form \e{fx}, and will agree with $\Phi(x)$ at $x=b$ and $x=q$.  To go to the interval $[q,b']$ use
$
  f(x) -(v' x-u')\Delta(q)$.
\end{proof}

For a convergent $C_{j+1}$, set
\begin{equation} \la{el}
  \Delta^*(C_{j+1})  :=  \sum_{\ell=1}^{a_{j+1}}
   \Delta(C_{j,\ell})
    = \rho(m_{C_{j+1}}) - a_{j+1}\rho(m_{C_{j}}) - \rho(m_{C_{j-1}}).
\end{equation}

\begin{theorem} \la{rat}
For a rising tree and an irrational $\alpha =[0, a_1, a_2, \dots ]$ with convergents $C_j=u_j/v_j$  we have
\begin{align}\la{iy}
  \Phi(\alpha) & = (u_1-v_1 \alpha)\rho(m_0) + \alpha \cdot \rho(m_{C_1}) + \sum_{j \gqs 1} 
    (-1)^j(v_j \alpha -u_j)\Delta^*(C_{j+1}),\\
  \Phi'(\alpha) & = \rho(m_{C_1}) - a_1\rho(m_0) +\sum_{j \gqs 1} 
  (-1)^j v_j  \Delta^*(C_{j+1}). \la{iz}
\end{align}
Restricting each series in \e{iy}, \e{iz} to $1\lqs j\lqs m$  makes an error of size $O(e^{-2\omega_T v_m})$.
\end{theorem}
\begin{proof}
Letting $k=a_{n+1}$ in Proposition \ref{awok} has   $\Lambda(x,C_{n+1}) = \Phi(x)$  when $x$ equals the convergents $C_n$ and  $C_{n+1}$. Then as  $n\to \infty$,
\begin{equation*}
  \Lambda(\alpha,C_{n+1}) \to \Phi(\alpha), \qquad \left.\frac{d}{dx}\right|_{x=\alpha} \Lambda(x,C_{n+1})  \to \Phi'(\alpha),
\end{equation*}
since $\Phi(x)$ is continuous and differentiable at $\alpha$ by Theorem \ref{mthm}. Using \e{el} in \e{cjk} and simplifying with
\begin{equation*}
  \sum_{\ell=2}^{a_{1}}  \Delta(C_{0,\ell}) = (1-a_1)\rho(m_0)-\rho(m_1)+\rho(m_{C_1})
\end{equation*}
gives \e{iy}, \e{iz}. To bound the tails of these series, we have by \e{deb},
\begin{multline*}
  \left|\Delta^*(C_{j+1})\right|  <  \sum_{\ell=1}^{\infty} | \Delta(C_{j,\ell})| \lqs \kappa_T\sum_{\ell=1}^{\infty}  e^{-2 \omega_T (\ell v_j+v_{j-1})}\\
  = \kappa_T e^{-2 \omega_T v_{j-1}} \frac{e^{-2 \omega_T  v_j}}{1-e^{-2 \omega_T  v_j}} \lqs 
  \frac{ \kappa_T}{1-e^{-2 \omega_T}} e^{-2 \omega_T (v_j+v_{j-1})},
\end{multline*}
as the denominator of $C_{j,\ell}$ is $\ell v_j+v_{j-1}$ according to \e{yuk}. The tails in \e{iy}, \e{iz} are now bounded by
\begin{equation*}
  \frac{ \kappa_T}{1-e^{-2 \omega_T}} \sum_{j > m} v_j e^{-2 \omega_T v_j} \cdot e^{-2 \omega_T v_{j-1}}.
\end{equation*}
By calculus, $x e^{-2 \omega_T x} \lqs e/(2 \omega_T)$, and using $v_{m+j}-v_m \gqs j$ also finds
\begin{equation*}
  \sum_{j \gqs m}  e^{-2 \omega_T v_{j}} =  e^{-2 \omega_T v_{m}} \sum_{j \gqs 0}  e^{-2 \omega_T (v_{m+j}-v_m)} \lqs 
   \frac{e^{-2 \omega_T v_{m}} }{1-e^{-2 \omega_T}}.
\end{equation*}
Altogether, an explicit bound for the sum of terms with $j>m$ in \e{iy}, \e{iz} is
\begin{equation} \la{exp}
  \frac{ e \cdot \kappa_T}{2 \omega_T(1-e^{-2 \omega_T})^2}  e^{-2 \omega_T v_{m}}. 
\end{equation}
\end{proof}
The Wallis recurrence $v_{n+1}=a_{n+1}v_n+v_{n-1}$ implies $v_{n+1} \gqs v_n+v_{n-1}$ so that $v_n $ is greater than the $n$th Fibonacci number. Hence $v_n$ has exponential growth and \e{exp} implies that, for the finite sum approximations to \e{iy} and \e{iz},  the number of correct decimal places grows exponentially for every term included.
For example, applying   Theorem \ref{rat} to the Markov tree $T=(3,15,6)$ at $\alpha = 2^{-1/3} \approx 0.794$, we obtain
\begin{align}
  \Phi(2^{-1/3})  & = 1.59306957359666129259760487635162761800212750536181, \la{15g}\\
  \Phi'(2^{-1/3}) & = 0.82247969498758584376728693112908805189093979679153, \la{16g}
\end{align}
correct to the accuracy shown,  using only the terms  $1\lqs j \lqs 4$. By comparison, $\phi(C_5)=\rho(m_{C_5})/v_5 = 1.59341\ldots$ gives a much cruder approximation to $\Phi(2^{-1/3})$. Using terms with $1\lqs j \lqs 10$ to compute \e{15g}, \e{16g} makes the error $<10^{-1909}$ by \e{exp} since $v_{10}=5429$.

The factor $(-1)^j(v_j \alpha -u_j)$ appearing in \e{iy} can also be expressed in terms of the remainders $r_j=[a_j,a_{j+1}, \dots]$, showing that it is positive, less than $1$ and decreasing with $j$:
\begin{lemma}
For $j\gqs 0$ and irrational $\alpha \in (0,1)$ we have $(-1)^j(v_j \alpha -u_j) = 1/(r_1 r_2 \cdots r_{j+1})$.
\end{lemma}
\begin{proof}
Let $f_j(x) := (-1)^j(v_j x -u_j)$ and note the useful identity $r_j=a_j+1/r_{j+1}$. The lemma is true in the base cases $j=0, 1$. By the Wallis recurrences for $u_j$ and $v_j$ we have $f_j(x)=-a_j f_{j-1}(x)+f_{j-2}(x)$. Combine this with the above $r_j$ identity to give the induction.
\end{proof}

We obtain the neat alternate expression for \e{iy}:
\begin{equation}\label{alt}
  \Phi(\alpha)  = \frac{\rho(m_0)}{r_1  r_{2}} + \alpha \cdot \rho(m_{C_1}) + \sum_{j \gqs 2} 
  \frac{\Delta^*(C_{j})}{r_1 \cdots r_{j}}.
\end{equation}

\section{Markov topographs and identities} \la{tops}

\subsection{Bowditch's trees of Markov triples} \la{bowd}
McShane's influential 1991 identity concerns the lengths of simple closed geodesics on a once-punctured torus $\mathbb T$, and has inspired a great deal of research; see the introductions to \cite{twz08,bt16} for example. 
Bowditch   gave a simpler proof  after first reducing it to a more combinatorial statement. 
The reduction is contained in the following theorem and, as explained in \cite[Sects. 1, 2]{bow96}, involves representations of the fundamental group $\pi_1(\mathbb T)$, which is free on two generators, into $\PSL(2,\R)$, 

\begin{theorem} \la{bowl} \cite{bow96} Let $L$ be the multiset of lengths of  simple
closed geodesics on a once-punctured torus $\mathbb T$, (with a complete finite-area hyperbolic
structure). Let $L'$ be the corresponding multiset $\{2\cosh(\ell/2) \ : \ \ell \in L\}$. Then $L'$ gives the region labels of a real Markov topograph $\mathcal G$ with regions $>2$ and discriminator $D=0$. The labels of any such topograph $\mathcal G$ arise in this way.  
\end{theorem}


\begin{theorem}[McShane's identity, topograph version]  
Let $\mathcal G$ be a Markov topograph with discriminator $D=0$ and all regions  $>2$. Then
\begin{equation} \la{1/2}
  \sum_{m \in \mathcal G} \frac 1{e^{2 \rho(m)}+1} = \frac 12,
\end{equation}
where the sum is over all regions $m$ of $\mathcal G$ and the numbers $2 \rho(m)$ correspond to geodesic lengths.
\end{theorem}

We prove identites in Section \ref{tmc} that are valid for all $D$ and reduce to  \e{1/2} when $D=0$. This requires some understanding of the structure of topographs with general discriminators. 

\begin{prop} \la{real}
Let $\mathcal G$ be a Markov topograph  with all regions $\gqs 2$. Then its discriminator satisfies $D\lqs 4$. 
It has an edge  or  vertex
 well if and only if $D<4$.
\end{prop}
\begin{proof}
If $D=4$ then  $\mathcal G$ does not have a well by Lemma \ref{wells}.
Also note that if $\mathcal G$ has a $2$-region  then, since its neighboring regions form an arithmetic progression, they must all  equal some $r \gqs 2$. In that case $D=4$. So we may assume that $D \neq 4$ and all regions are $>2$. If we can show that $\mathcal G$ has a well in this case then we are done, since that will also imply that $D<4$ by Lemma \ref{wells} again.

Give all edges in this $\mathcal G$ their positive edge direction. By the climbing lemma, if a vertex has an edge directed in to it, or an incident undirected edge, then the other two edges must be directed outwards. Hence any vertex  is either (i) a vertex well with all edges directed out, (ii) an endpoint of an edge well with two edges directed out and one edge undirected, or else (iii) has one edge directed in and two directed out. 

Next,  to obtain a contradiction, assume $\mathcal G$  does not contain a well and all vertices are of type (iii). Start at any region $C>2$ and descend for $\ell$ edges, going against the edge arrows,  to a triple $(s,u,t)$ where the edge between $s$ and $t$ is positively directed towards $u$. 
Treat $(s,u,t)$ as the base triple of a strictly rising tree $T$, and suppose  region $C$ has Farey index $k/n$ in relation to this base. Then $n>\ell$ and, by Lemma \ref{s},
\begin{align*}
  2 e^{\omega_T n} \lqs m(T)_{k/n} & \implies  e^{\omega_T \ell} < C/2\\
    & \implies \min(s,t)/2 < (C/2)^{1/\ell}.
\end{align*}
Taking $s$ to be the minimum, it is forced to be arbitrarily close to $2$  by choosing $\ell$ large enough. Going further along the border of region $s$, the descending path must eventually leave this region. This is because the border regions start increasing again by Proposition \ref{start}. At the vertex where the path leaves $s$ is a triple $(r,s,w)$, say, with $2< r, w< s$. Therefore $D = r^2+s^2+w^2 -r s w$ is arbitrarily close to $4$, giving a contradiction. We have shown that topographs with   $D \neq 4$ and all regions  $>2$ must have a well, as desired.
\end{proof}

\begin{adef} \la{wellb}
{\rm   A {\em negative vertex well} is a vertex in a topograph having all three neighboring regions  $\lqs -2$.
}
\end{adef}

Recall the definition of sign variants after Lemma \ref{wells}.

\begin{lemma} \la{neg}
 Let $\mathcal G$ be a Markov topograph   with  labels in $(-\infty,-2] \cup [2,\infty)$. Suppose  it has a negative vertex well $v$. Then its discriminator satisfies $D\gqs 20$. Also, sign variants of each of the three trees with root $v$ are strictly rising or progressing. Hence $v$ is the unique vertex in $\mathcal G$ with an odd number of negative neighboring regions.
\end{lemma}
\begin{proof}
Let $v$ be surrounded by the triple $(r, s, t)$ with $r,s,t \lqs -2$. Clearly \e{suwx} implies $D\gqs 20$. Let $u$ be the region on the other side of $s, t$ as in the middle of Figure \ref{root}. Then $u\gqs 6$ and the edge from $r$ to $u$ must have a  label $u-r \gqs 8$. Consider the sign variant  of $\mathcal G$  containing $(-s,u, -t)$. The climbing lemma shows that the tree $T$ with this base has increasing regions and so must be strictly rising or progressing. In particular, all regions of $T$ are positive. The same can be done for the  other two trees with root $v$, requiring the other two sign variants.
\end{proof}

\begin{lemma} \la{neg2}
 Let $\mathcal G$ be a Markov topograph  of discriminator $D$ with  labels in $(-\infty,-2] \cup [2,\infty)$. Suppose  it contains at least one negative region, and is not a sign variant of a  topograph  with all regions positive.    Then $\mathcal G$, or a sign variant of it, contains a negative vertex well.
\end{lemma}
\begin{proof}
 Suppose, to the contrary, that all triples have zero or two  negative labels. If all triples in fact have two  negative labels then, as seen in Figure \ref{topv}, the positive regions must all be the same color. However this is proscribed since $\mathcal G$ is not a sign variant of a  topograph  with all regions positive. Hence there must be a vertex $v_1$ with all neighboring regions positive. A path from $v_1$ to a vertex $v_2$ with two negative neighboring regions must pass a vertex with exactly one negative. This contradiction proves the lemma's claim.
\end{proof}

\begin{cor} \la{cor}
 Let $\mathcal G$ be a Markov topograph  of discriminator $D \neq 4$ with  labels in $(-\infty,-2] \cup [2,\infty)$.    Then $\mathcal G$, or a sign variant of it, contains an edge well, a vertex well, or a negative vertex well.
\end{cor}

The topographs covered by Corollary \ref{cor} have a simple basic structure, with regions increasing in absolute value from a central well. Topographs with $D=4$ have a different shape, explained in Section \ref{d4}. The complete classification is attained with Theorem \ref{class}.


\subsection{Generalizations of McShane's identity} \la{tmc}

Our first 
goal in this section is to prove the following generalization of \e{1/2}. This generalization was established  by Hu,  Tan, and  Zhang in  \cite[p.~495]{tan15} on the way to another identity \e{pan2} discussed below. The relatively short proof here is new, though 
originating in the partial proof from \cite[Sect.~4]{hin} of the $D=0$ case.

\begin{theorem} \cite{tan15} \la{summ}
Let $\mathcal G$ be a Markov topograph with labels in $(-\infty,-2] \cup [2,\infty)$ and discriminator $D\neq 4$. Then
\begin{equation}\label{topx}
   \sum_{ m  \in \mathcal G} \frac 2{e^{2 \rho(|m|)}+1} = 1 + \sum_{
\psset{xunit=0.08cm, yunit=0.08cm, runit=0.2cm}
\psset{linewidth=1pt}
\psset{dotsize=7pt 0,dotstyle=*}
\begin{pspicture}(2,3)(14,8.3)
          \psline(5,5)(3,8)
\psline(5,5)(3,2)
\psline(5,5)(9,5)


\rput(2,5){$_r$}
\rput(7,7.2){$_s$}
\rput(7,2.8){$_t$}

\rput(12.5,5){$_{\ \in \mathcal G}$}
        \end{pspicture}
} \frac{D}{|r s t|}
\end{equation}
where the left side is a sum over every region  of $\mathcal G$, and the right a sum over vertices of $\mathcal G$ with each term involving the three regions meeting there.
\end{theorem}

The case of Theorem \ref{summ} for a topograph with regions $>2$ and discriminator $D=0$ was proved in this form by Bowditch in \cite{bow96}. This was extended further in \cite[Thm. 3]{bow98} to allow topograph labels  $m\in \C$ with the BQ conditions: $m \not\in [-2,2]$ and only finitely many regions with $|m|\lqs 2$. Then in \cite[p.~495]{tan15} the condition $D=0$ was removed to allow any $D\in \C$, $D\neq 4$, and they remark that labels $m=\pm2$ are allowable.

Our proof 
is based on the following general  tree identity, valid for all $D\in \R$.

\begin{theorem} \la{mc}
For every rising or progressing tree, with base triple $(s_0, u_0, t_0)$ and  discriminator $D$, 
\begin{equation}\label{mcs}
  \sum_{q \, \in \, \Q \cap (0,1)} \left(\frac{2}{m_q \lambda_{m_q}} - \frac{D}{m_{q_0} m_q m_{q_0'}}\right)
    = 1-\frac{u_0}{s_0 t_0} -  \frac{1}{s_0 \lambda_{s_0}} -  \frac{1}{t_0 \lambda_{t_0}}.
\end{equation}
As usual, $q_0$ and $q_0'$ denote the unique  Farey neighbors that have mediant $q$.
\end{theorem}
\begin{proof}
For $T$ rising or progressing, define a function $\Upsilon = \Upsilon_T$ as follows.
 For   Farey neighbors $b<b'$ in $[0,1]$ with mediant $q$, set 
\begin{equation} \la{up}
  \Upsilon(b,b'):= 1-\frac{m_{q}}{m_b m_{b'}}.
\end{equation}
This definition generates the interesting identity
\begin{equation}\label{inr}
  \Upsilon(b,b') = \Upsilon(b,q)+\Upsilon(q,b') -\frac{D}{m_b m_q m_{b'}}.
\end{equation}
Letting $c$ be the mediant of $b$ and $q$, \e{inr} uses that $m_c+m_{b'}=m_b m_q$  by the local rule, and similarly for $q$ and $b'$.
Starting with
\begin{equation*}
  1-\frac{u_0}{s_0 t_0}  = \Upsilon(0, 1) 
    = \Upsilon(0, 1/2) + \Upsilon(1/2, 1) -D/(m_0 m_{1/2} m_1),
\end{equation*}
and repeatedly applying \e{inr} gives
\begin{equation} \la{giv}
  \sum_{0<i\lqs n} \Upsilon(b_{i-1}, b_i) = 1-\frac{u_0}{s_0 t_0} + \sum_{0<i< n} \frac{D}{m_{b_{i-1}} m_{b_i} m_{b_{i+1}}},
\end{equation}
for any sequence $0=b_0 < b_1 < \dots < b_n=1$ obtained by adding mediants.
We will use the  sequence  $\mathcal R_\ell=\{b_0, \dots, b_{2^\ell}\}$ in \e{bl}, the union of the Farey rows up to row $\ell$. 
Take $\ell$ to be large. Then each $b_i$ for $i$ odd is in row $\ell$  with $m_{b_i}$ relatively large. The even indexed $b_i$ are in lower rows, with $m_{b_i}$ smaller. We consider $\Upsilon(b_{i-1}, b_i)$ with $i$ odd and
 write the ordered triple $(m_{b_{i-1}}, m_q, m_{b_{i}})$ as $(s,u,t)$, with $t$ much larger than $s$. Then using
 \e{iv} in \e{up} implies that
\begin{align*}
  2\Upsilon(b_{i-1}, b_i) & =  1-\left(1+ \frac{4D}{s^2t^2} -4\left( \frac 1{s^2} +\frac 1{t^2}\right)\right)^{1/2}\\
  & =  1-\left(1 - \frac 4{s^2}\right)^{1/2} +O\left(\frac 1{t^2}\right)
  =  \frac 2{s\lambda_s} +O\left(\frac 1{t^2}\right).
\end{align*}
Hence, the left side of \e{giv} is
\begin{equation} \la{hu}
 \frac 1{s_0\lambda_{s_0}} +\frac 1{t_0\lambda_{t_0}} + \sum_{0 < i < 2^{\ell}, \ 2|i} \frac 2{m_{b_i}\lambda_{m_{b_i}}}
 +O\left(\sum_{0\lqs i \lqs 2^{\ell}, \ 2 \nmid i}\frac 1{m_{b_i}^2}\right).
\end{equation}
The denominators in row $\ell$ are between $\ell+1$ and the Fibonacci number $f_{\ell+2}$. 
Recall the bounds from Propositions \ref{sim2}, \ref{sih}
\begin{equation}\label{bs}
  m_{k/n} \gg e^{\omega_T n}, \qquad  m_{k/n} \gg n e^{\beta_T \min(k,n-k)},
\end{equation}
for rising and progressing trees, respectively. In the former case, 
since there are at most $\phi_{\text{Euler}}(n) <n$ Farey fractions with denominator $n$ for $n\gqs 2$, the error in \e{hu} has order at most
\begin{equation} \la{spring}
  \sum_{n\gqs \ell} n e^{-2 \omega_T n} \ll  \ell e^{-2 \omega_T \ell}.
\end{equation}
In the latter case, 
 the error  has order at most
\begin{equation} \la{spring2}
  \sum_{n\gqs \ell} \sum_{k=1}^{n-1}  n^{-2} e^{-2\beta_T \min(k,n-k)} \ll  \sum_{n\gqs \ell} n^{-2} \sum_{k\gqs 1}  e^{-2\beta_T k} \ll \ell^{-1}.
\end{equation}

Altogether, we have shown
\begin{equation}\label{pol}
   \sum_{0 < i < 2^{\ell}, \ 2|i} \frac 2{m_{b_i}\lambda_{m_{b_i}}} =  1-\frac{u_0}{s_0 t_0}  -\frac 1{s_0\lambda_{s_0}} -\frac 1{t_0\lambda_{t_0}} + \sum_{0<i< 2^{\ell}} \frac{D}{m_{b_{i-1}} m_{b_i} m_{b_{i+1}}}+ O \left(\ell^{-1} \right).
\end{equation}
Increasing $\ell$, we see that the contribution to the left side of \e{pol} from terms with $b_i$ in Farey rows $\ell_0$ to $\ell$ is at most $O(\ell_0^{-1})$. This uses $\lambda_x \gqs x/2$ and the  reasoning from \e{spring} and  \e{spring2}. Hence the left side  is absolutely convergent as $\ell \to \infty$. This implies the same for the right side and the proof is complete.
\end{proof}

\begin{proof}[Proof of Theorem \ref{summ}]
Note that both sides of \e{topx} depend only on the absolute value of labels and so are unchanged under sign variants.
By Corollary \ref{cor}, $\mathcal G$ or a sign variant has   a  well. If it has an edge or vertex well then we may choose any vertex $v_0$ in  $\mathcal G$ and  the three  trees rooted at $v_0$ will be rising. If it has a negative vertex well then we should use that as the root. Sign variants of each of the three  trees rooted at this well will be rising or progressing by Lemma \ref{neg}. Let $T_0=(s_0,u_0,t_0)$, $T_1=(t_0,v_0,r_0)$, $T_2=(r_0,w_0,s_0)$ 
be these rooted trees, as displayed on the left of Figure \ref{topb}. Noting that
\begin{equation*}
   e^{2\rho(x)}+1 = \lambda_x^2+1 = (\lambda_x + \lambda_x^{-1}) \lambda_x = x \lambda_x  \qquad (x\gqs 2),
\end{equation*}
gives
\begin{multline*}
   \sum_{ m  \in \mathcal G} \frac 2{e^{2 \rho(|m|)}+1}- \sum_{
\psset{xunit=0.08cm, yunit=0.08cm, runit=0.2cm}
\psset{linewidth=1pt}
\psset{dotsize=7pt 0,dotstyle=*}
\begin{pspicture}(2,3)(14,8.3)
          \psline(5,5)(3,8)
\psline(5,5)(3,2)
\psline(5,5)(9,5)
\rput(2,5){$_r$}
\rput(7,7.2){$_s$}
\rput(7,2.8){$_t$}
\rput(12.5,5){$_{\ \in \mathcal G}$}
        \end{pspicture}
} \frac{D}{|r s t|} = \sum_{i=0}^2  \sum_{q \, \in \, \Q \cap (0,1)} \left(\frac{2}{m(T_i)_q \lambda_{m(T_i)_q}} - \frac{D}{m(T_i)_{q_0} m(T_i)_q m(T_i)_{q_0'}}\right)\\
  + \frac{2}{s_0 \lambda_{s_0}}  + \frac{2}{t_0 \lambda_{t_0}}  + \frac{2}{r_0 \lambda_{r_0}} -  \frac{D}{s_0 t_0 r_0}.
\end{multline*}
Applying Theorem \ref{mc} to the above expression, and simplifying with the local rule, shows it  equals $1$. 
\end{proof}

In \cite{tan15} they continue as follows. For the sum on the right of \e{topx}, fix a region $u=m_q$ and consider all the terms that involve it:
\begin{equation*}
 (m_q m_{q_0} m_{q_0'})^{-1} +  \sum_{j\gqs 0}(m_q m_{q_{j+1}} m_{q_j})^{-1} +  \sum_{j\gqs 0}(m_q m_{q'_{j+1}} m_{q'_j})^{-1} = \sum_{j \in \Z}(m_q m_{q_{j+1}} m_{q_j})^{-1},
\end{equation*}
using the formulas in Proposition \ref{start}. This sum telescopes by \cite[Prop 4.1]{tan15} and converges to 
\begin{equation} \la{sur}
  \frac{1}{f(s,u,t) g(s,u,t) \cdot u (\lambda_u - \lambda_u^{-1})} = \frac{\sqrt{u^2-4}}{u(u^2-D)} =  \left( 1-\frac 2{e^{2 \rho(u)}+1}\right)  \frac{1}{u^2-D},
\end{equation}
using \e{ib}, \e{thu}. Hence 
\begin{equation}\label{topx2}
   \sum_{
\psset{xunit=0.08cm, yunit=0.08cm, runit=0.2cm}
\psset{linewidth=1pt}
\psset{dotsize=7pt 0,dotstyle=*}
\begin{pspicture}(2,3)(14,8.3)
          \psline(5,5)(3,8)
\psline(5,5)(3,2)
\psline(5,5)(9,5)


\rput(2,5){$_r$}
\rput(7,7.2){$_s$}
\rput(7,2.8){$_t$}

\rput(12.5,5){$_{\ \in \mathcal G}$}
        \end{pspicture}
} \frac{D}{|r s t|}
=
 \sum_{ m  \in \mathcal G}  \left( 1-\frac 2{e^{2 \rho(|m|)}+1} \right) \frac{D/3}{m^2-D}.
\end{equation}
Finally, combining \e{topx} and \e{topx2} gives the following version of \cite[Thm.~1.1]{tan15}, specialized to real labels.

\begin{theorem} \cite{tan15} \la{summg}
Let $\mathcal G$ be a Markov topograph with labels in $(-\infty,-2] \cup [2,\infty)$ and discriminator $D\neq 4$. Then
\begin{equation}\label{pan2}
  \sum_{ m  \in \mathcal G} \left[ \left(1+\frac{D/3}{m^2-D}\right)\left(\frac 2{e^{2 \rho(|m|)}+1}-1\right)+1\right] =1.
\end{equation}
\end{theorem}

A version using weights for regions of each color in Figure \ref{topv} is also given in \cite[Thm.~1.3]{tan15}.
It is possible that the right side of \e{topx}  simplifies further. Such sums for quadratic form topographs are evaluated explicitly in \cite[Thm.~9.1]{OStop}, and also by telescoping in \cite{ka}.

\begin{ex} 
{\rm In the slowest converging case of $\mathcal G=(-2,-2,-2)$ and $D=20$, the left side of \e{pan2} is $\approx 1.00007009$ after summing over all regions $m$ with $|m|\lqs 10^5$.}
\end{ex}

We also note a useful byproduct of \e{sur}:

\begin{cor} \la{surx}
If a Markov topograph of discriminator $D$ has all regions $m$ satisfying $m>2$, then they also satisfy $m^2>D$.
\end{cor}

\subsection{Generalizing  an identity of Hines} \la{hiid}


The following interesting identity was discovered by Hines  \cite{hin} in the case $D=0$. Recall the height $h_{r,s,t}(m)$ of a region $m$ from Definition \ref{hei}.

\begin{theorem} \la{hnn}
Let $\mathcal G$ be a Markov topograph with labels in $(2,\infty)$ and discriminator $D\neq 4$.
For any triple $(r,s,t)$ of  $\mathcal G$ we have
\begin{equation}\label{hin2b}
  \sum_{m \in \mathcal G} h_{r,s,t}(m) \cdot \log\left( 1-\frac{D-4}{m^2 -4}\right) =\rho(r)+\rho(s)+\rho(t),
\end{equation}
where the  sum is over all regions of $\mathcal G$. In our abused notation,  each region  has  label $m$ and height $ h_{r,s,t}(m)$.
\end{theorem}

Note that the logarithms above are always evaluated at positive real numbers by Corollary \ref{surx}. The simple idea behind Theorem \ref{hnn} is to use the slope differences in \e{sps} to make a telescoping sum over the rationals in $[0,1]$. This is done for each of the three trees attached to a vertex of the topograph and the totals added together. No further justification was given in \cite[p. 8]{hin}. We supply a proof here that this 
is correct, based on our work in Sections \ref{bou} and \ref{y4}.

\begin{theorem} \la{hn}
Let $T$ be a rising tree with  base triple  $(s_0, u_0, t_0)$ and any discriminator $D$. Then
\begin{equation}\label{hin}
  \sum_{k/n \, \in \, \Q \cap (0,1)} n \log\left(\frac{m_{k/n}^2 -D}{m_{k/n}^2 -4}\right) = \log\left(\frac{\lambda_{s_0}\lambda_{t_0} (\lambda_{s_0}^2-1)(\lambda_{t_0}^2-1)}{(\lambda_{s_0} u_0 -t_0)(\lambda_{t_0} u_0 -s_0)}\right).
\end{equation}
\end{theorem}
\begin{proof}
Use the  sequence $\mathcal R_\ell$  in \e{bl} again. Write $\g(b_i,b_{i+1})$ for the slope between the points $(b_i, \phi(b_i))$ and $(b_{i+1}, \phi(b_{i+1}))$. Then by telescoping
\begin{equation}\label{er}
  \sum_{0<i<2^\ell} \bigl(\g(b_i,b_{i+1}) - \g(b_{i-1},b_{i}) \bigr) = \g(b_{2^\ell-1},1) -\g(0, b_1).
\end{equation}
Let $q=b_i$ be in row $r$ with  $1 \lqs r \lqs \ell-1$, so that $0<q<1$ and $i$ is even. Its two neighbors in $\mathcal R_{\ell}$ are in row $\ell$ so that, in the notation of Section \ref{rey}, we have
$b_{i-1}=q_{\ell-r}$ and $b_{i+1}=q'_{\ell-r}$. Therefore, as in the proof of  Proposition \ref{sjs}, 
\begin{align*}
  \g(b_i,b_{i+1}) - \g(b_{i-1},b_{i}) & = S'_{\ell-r}(q) - S_{\ell-r}(q) \\
  & = S'(q) + n \varepsilon'_{\ell-r} - (S(q)  - n \varepsilon_{\ell-r}) \\
  & = n \log\left(\frac{m_q^2-D}{m_q^2-4}\right) + n (\varepsilon_{\ell-r} +  \varepsilon'_{\ell-r}),
\end{align*}
where $n = n_q$ is the denominator of $q$.
With Proposition \ref{shut},
 when $1 \lqs r \lqs \ell-1$,
\begin{equation*}  
  n (\varepsilon_{\ell-r} +  \varepsilon'_{\ell-r})  \lqs 2C_T n e^{-2\rho(m_q) (\ell-r)}.
\end{equation*}
Using \e{cru} and Proposition \ref{sim2} then gives
\begin{equation}  \la{era}
  n (\varepsilon_{\ell-r} +  \varepsilon'_{\ell-r}) = O\left( n e^{-2\omega_T n (\ell-r)}\right).
\end{equation}

 If $q=b_i$ is in row $\ell$ then $i$ is odd and the neighbors of $q$ are in lower rows so that
$b_{i-1}=q_{0}$ and $b_{i+1}=q'_{0}$. In this case use Corollary \ref{sldi}  to obtain
\begin{equation} \la{erb}
   |\g(b_i,b_{i+1}) - \g(b_{i-1},b_{i})| = |S_0'(q) -S_0(q)| = O\left( n e^{-2\omega_T n}\right).
\end{equation}
Altogether, the left side of \e{er} can be written as  
\begin{equation} \la{le}
  E_\ell+  \sum_{q \in \mathcal R_{\ell-1}, \ 0<q<1} n_q \log\left(\frac{m_q^2-D}{m_q^2-4}\right),
\end{equation}
with the terms of   \e{era} and \e{erb} incorporated into an error term $E_\ell$. There are $2^\ell$ terms of the form \e{erb}, each with a $q$ that has a denominator $n$ satisfying  $n\gqs \ell+1$. There can be at most $\phi_{\text{Euler}}(n) <n$ such denominators, so their total has order less than
\begin{equation} \la{con}
  \sum_{n \gqs \ell} n^2 e^{-2\omega_T n} \ll \ell^2 e^{-2\omega_T \ell}.
\end{equation}
For the $2^r$ terms of the form \e{era}, consider first those with $r\lqs \ell/2$. They contribute less than
\begin{equation*}
  \sum_{n\gqs 1} n^2 e^{-\omega_T n \ell} \ll   e^{-\omega_T  \ell}.
\end{equation*}
The contribution from the terms with $r >\ell/2$  is the same as \e{con}, but summing over $n>\ell/2$.

By Propositions \ref{shut2} and \ref{sjs2},
the right side of \e{er} equals
$
  S(1)-S'(0)-(\varepsilon_{\ell-1}(1) +  \varepsilon'_{\ell-1}(0)),
$
and in total we have established
\begin{equation} \la{lex}
    \sum_{q \in \mathcal R_{\ell-1}, \ 0<q<1} n_q \log\left(\frac{m_q^2-D}{m_q^2-4}\right) = S(1)-S'(0) + E^*_\ell
\end{equation}
with $E^*_\ell \to 0$ as $\ell \to \infty$. Similar arguments to those used in \e{con} show that the left side of \e{lex} is absolutely convergent as $\ell \to \infty$.
The formulas \e{sg} and \e{sgb} complete the proof.
\end{proof}

\begin{proof} [Proof of Theorem \ref{hnn}]
We use the same setup as in the proof of Theorem \ref{summ}. Our restriction to labels $>2$ means $ \mathcal G$ has an edge or vertex well by Proposition \ref{real}, and the trees $T_i$ are now all rising. Then
\begin{multline*}
   \sum_{ m  \in \mathcal G}  h_{r_0,s_0,t_0}(m) \cdot \log\left( 1-\frac{D-4}{m^2 -4}\right) = \sum_{i=0}^2  \sum_{k/n \, \in \, \Q \cap (0,1)} 
   n \cdot \log\left( \frac{m(T_i)_{k/n}^2 -D}{m(T_i)_{k/n}^2 -4}\right)
   \\
  + \log\left(\frac{r_0^2-D}{r_0^2 -4}\right)+ \log\left(\frac{s_0^2-D}{s_0^2 -4}\right)+ \log\left(\frac{t_0^2-D}{t_0^2 -4}\right).
\end{multline*}
Apply Theorem \ref{hn} to the right side above, then use the identity
\begin{equation*}
   \frac{\lambda_{s_0}}{\lambda_{s_0} u_0 -t_0} =  \frac{1}{ u_0 -t_0 (s-\lambda_{s_0})}
   =  \frac{1}{t_0\lambda_{s_0} - r_0}
\end{equation*}
and its variants to express everything in terms of the central regions $r_0, s_0, t_0$. Writing these as $r,s,t$ for simplicity,  the right side is
\begin{multline*}
  \log\left( \frac{(\lambda_s^2-1)(\lambda_t^2-1)}{(t \lambda_s -r)(s \lambda_t -r)}
   \cdot \frac{(\lambda_t^2-1)(\lambda_r^2-1)}{(r \lambda_t -s)(t \lambda_r -s)}
   \cdot \frac{(\lambda_r^2-1)(\lambda_s^2-1)}{(s \lambda_r -t)(r \lambda_s -t)}
  \right) \\
  +  \log\left( \frac{r^2-D}{r^2-4} \cdot \frac{s^2-D}{s^2-4} \cdot \frac{t^2-D}{t^2-4}\right).
\end{multline*}
Variants of the identity
$
(t\lambda_{s} - r) (r\lambda_{s} - t) = \lambda_{s}(s^2 -D)
$, where all factors are positive,
allow simplification to
\begin{equation*}
  \log\left( \frac{(\lambda_r^2-1)^2(\lambda_s^2-1)^2(\lambda_t^2-1)^2}{\lambda_r \lambda_s  \lambda_t (r^2-4)(s^2-4)(t^2-4)
  }
  \right)  = \log\left( \lambda_r \lambda_s  \lambda_t 
  \right),
\end{equation*}
using  $(\lambda_{x}^2-1)^2 = \lambda_{x}^2(x^2-4)$ finally.
\end{proof}

\section{The case $D=4$} \la{d4}
\subsection{Rising trees}

The simplest rising trees have  discriminator $D=4$. By Lemma \ref{tr4} they cannot have descending paths and must be strictly rising. Therefore all of them may be  found by choosing $s_0, t_0>2$ and then taking $2u_0=s_0 t_0+\sqrt{(s_0^2-4)(t_0^2-4)}$ as in \e{iv} to obtain the base triple.
We know by Theorem \ref{mthm}  that the graph of $\Phi$ is a straight line, so that
\begin{equation} \la{line}
  \Phi(x) = (\rho(t_0)-\rho(s_0))x + \rho(s_0).
\end{equation}
Therefore
\begin{equation*}
  m_{k/n} = 2\cosh \bigl(n \Phi(k/n)\bigr) = 2\cosh \bigl((\rho(t_0)-\rho(s_0))k + \rho(s_0)n \bigr).
\end{equation*}
This simplifies further when $\rho(s_0)=\rho(t_0)$, which can only happen if $s_0$ and $t_0$ are adjacent to an implicit $2$-region below, by the local rule. Let $x= s_0=t_0$. Then
\begin{equation*}
  m_{k/n}   = 2\cosh(\rho(x)n)  = e^{\rho(x)n}+ e^{-\rho(x)n} = \lambda_x^n + \lambda_x^{-n}  = 2 T_n(x/2),
\end{equation*}
for the Chebyshev polynomial $T_n$, as in \cite[Sect. 3]{mcg}. In this case $\Gamma_T$, the right side of the wedge $\mathcal W_T$, is a vertical line by \e{f2}.

For $\rho(s_0) \neq \rho(t_0)$, $\Gamma_T$ is a line of slope $\rho(s_0)/(\rho(s_0)-\rho(t_0))$. If this slope is rational, or infinite  with $\rho(s_0)=\rho(t_0)$,  then there will be many repeated numbers in these trees by Lemma \ref{latt} and Theorem \ref{ag}. On the other hand, if  $\rho(s_0)/\rho(t_0)$ is irrational then each number must appear uniquely on the tree, since  $t \cdot \Gamma_T$ cannot intersect two lattice points simultaneously. We have shown:

\begin{prop} \la{ris}
A rising tree, with discriminator $D=4$ and base triple $(s_0,u_0,t_0)$, has  repeated numbers if and only if $\rho(s_0)/\rho(t_0)$ is rational.
\end{prop}

We may give a simple criterion for this irrationality:

\begin{lemma}
For  $s,t \in \Q_{>2}$, the ratio $\rho(s)/\rho(t)$ is irrational if  $\sqrt{(s^2-4)(t^2-4)}$ is irrational.
\end{lemma}
\begin{proof}
If the ratio $\rho(s)/\rho(t)$  equals $a/b \in \Q$ then $\lambda_t^a=\lambda_s^b$ for positive integers $a, b$. From this,
\begin{align*}
  \left(\tfrac{t+\sqrt{t^2-4}}2 \right)^a = \left(\tfrac{s+\sqrt{s^2-4}}2 \right)^b & \implies u_1 +v_1 \sqrt{t^2-4} =  u_2 +v_2 \sqrt{s^2-4}\\
  & \implies \left(v_1 \sqrt{t^2-4} - v_2 \sqrt{s^2-4}\right)^2 = \left( u_2 -u_1\right)^2
\end{align*}
for positive rationals $u_1, u_2, v_1, v_2$ and the result follows.
\end{proof}

\begin{ex} 
{\rm The rising tree with base $(3,6+\sqrt{15},4)$ has $D=4$, $\rho(3)/\rho(4)$ irrational since $\sqrt{5 \cdot 12}$ is irrational, and no repeats in its numbers
\begin{equation*}
  3,4,6+\sqrt{15},14+3 \sqrt{15},21+4 \sqrt{15},36+8 \sqrt{15},78+15 \sqrt{15},94+21
   \sqrt{15},  \dots
\end{equation*}
By Lemma \ref{uniq} this tree has uniqueness.}
\end{ex}

\subsection{Leveling trees} \la{levf}

Assume that $\mathcal G$ is a real Markov topograph with $D=4$ and all regions $\gqs 2$. One possibility is that $\mathcal G$ is constant with all regions equal to $2$. Otherwise, direct all non-zero edges to make them positive. Let the regions $s, u, t$ meet at a vertex. If $u > st/2$ then the edge $e$ between $s$ and $t$ is directed to $u$ and the tree with base $(s,u,t)$ is strictly rising. If $u \lqs st/2$ then $e$ is undirected or directed away from $u$. The tree with this base has some interesting new features.

\begin{adef}\la{dxs}
{\rm A {\em leveling tree} is a real 
Markov tree, with discriminator $D=4$, where the components of its  base triple $(s_0, u_0, t_0)$ satisfy
\begin{equation}\label{stb}
  2 < s_0, \qquad 2  < t_0, \qquad 2 \lqs u_0 < s_0 t_0/2,
\end{equation}}
or else the base triple takes the form $(2,t_0,t_0)$ or $(s_0,s_0,2)$ for $s_0,t_0>2$.
\end{adef}

All leveling trees may be  found by choosing $s_0, t_0 \gqs 2$ with $s_0t_0 \neq 4$ and then taking $2u_0=s_0 t_0-\sqrt{(s_0^2-4)(t_0^2-4)}$ to obtain the base triple.

\begin{prop} \la{www}
The numbers in a leveling tree are $\gqs 2$. All edges are positively directed away from the root except for those on a simple {\em descending path} from the root. If $\rho(s_0)/\rho(t_0)$ is rational then there is exactly one $2$-region and it occurs at the  end of the descending path.  If $\rho(s_0)/\rho(t_0)$ is irrational then there are no  $2$-regions, the descending path is infinite, and regions adjacent to the path decrease with limit $2$. 
\end{prop}
\begin{proof}
Give all  edges  their positive edge direction. The edge between $s_0$ and $t_0$ is directed away from $u_0$.
Starting at the  base triple $(s_0,u_0,t_0)$, we assume for now that there are no $2$-regions nearby. The central vertex $v$ of this base triple cannot have all three edges directed out since that would make a vertex well, see Lemma \ref{wells}. So at least one edge $e$ is directed in, say the edge between $s_0$ and $u_0$. By the arguments of Proposition \ref{zag2},  $\rho(s_0)+\rho(u_0)=\rho(t_0)$. So $\rho(t_0) >\rho(s_0)$, $t_0>s_0$ and the third edge, between $t_0$ and $u_0$, is directed out from $v$. Now let $w$ be the fourth region around edge $e$, that $e$ is directed away from. We have
\begin{equation}\label{ewa}
  2w =s_0 u_0 - \sqrt{(s_0^2-4)(u_0^2-4)} \gqs 4
\end{equation}
where the equality in \e{ewa} is as in \e{iv}. The inequality in \e{ewa} is equivalent to $(s_0-u_0)^2\gqs 0$. Hence the triple $(s_0,w,u_0)$ satisfies the same conditions  as the initial triple with \e{stb}. Repeating our argument produces a unique path $P$ of edges leading away from the root, and with all of its edges directed towards to root. All other edges are directed away from the root by the climbing lemma. The path ends if a region with label $2$ is reached; all further edges are in copies of a rising tree directed away from this $2$-region.

The path $P$ may be described as follows. In the above example with $\rho(t_0) >\rho(s_0)$, the path moved left. Altogether, it will move left $a_0$ times where $a_0$  satisfies $0\lqs \rho(t_0) -a_0 \rho(s_0) < \rho(s_0)$, reaching the adjacent regions $s_0, u_1$ with $\rho(u_1)= \rho(t_0) -a_0 \rho(s_0)$ by Proposition \ref{zag2}. The  path will then move right $a_1$ times with  $0\lqs \rho(s_0) -a_1 \rho(u_1) < \rho(u_1)$, and so on. If  $\rho(t_0) /\rho(s_0) \in \Q$, equaling $[a_0, a_1, \dots, a_m]$ with $a_m\gqs 2$, then we see that the path moves left and right alternately by these number of edges. At the end it moves left or right by $a_m-1$ edges to reach a region with label $2$. The Farey index of the $2$-region is $q=[0, 1+a_0, a_1, \dots, a_m]$, making $m_q=2$. There is exactly one $2$-region in this case.

If  $\rho(t_0) /\rho(s_0) \notin \Q$, then the path $P$ continues indefinitely, passing regions with Farey indices given by the convergents to 
\begin{equation} \la{frh}
 \frac{\rho(s_0)}{\rho(s_0)+\rho(t_0)} = [0, 1+a_0, a_1, \dots ].
\end{equation}
The tree can have no $2$-regions in this case. The argument from the proof of Proposition \ref{real} shows that the regions adjacent to the path must approach $2$. Also, by the climbing lemma,  regions adjacent to the left of the path strictly decrease, as do those to the right.
\end{proof}

\SpecialCoor
\psset{griddots=5,subgriddiv=0,gridlabels=0pt}
\psset{xunit=0.35cm, yunit=0.35cm, runit=0.5cm}
\psset{linewidth=1pt}
\psset{dotsize=4pt 0,dotstyle=*}
\begin{figure}[ht]
\centering

\psset{arrowscale=1.4,arrowinset=0.3,arrowlength=1.1}
\newrgbcolor{light}{0.8 0.8 1.0}
\newrgbcolor{pale}{1 0.7 1}
\newrgbcolor{pale}{1 0.7 0.4}
\newrgbcolor{pale}{1 0.85 0.65}

\begin{pspicture}(-15,-1.6)(15,17.5) 

\psset{arrowscale=1.8,arrowinset=0.3,arrowlength=1.1}
\newrgbcolor{light}{0.9 0.7 1.0}
\newrgbcolor{blue2}{0.3 0.3 0.8}
\newrgbcolor{light}{0.7 0.7 1.0}

\newrgbcolor{light}{0.87 0.7 0.8}

\newrgbcolor{yello2}{0.964844 0.871094 0.453125}

\psset{arrowscale=1.4,arrowinset=0.3,arrowlength=1.0}



\psset{linecolor=yello2}


\pscurve[linewidth=7pt,linecolor=light](-4.03543,17.182)(-4.01956,14.582)(-3.22543,11.689)(-3.36415,8.69223)
(-2.34549,5.87047)(0,4)(0,2)(0,0)

\psline[linecolor=light,linestyle=dashed,dash=3pt 2pt](-2,-1.6)(0,0)(2,-1.6)

\psline[ArrowInside=->,ArrowInsidePos=0.6](0,4)(2.34549,5.87047)(5.32318,6.23569)(8.21397,5.4336)(10.7154,3.77746)(12.438,1.82995)

\psline[ArrowInside=->,ArrowInsidePos=0.6](10.7154,3.77746)(13.1809,2.95215)

\psline[ArrowInside=->,ArrowInsidePos=0.6](8.21397,5.4336)(11.2111,5.56406)(13.7494,5.00099)

\psline[ArrowInside=->,ArrowInsidePos=0.6](11.2111,5.56406)(13.6909,6.34558)

\psline[ArrowInside=->,ArrowInsidePos=0.6](5.32318,6.23569)(7.93444,7.71262)(10.8743,8.31034)(13.4694,8.15128)

\psline[ArrowInside=->,ArrowInsidePos=0.6](10.8743,8.31034)(13.2013,9.47015)

\psline[ArrowInside=->,ArrowInsidePos=0.6](7.93444,7.71262)(9.9615,9.92419)(12.1545,11.3209)

\psline[ArrowInside=->,ArrowInsidePos=0.6](9.9615,9.92419)(11.1623,12.2303)

\psline[ArrowInside=->,ArrowInsidePos=0.6](2.34549,5.87047)(3.36415,8.69223)(5.38512,10.9094)(7.9923,12.3935)(10.5078,13.0511)

\psline[ArrowInside=->,ArrowInsidePos=0.6](7.9923,12.3935)(9.84196,14.2207)

\psline[ArrowInside=->,ArrowInsidePos=0.6](5.38512,10.9094)(6.62203,13.6425)(8.27056,15.6531)

\psline[ArrowInside=->,ArrowInsidePos=0.6](6.62203,13.6425)(7.04442,16.208)

\psline[ArrowInside=->,ArrowInsidePos=0.6](3.36415,8.69223)(3.22543,11.689)(4.01956,14.582)(5.33328,16.8257)

\psline[ArrowInside=->,ArrowInsidePos=0.6](4.01956,14.582)(4.03543,17.182)

\psline[ArrowInside=->,ArrowInsidePos=0.6](3.22543,11.689)(2.16744,14.4963)(1.91145,17.0836)

\psline[ArrowInside=->,ArrowInsidePos=0.6](2.16744,14.4963)(0.652066,16.609)

\psline[ArrowInside=->,ArrowInsidePos=0.6](-3.22543,11.689)(-2.16744,14.4963)(-0.652066,16.609)

\psline[ArrowInside=->,ArrowInsidePos=0.6](-2.16744,14.4963)(-1.91145,17.0836)

\psline[ArrowInside=->,ArrowInsidePos=0.6](-4.01956,14.582)(-5.33328,16.8257)

\psline[ArrowInside=->,ArrowInsidePos=0.6](-3.36415,8.69223)(-5.38512,10.9094)(-6.62203,13.6425)(-7.04442,16.208)

\psline[ArrowInside=->,ArrowInsidePos=0.6](-6.62203,13.6425)(-8.27056,15.6531)

\psline[ArrowInside=->,ArrowInsidePos=0.6](-5.38512,10.9094)(-7.9923,12.3935)(-9.84196,14.2207)

\psline[ArrowInside=->,ArrowInsidePos=0.6](-7.9923,12.3935)(-10.5078,13.0511)

\psline[ArrowInside=->,ArrowInsidePos=0.6](-2.34549,5.87047)(-5.32318,6.23569)(-7.93444,7.71262)(-9.9615,9.92419)(-11.1623,12.2303)

\psline[ArrowInside=->,ArrowInsidePos=0.6](-9.9615,9.92419)(-12.1545,11.3209)

\psline[ArrowInside=->,ArrowInsidePos=0.6](-7.93444,7.71262)(-10.8743,8.31034)(-13.2013,9.47015)

\psline[ArrowInside=->,ArrowInsidePos=0.6](-10.8743,8.31034)(-13.4694,8.15128)

\psline[ArrowInside=->,ArrowInsidePos=0.6](-5.32318,6.23569)(-8.21397,5.4336)(-11.2111,5.56406)(-13.6909,6.34558)

\psline[ArrowInside=->,ArrowInsidePos=0.6](-11.2111,5.56406)(-13.7494,5.00099)

\psline[ArrowInside=->,ArrowInsidePos=0.6](-8.21397,5.4336)(-10.7154,3.77746)(-13.1809,2.95215)

\psline[ArrowInside=->,ArrowInsidePos=0.6](-10.7154,3.77746)(-12.438,1.82995)

\psline[ArrowInside=->,ArrowInsidePos=0.6,linecolor=red](-4.03543,17.182)(-4.01956,14.582)
(-3.22543,11.689)(-3.36415,8.69223)(-2.34549,5.87047)(0,4)(0,0)

\rput(-2,10){$P$}
\rput(-2.34549,2){$3$}
\rput(2.34549,2){$4$}

\rput(-11.1048,4.6315){$_{12.67}$}

\rput(-10.5457,9.18954){$_{8.94}$}
\rput(-8.30087,6.6009){$_{5.01}$}

\rput(-7.40607,13.1265){$_{2.79}$}

\rput(-4.38281,11.514){$_{2.06}$}
\rput(-4.69099,7.74094){$2.38$} 
\rput(-3.0867,14.6858){$_{2.01}$}

\rput(0.,6.8){$2.13$}

\rput(3.0867,14.6858){$_{10.90}$}
\rput(4.69099,7.74094){$5.51$}
\rput(4.38281,11.514){$_{7.72}$}

\rput(7.40607,13.1265){$_{40.37}$}

\rput(8.30087,6.6009){$_{19.91}$}
\rput(10.5457,9.18954){$_{105.64}$}

\rput(11.1048,4.6315){$_{74.11}$}

\psset{linecolor=gray}

\psset{linecolor=black}

\psdot[linecolor=black](0,0)

\end{pspicture}
\caption{The leveling  tree $(3,6-\sqrt{15},4)$, so $D=4$, with an infinite descending path $P$. Numbers are shown correct to the nearest hundredth.}
\label{dtree}
\end{figure}
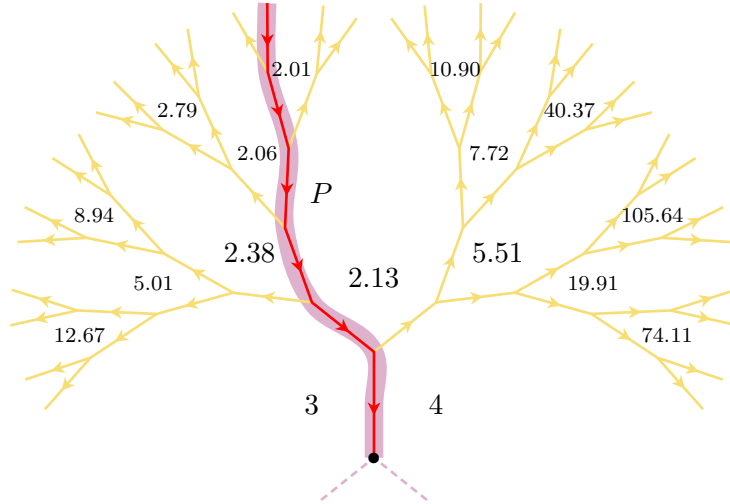

The leveling tree with $(s_0,u_0,t_0) = (3,6-\sqrt{15},4)$ is drawn in Figure \ref{dtree}. The ratio $\rho(t_0) /\rho(s_0)$ is irrational and equals $[1,2, 1, 2, 1, 1, 61, \dots]$ so that the path $P$, starting from the root edge, goes  $1$ edge left, $2$ edges right,  $1$ edge left, and so on. This path is infinite with all its edges directed towards the root. Leveling trees have a feature not present in rising or progressing trees: a sequence of regions moving away from the root that levels off, either along the infinite descending path $P$ or along the boundary of a $2$-region.


\SpecialCoor
\psset{griddots=5,subgriddiv=0,gridlabels=0pt}
\psset{xunit=1cm, yunit=1cm, runit=1cm}
\psset{linewidth=1pt}
\psset{dotsize=1.4pt,dotstyle=*}
\begin{figure}[t]
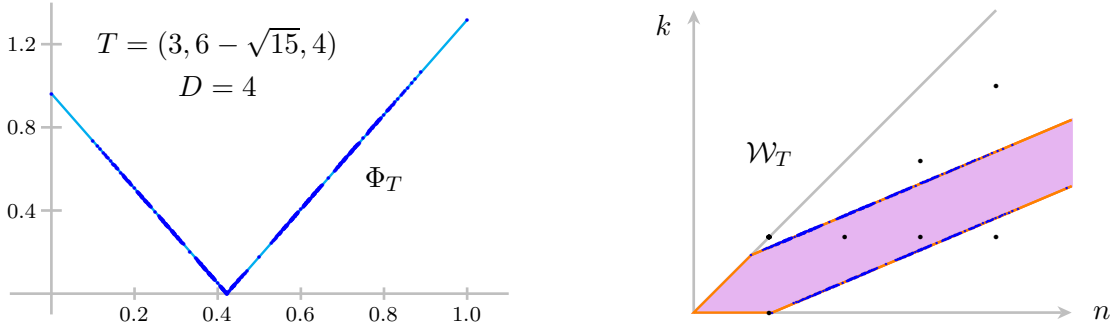

\centering

\psset{arrowscale=1.4,arrowinset=0.3,arrowlength=1.1}
\newrgbcolor{light}{0.8 0.8 1.0}
\newrgbcolor{pale}{1 0.7 1}
\newrgbcolor{pale}{1 0.7 0.4}

\newrgbcolor{markov}{0.902344 0.710938 0.945313}
\psset{linecolor=markov}

}

\end{pspicture}
\caption{$\Phi_T$ and $\mathcal W_T$ for the leveling tree $T=(3,6-\sqrt{15},4)$.}
\label{tg}
\end{figure}

\begin{prop}
Let $T$ be a leveling tree with base triple $(s_0,u_0,t_0)$. The graph of its encoding function $\phi_T(x)$   
consists of two  lines on each side of the $x$ value $\rho(s_0)/(\rho(s_0)+\rho(t_0))$. The slope of the left line is $-\rho(s_0)-\rho(t_0)$ and the slope  on the right is $\rho(s_0)+\rho(t_0)$. In the special cases of $s_0=2$ or $t_0=2$  there is only one line.
\end{prop}
\begin{proof}
Let $\alpha = \rho(s_0)/(\rho(s_0)+\rho(t_0))$, which may be rational or irrational.  All vertices that are not on the descending path  $P$, and not on the boundary of a $2$-region, are surrounded by a usual ordered triple $(s,u,t)= (m_{q_0}, m_q, m_{q'_0})$,  with $q_0 < q < q_0' < \alpha$ or  $\alpha < q_0 < q < q_0'$. The identity $\rho(u)=\rho(s)+\rho(t)$ implies that the slopes on the graph of $\phi_T$ between $q_0, q, q_0'$ are equal by \e{zu2}: $\g(q_0, q) = \g(q, q_0')$. Since this is true for all triples to the left and right of $\alpha$, it follows that the graph consists of two  lines. This must be amended if $s_0=2$ or $t_0=2$; then  $\alpha = 0$ or $\alpha=1$, respectively, and there is just one line.


The slope of the left line must equal
\begin{equation*}
  S'_{j+1}(0) = \g(0, q'_{j+1}) = \frac{\phi(q'_{j+1}) -\phi(0)}{q'_{j+1} -0} = \rho(m_{q'_{j+1}}) - (j+2)\rho(m_0),
\end{equation*}
for $j$ large enough. By Proposition \ref{stx}, Lemma \ref{rox} and \e{cru},
\begin{align*}
  S'_{j+1}(0) & = \rho\left( f(u_0,s_0,t_0) e^{\rho(s_0) j} + g(u_0,s_0,t_0) e^{-\rho(s_0) j}\right) - (j+2)\rho(s_0)\\
   &  = \rho\left( f(u_0,s_0,t_0) e^{\rho(s_0) j} \right) - (j+2)\rho(s_0) + O\left(  e^{-2\rho(s_0) j}\right)\\
    &  = \log f(u_0,s_0,t_0)+ \rho(s_0) j  - (j+2)\rho(s_0) + O\left(  e^{-2\rho(s_0) j}\right).
\end{align*}
Then $S'_{j+1}(0) \to S'(0) = \log( f(u_0,s_0,t_0)) - 2\rho(s_0)$ as $j \to \infty$, giving the slope on the left, which must be constant in $j$ after a certain point. This agrees with \e{sg}. To simplify $S'(0)$ use
\begin{align*}
  2u_0 & = s_0 t_0 - \sqrt{(s_0^2-4)(t_0^2-4)} \\
   & =(\lambda_{s_0} + \lambda_{s_0}^{-1})(\lambda_{t_0} + \lambda_{t_0}^{-1}) - (\lambda_{s_0} -\lambda_{s_0}^{-1})(\lambda_{t_0} - \lambda_{t_0}^{-1})\\
   & = 2\lambda_{s_0} \lambda_{t_0}^{-1} + 2\lambda_{t_0} \lambda_{s_0}^{-1},
\end{align*}
with \e{use}. Then
\begin{equation*}
  e^{S'(0)} = \frac{\lambda_{s_0} u_0 - t_0}{\lambda_{s_0}^2 \sqrt{s_0^2-4}} = \frac{\lambda_{s_0} (\lambda_{s_0} \lambda_{t_0}^{-1} + \lambda_{t_0} \lambda_{s_0}^{-1}) - (\lambda_{t_0}+\lambda_{t_0}^{-1})}{\lambda_{s_0}^2 (\lambda_{s_0}-\lambda_{s_0}^{-1})}
  =\frac 1{\lambda_{s_0}\lambda_{t_0}},
\end{equation*}
producing $S'(0) = -\rho(s_0)-\rho(t_0)$ as we wanted to show.
Similarly on the right we find the slope $S(1)$ given by \e{sgb} and equaling $\rho(s_0)+\rho(t_0)$.
\end{proof}

Since $\phi_T(0)=\rho(s_0)$ and $\phi_T(1)=\rho(t_0)$ for the above $T$, it follows that the two lines of the graph meet at the point $(\rho(s_0)/(\rho(s_0)+\rho(t_0)),0)$, making a $V$ shape. We obtain the next corollary.

\begin{cor} \la{vee}
Let $T$ be a leveling tree with base triple $(s_0,u_0,t_0)$. Then $\Phi_T$ is well-defined with \e{pph} and
\begin{equation} \la{rru}
  \Phi_T(x)=   \begin{cases}
                         -(\rho(s_0)+\rho(t_0))x + \rho(s_0) & \mbox{if \quad $0\lqs x \lqs \frac{\rho(s_0)}{\rho(s_0)+\rho(t_0)}$},  \\
                          (\rho(s_0)+\rho(t_0))x - \rho(s_0) & \mbox{if \quad $\frac{\rho(s_0)}{\rho(s_0)+\rho(t_0)} \lqs x \lqs 1$}.
                       \end{cases}
\end{equation}
\end{cor}

Our running example of the leveling  tree $T=(3,6-\sqrt{15},4)$ has its encoding function $\Phi_T$ shown on the left of Figure \ref{tg}. The corner is at $x=\rho(3)/(\rho(3)+\rho(4)) \approx 0.4222$. The right of the figure shows its lattice curve $\Gamma_T$ as two parallel lines. The corresponding wedge $\mathcal W_T$ will contain infinitely many lattice points.

\begin{cor} \la{exte}
For a rising, progressing or leveling tree, the unique extension of its encoding function $\phi$ over the rationals to a continuous function on $\R \cap [0,1]$ is given by $\Phi$  in Definition \ref{phid2}.
\end{cor}
\begin{proof}
Any continuous extension of $\phi$ must satisfy \e{pph}. With Corollary \ref{vee} and Theorem \ref{mthm} this definition is well-defined and gives a continuous function.
\end{proof}

\subsection{Uniqueness when $D=4$}

\begin{adef}\la{irr}
{\rm Let $T$ be a Markov tree. Consider its adjacent regions $s, t$ with $s>2$ and $t>2$. We call $T$ {\em rational} if  $\rho(s)/\rho(t)$ is  rational for all such pairs. It is {\em irrational} if $\rho(s)/\rho(t)$ is  irrational for all such pairs.}
\end{adef}

As we will see in Theorem \ref{class2}, rising trees with $D=4$ and leveling trees are the only possible non-constant Markov trees with  labels $\gqs 2$ and discriminator $D = 4$.

\begin{lemma} \la{ratw}
Let $T$ be a rising tree with $D=4$ or a leveling tree. Then $T$ is either rational or irrational.
\end{lemma}
\begin{proof}
If $T$ is rising then it has no $2$-regions. If $T$ is leveling then it has at most one by Proposition \ref{www}.
Starting with a pair of adjacent regions $s, t >2$, there is an adjacent region $r>2$   making a triple with $r,s,t$. A permutation of this must give an ordered triple, say $(t,s,r)$. The proof of Proposition \ref{zag2} then implies that $\rho(s)=\rho(t)+\rho(r)$. Hence $\rho(s)/\rho(t)$ and $\rho(r)/\rho(t)$ are both rational or they are both irrational. Repeating this argument for an expanding set of adjacent regions gives the result. 
\end{proof}

\begin{theorem} \la{thm-rat}
Let $T$ be a rising tree with $D=4$ or a leveling tree, so it is rational or irrational. Then 
\begin{enumerate}
\item[\bf (a)] $T$ has repeated numbers if and only if it is rational,
\item[\bf (b)] $T$ has uniqueness if and only if it is irrational.
\end{enumerate}
\end{theorem}
\begin{proof}
{\bf Part I.}
  In this part of the proof assume $T$ is rising with $D=4$. By Lemma \ref{tr4}, it is strictly rising. Part (a)  now follow directly from Proposition \ref{ris} and Lemma \ref{ratw}. If $T$ is irrational then its base triple has distinct components and Lemma \ref{uniq} implies it has uniqueness. If $T$ is rational then its base triple may have $s_0=t_0$. Moving outward one edge must give a triple with all components distinct and we take the rational subtree $T_1$ with this base.  Lemma \ref{uniq} applies to $T$ or $T_1$, showing it does not have uniqueness. This confirms the theorem for rising trees. 
  
 {\bf Part II.} In this part assume $T$ is leveling, with its structure is given in Proposition \ref{www}.  First take $T$ to be rational. Moving out from the $2$-region, we can find a triple to form the base of a rising tree $T_2$ contained in $T$. Since $T_2$ is rational, it has repeated numbers and counterexamples to uniqueness by Part I of the proof.

 Lastly, assume  $T$ is leveling and irrational. Suppose $T$ has two repeated numbers or two triples giving a counterexample to uniqueness. Descend from these  regions or triples, going against the positive edge directions, until a common vertex $v$ on the descending path is reached. The ordered triple surrounding $v$ can form the base of a rising tree $T_3$ that is  partially contained in $T$.  Part I now contradicts the existence of  repeated numbers and counterexamples to uniqueness since $T_3$ is irrational. This completes the proof.
\end{proof}

The analog of Corollary \ref{fin} for leveling trees is:

\begin{cor}
Let $T$ be a leveling tree. If it is rational then its minimal number is $2$ and there are only finitely many numbers below any given bound. If it is irrational then there is no minimal number and there are infinitely many numbers below any given bound  $B>2$.
\end{cor}

\section{Further results}

\subsection{A classification of real Markov topographs and trees} \la{rea}
We may now complete our classification of real Markov topographs. 
\begin{theorem} \la{class}
Every real Markov topograph $\mathcal G$
 with  labels in $(-\infty,-2] \cup [2,\infty)$ is  one of these distinct forms, or a sign variant:
\begin{enumerate}
  \item A topograph with a vertex well or an edge well. Hence $D<4$ and it has no $2$-regions.
  \item A topograph with all regions labeled $2$, so that $D=4$.
  \item A topograph with $D=4$ and at most  one  $2$-region. Fixing any edge $e$, it consist of a strictly rising tree going one way from $e$ and a leveling tree going the other way, (or two leveling trees if $e$ is adjacent to a $2$-region), as on the right of Figure \ref{topb}. 
       \begin{enumerate}
         \item If the topograph is rational then it has one  $2$-region and repeated numbers on every 
         subtree.
         \item  If it is irrational then it has no $2$-regions and no repeated numbers.
       \end{enumerate}
   \item A topograph with a negative vertex well. Hence $D\gqs 20$ and it has up to three regions with absolute value $2$: they must be adjacent to the well.
\end{enumerate}
\end{theorem}
\begin{proof}
The cases with discriminator $D\neq 4$ have already been covered in Corollary \ref{cor}, so assume that $\mathcal G$ has $D=4$. 
This topograph cannot contain  a negative vertex well since that would require $D\gqs 20$ by Lemma \ref{neg}. 
Taking a sign variant if necessary, we may suppose that $\mathcal G$ has all regions positive by Lemma \ref{neg2}.

Let $r,s,u,t$ be the regions surrounding some edge $e$, as in the middle of Figure \ref{root}. If $e$ has label $0$ then $u=st/2=r$ and $s=2$ or $t=2$. Also it is easy to show that all the regions surrounding a $2$-region are equal when $D=4$. If $s=t=2$ then we are in the constant topograph case (ii). Otherwise, the trees with bases $(s,u,t)$ and $(t,r,s)$ are both leveling and we are in case (iii).  If $e$ does not have label $0$ then we may suppose  $u>st/2>r$. The triples $(s,u,t)$ and $(t,r,s)$ form the bases of a strictly rising tree and a leveling tree, respectively, completing case (iii).
Applying Proposition \ref{www} and Theorem \ref{thm-rat} to the two trees based at edge $e$ gives parts (a) and (b).
\end{proof}

We may also say that a type (iii) topograph above has uniqueness if and only if it is irrational, thanks to Theorem \ref{thm-rat}.

\begin{theorem} \la{class2}
Every  
real Markov tree $T$ with 
labels $\gqs 2$ takes one of these distinct forms: it is either  rising, progressing, leveling or constant. In particular,   Markov trees with  labels $> 2$ and $D\neq 4$ are rising. 
\end{theorem}
\begin{proof}
Let $T$ have base $(s_0,u_0,t_0)$. 
Suppose first that $T$ has a $2$-region. If $s_0=u_0=t_0=2$ then $T$ is constant, with all regions $2$. If $s_0=t_0=2$ and $u_0>2$ then $T$ is progressing. If only one of $s_0, t_0$ is $2$, say $s_0$, then $T$ is leveling if $u_0=t_0$ and progressing if $u_0>t_0$. The remain case has $T$ containing an oriented triple of the form $(s,2,t)$. The neighboring regions of this $2$-region form an arithmetic progression and, since they are all contained in $T$ and supposed to be $\gqs 2$, we must have $s=t$. Then $D=2^2+s^2+s^2-2s\cdot s = 4$ and we may assume $s>2$. The triples adjacent to this $2$-region all have the form $(s,s^2-2,s)$, and  trees with this base are strictly rising, showing that all edges of $T$ are positively directed away from this $2$-region. In particular, the edge between $s_0$ and $t_0$ is directed away from $u_0$ and therefore $T$ is leveling by \e{stb}. 

Now suppose that all regions of $T$ are $>2$. Direct all edges positively. If there is an undirected edge or a non-root vertex with all edges directed outward then $T$ is a rising tree. Otherwise, by the climbing lemma,  every non-root vertex must have one edge directed in and two out. Descending from any vertex, (i.e. going against the edge directions), we either reach the root so that $T$ is strictly rising, or we descend forever. The latter case means $D=4$, as in the proof of Proposition \ref{real}, and that the edge between $s_0$ and $t_0$ is directed away from $u_0$ so that $T$ is leveling.
\end{proof}

\subsection{Lattice curves for topographs} \la{norm}

\SpecialCoor
\psset{griddots=5,subgriddiv=0,gridlabels=0pt}
\psset{xunit=0.28cm, yunit=0.28cm, runit=0.28cm}
\psset{linewidth=1pt}
\psset{dotsize=7pt 0,dotstyle=*}
\begin{figure}[!htb]
\centering

\psset{arrowscale=1.4,arrowinset=0.3,arrowlength=1.1}
\newrgbcolor{light}{0.8 0.8 1.0}
\newrgbcolor{darkbrown}{0.5 0.324219 0.257813}
\newrgbcolor{pale}{1 0.7 0.4}

\newrgbcolor{cya2}{0 0.9 0.9}

\psset{linecolor=darkbrown}


\begin{pspicture}(-24,-10)(24,14) 

\rput(-14,2){
\begin{pspicture}(-10,-10)(10,14) 

\psline(0,1)(0,3)


\psline(0,3)(1.82628,5.38006)(4.59792,6.52811)(7.59507,6.65897)(9.16566,6.00841)

\psline(7.59507,6.65897)(9.10299,7.44394)

\psline(4.59792,6.52811)(6.80975,8.55488)(8.43107,9.06608)

\psline(6.80975,8.55488)(7.46032,10.1255)

\psline(1.82628,5.38006)(2.21786,8.35439)(3.60311,11.0154)(4.95181,12.0503)

\psline(3.60311,11.0154)(3.67726,12.7138)

\psline(2.21786,8.35439)(1.56854,11.2833)(1.93649,12.943)

\psline(1.56854,11.2833)(0.53365,12.632)


\psline(0,3)(-1.82628,5.38006)(-2.21786,8.35439)(-1.56854,11.2833)(-0.53365,12.632)

\psline(-1.56854,11.2833)(-1.93649,12.943)

\psline(-2.21786,8.35439)(-3.60311,11.0154)(-3.67726,12.7138)

\psline(-3.60311,11.0154)(-4.95181,12.0503)

\psline(-1.82628,5.38006)(-4.59792,6.52811)(-6.80975,8.55488)(-7.46032,10.1255)

\psline(-6.80975,8.55488)(-8.43107,9.06608)

\psline(-4.59792,6.52811)(-7.59507,6.65897)(-9.10299,7.44394)

\psline(-7.59507,6.65897)(-9.16566,6.00841)


\psline(0,1)(1.82628,-1.38006)(4.59792,-2.52811)(7.59507,-2.65897)(9.16566,-2.00841)

\psline(7.59507,-2.65897)(9.10299,-3.44394)

\psline(4.59792,-2.52811)(6.80975,-4.55488)(8.43107,-5.06608)

\psline(6.80975,-4.55488)(7.46032,-6.1255)

\psline(1.82628,-1.38006)(2.21786,-4.35439)(3.60311,-7.0154)(4.95181,-8.0503)

\psline(3.60311,-7.0154)(3.67726,-8.7138)

\psline(2.21786,-4.35439)(1.56854,-7.2833)(1.93649,-8.943)

\psline(1.56854,-7.2833)(0.53365,-8.632)


\psline(0,1)(-1.82628,-1.38006)(-2.21786,-4.35439)(-1.56854,-7.2833)(-0.53365,-8.632)

\psline(-1.56854,-7.2833)(-1.93649,-8.943)

\psline(-2.21786,-4.35439)(-3.60311,-7.0154)(-3.67726,-8.7138)

\psline(-3.60311,-7.0154)(-4.95181,-8.0503)

\psline(-1.82628,-1.38006)(-4.59792,-2.52811)(-6.80975,-4.55488)(-7.46032,-6.1255)

\psline(-6.80975,-4.55488)(-8.43107,-5.06608)

\psline(-4.59792,-2.52811)(-7.59507,-2.65897)(-9.10299,-3.44394)

\psline(-7.59507,-2.65897)(-9.16566,-2.00841)


\rput(-4.33013,2){$\displaystyle \frac 10$}

\rput(4.33013,2){$\displaystyle \frac 01$}

\rput(0,7.){$\displaystyle \frac 11$}

\rput(4.26133,8.55347){$ \frac 12$}

\rput(-4.26133,8.55347){$ \frac 21$}

\rput(7.36956,7.67616){$\frac 13$}

\rput(2.60944,11.3287){$\frac 23$}

\rput(-2.60944,11.3287){$\frac 32$}

\rput(-7.36956,7.67616){$\frac 31$}



\rput(0,-3.){$\displaystyle \frac{-1}1$}

\rput(4.26133,-4.55347){$ \frac {-1}2$}

\rput(-4.26133,-4.55347){$ \frac {-2}1$}

\rput(7.36956,-3.67616){$\frac {-1}3$}

\rput(2.60944,-7.3287){$\frac {-2}3$}

\rput(-2.60944,-7.3287){$\frac {-3}2$}

\rput(-7.36956,-3.67616){$\frac {-3}1$}

\end{pspicture}}

\rput(14,2){
\begin{pspicture}(-10,-10)(10,14) 

\psset{linecolor=orange}

\psline(0,3)(1.82628,5.38006)(4.59792,6.52811)(7.59507,6.65897)(9.16566,6.00841)

\psline(7.59507,6.65897)(9.10299,7.44394)

\psline(4.59792,6.52811)(6.80975,8.55488)(8.43107,9.06608)

\psline(6.80975,8.55488)(7.46032,10.1255)

\psline(1.82628,5.38006)(2.21786,8.35439)(3.60311,11.0154)(4.95181,12.0503)

\psline(3.60311,11.0154)(3.67726,12.7138)

\psline(2.21786,8.35439)(1.56854,11.2833)(1.93649,12.943)

\psline(1.56854,11.2833)(0.53365,12.632)

\psset{linecolor=blue}

\psline(0,3)(-1.82628,5.38006)(-2.21786,8.35439)(-1.56854,11.2833)(-0.53365,12.632)

\psline(-1.56854,11.2833)(-1.93649,12.943)

\psline(-2.21786,8.35439)(-3.60311,11.0154)(-3.67726,12.7138)

\psline(-3.60311,11.0154)(-4.95181,12.0503)

\psline(-1.82628,5.38006)(-4.59792,6.52811)(-6.80975,8.55488)(-7.46032,10.1255)

\psline(-6.80975,8.55488)(-8.43107,9.06608)

\psline(-4.59792,6.52811)(-7.59507,6.65897)(-9.10299,7.44394)

\psline(-7.59507,6.65897)(-9.16566,6.00841)

\psset{linecolor=purple}

\psline(0,1)(1.82628,-1.38006)(4.59792,-2.52811)(7.59507,-2.65897)(9.16566,-2.00841)

\psline(7.59507,-2.65897)(9.10299,-3.44394)

\psline(4.59792,-2.52811)(6.80975,-4.55488)(8.43107,-5.06608)

\psline(6.80975,-4.55488)(7.46032,-6.1255)

\psline(1.82628,-1.38006)(2.21786,-4.35439)(3.60311,-7.0154)(4.95181,-8.0503)

\psline(3.60311,-7.0154)(3.67726,-8.7138)

\psline(2.21786,-4.35439)(1.56854,-7.2833)(1.93649,-8.943)

\psline(1.56854,-7.2833)(0.53365,-8.632)

\psset{linecolor=cya2}

\psline(0,1)(-1.82628,-1.38006)(-2.21786,-4.35439)(-1.56854,-7.2833)(-0.53365,-8.632)

\psline(-1.56854,-7.2833)(-1.93649,-8.943)

\psline(-2.21786,-4.35439)(-3.60311,-7.0154)(-3.67726,-8.7138)

\psline(-3.60311,-7.0154)(-4.95181,-8.0503)

\psline(-1.82628,-1.38006)(-4.59792,-2.52811)(-6.80975,-4.55488)(-7.46032,-6.1255)

\psline(-6.80975,-4.55488)(-8.43107,-5.06608)

\psline(-4.59792,-2.52811)(-7.59507,-2.65897)(-9.10299,-3.44394)

\psline(-7.59507,-2.65897)(-9.16566,-2.00841)

\psline[linecolor=darkbrown](0,0.95)(0,3.05)


\rput(-4.33013,2){$s$}

\rput(4.33013,2){$t$}

\rput(0,7.){$u$}

\rput(4.26133,8.55347){$_{ut-s}$}

\rput(7.88346, 12.4239){$T_1$}

\rput(-4.26133,8.55347){$_{su-t}$}

\rput(-7.88346, 12.4239){$T_0$}




\rput(0,-3.){$r$}

\rput(4.26133,-4.55347){$_{rt-s}$}

\rput(7.88346, -8.4239){$T_2$}

\rput(-4.26133,-4.55347){$_{rs-t}$}

\rput(-7.88346, -8.4239){$T_3$}


\end{pspicture}}

\end{pspicture}
\caption{The Farey topograph on the left may be used to index Markov topographs}
\label{fatop}
\end{figure}
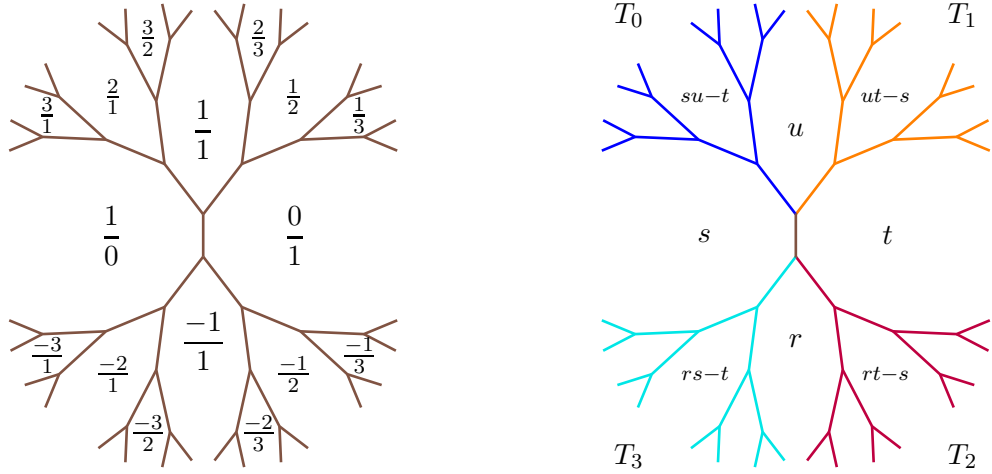

McShane and Rivin in \cite{mr95,mr95b} introduced a norm on the homology of  punctured tori and found that the boundary of the norm ball has interesting properties, being continuous, convex and having corners where it meets lines through the origin with rational slope. See also \cite{nxv} where these properties are studied in detail. We show next that these boundary curves may be broken into eight of the `lattice curves' we introduced in Definition \ref{enc}. In this way we may rederive the properties of these norm ball boundaries and also generalize them from the $D=0$ case of McShane and Rivin to all $D$.

The regions of a topograph can all be indexed by 
what we may call the {\em Farey topograph}, as seen on the left of Figure \ref{fatop}. It is the dual to the Farey diagram as in \cite[Sect. 4.1]{ha22}. The Farey topograph uniquely extends the Farey tree using the rule: for  indexing regions $k/n$, $k'/n'$ adjacent to an edge $e$, the other two regions meeting the ends of $e$ must have indices $(k+k')/(n+n')$ and $(k-k')/(n-n')$. Every pair of relatively prime integers $k, n$ appear as $k/n$ in the Farey topograph, including $1/0=-1/0$.

We say a topograph $\mathcal G$ is Farey-indexed if we have chosen a triple $(s,u,t)$  and overlaid the Farey topograph indices $(1/0,1/1,0/1)$  on these regions. So the topograph on the left of Figure \ref{fatop} gets placed on top of the topograph on the right.
Write the region labels as $m^*_{k/n}$, starting with 
\begin{equation}\label{tsg}
  m^*_{1/0}=s, \qquad m^*_{1/1}=u, \qquad  m^*_{0/1}=t,  \qquad m^*_{-1/1} =  r=st-u. 
\end{equation}
With this setup, consider the set of points
\begin{equation} \la{pg}
  P_{\mathcal G}:= \left\{\left( \frac{n}{\rho(|m^*_{k/n}|)}, \frac{k}{\rho(|m^*_{k/n}|)} \right)\ : \   k, n \in \Z, \ \gcd(k,n)=1\right\}.
\end{equation}
These points lie symmetrically  on lines through the origin with slope $k/n \in \P^1(\Q)$. We will show that, for $\mathcal G$ with  labels in $(-\infty,-2] \cup [2,\infty)$, $P_{\mathcal G}$  may be uniquely extended to a locally continuous polar curve $\G_{\mathcal G}$. Notice that the absolute value of labels appear in \e{pg}, so replacing $\mathcal G$ by a sign variant leaves  $P_{\mathcal G}$  unchanged.

The Farey topograph can be matched up with our usual Farey tree index using the trees $T_i$ for $0\lqs i\lqs 3$ shown in Figure \ref{fatop} (note that $T_1$ and $T_3$ are reflected):
\begin{equation} 
\begin{aligned} \label{r3}
   m^*_{k/n} & = m(T_0)_{n/k} \quad &\text{for} \quad T_0 & =(s,su-t,u)  \quad  & \text{when} \quad  & k  \gqs n \gqs 0,\\
  m^*_{k/n} & = m(T_1)_{k/n} \quad &\text{for} \quad T_1 & =(t,ut-s,u)  \quad  & \text{when} \quad  & 0  \lqs k\lqs n,\\
   m^*_{k/n} & = m(T_2)_{-k/n} \quad &\text{for} \quad T_2 & =(t,rt-s,r)  \quad  & \text{when} \quad  & 0  \lqs {-}k\lqs n, \\
    m^*_{k/n} & = m(T_3)_{n/-k} \quad &\text{for} \quad T_3 & =(s,rs-t,r)  \quad  & \text{when} \quad  & {-}k  \gqs n \gqs 0.
\end{aligned}
\end{equation}
Replace each $T_i$ with a sign variant, if necessary, to make  its labels $\gqs 2$. If this is possible then these trees are rising, progressing or leveling by Theorem \ref{class2} (assuming they are not constant) and we can write
\begin{equation} \la{pgp}
 P_{\mathcal G} = \pm \left\{(q, 1)/\phi_{T_0}(q), \ (1, q)/\phi_{T_1}(q), \ ({-}1,q)/\phi_{T_2}(q), \ (q,-1)/\phi_{T_3}(q) \ : \ q \in \Q\cap [0,1]\right\}. 
\end{equation}
Therefore, the curve $\G_{\mathcal G}$ we seek, containing $P_{\mathcal G}$,  is made up of eight components, with each a familiar lattice curve  $\Gamma_T$  under reflections and central symmetries.

\begin{adef}\la{enc2} 
{\rm The {\em  lattice curve} for a Farey-indexed Markov topograph $\mathcal G$ is denoted $\Gamma_{\mathcal G}$. In the above notation, if we have non-constant trees $T_i$ with labels $\gqs 2$, it is defined by
\begin{equation}\label{refl}
  \G_{\mathcal G} := \pm \left\{\operatorname{Refl}_{y=x} \G_{T_0} \cup  \G_{T_1} \cup \operatorname{Refl}_{y=0} \G_{T_2} \cup
  \operatorname{Refl}_{y=-x} (\operatorname{Refl}_{y=0}\G_{T_3})\right\}.
\end{equation}}
\end{adef}

\begin{theorem} \la{ball}
Let $\mathcal G$ be a Farey-indexed, type (i) topograph from Theorem \ref{class}, with an edge or vertex well and $D<4$. Then $P_{\mathcal G}$ extends uniquely to a  continuous, closed,  polar curve  about the origin. This curve is $\Gamma_{\mathcal G}$ and it is  strictly convex.
\end{theorem}
\begin{proof}
Replace $\mathcal G$ by a sign variant with all labels positive, if necessary. Then it has an edge or vertex well and the trees $T_i$ in \e{r3} are all  rising. Hence $P_{\mathcal G}$ is given by \e{pgp} and  Corollary \ref{exte} implies the functions $\phi_{T_i}$ extend uniquely to the continuous functions $\Phi_{T_i}$ on $\R \cap [0,1]$. Therefore $P_{\mathcal G}$ extends uniquely to the curve $\Gamma_{\mathcal G}$, made up of polar curves that fit together so that it is continuous. Each component is strictly convex by Lemma \ref{cfv} and it remains to verify that  their union is strictly convex. With Proposition \ref{tan} giving the slopes at the endpoints, this reduces to checking the inequalities
\begin{equation} 
\begin{aligned} \label{rti}
   S'_{T_0}(0)+ S'_{T_3}(0) & \gqs 0, &\qquad S'_{T_1}(0)+ S'_{T_2}(0) & \gqs 0, \\
  S_{T_0}(1)+ S_{T_1}(1) &\lqs \rho(u), &\qquad 
  S_{T_2}(1)+ S_{T_3}(1) &\lqs \rho(r),
\end{aligned}
\end{equation}
with the notation \e{tsg}, \e{r3} and \e{sg}, \e{sgb}. Note that, since the slopes can be of any sign or infinite, the inequalities \e{rti} were derived by first rotating the figure 
so that slopes of line segments on each side of a vertical line were being compared. The conditions \e{rti} can in fact all be shown with strict inequalities and we give the second, $S'_{T_1}(0)+ S'_{T_2}(0)  > 0$, as a representative example. This involves proving
\begin{equation}\label{wp}
  \left( \lambda_t(ut-s) -u\right)\left( \lambda_t(rt-s) -r\right) > \lambda_t^4(t^2-4).
\end{equation}
To eliminate $r$ from \e{wp}, write
\begin{equation*}
  \lambda_t(rt-s) -r = (t\lambda_t-1)r -s\lambda_t = r\lambda_t^2 - s\lambda_t = \lambda_t(\lambda_t(st-u)-s),
\end{equation*}
since $\lambda_t^2 - t\lambda_t+1 =0$ and $r=su-t$. To demonstrate
\begin{equation} \label{wp2}
  \left( \lambda_t(ut-s) -u\right)\left( \lambda_t(st-u)-s\right) > \lambda_t^3(t^2-4),
\end{equation}
expand and simplify with $s^2+u^2+t^2=sut+D$. Then \e{wp} and \e{wp2} are equivalent to $D<4$, and all the conditions in \e{rti} are equivalent to $D\lqs 4$.
\end{proof}

\SpecialCoor
\psset{griddots=5,subgriddiv=0,gridlabels=0pt}
\psset{xunit=2cm, yunit=2cm, runit=2cm}
\psset{linewidth=1pt}
\psset{dotsize=1.4pt,dotstyle=*}
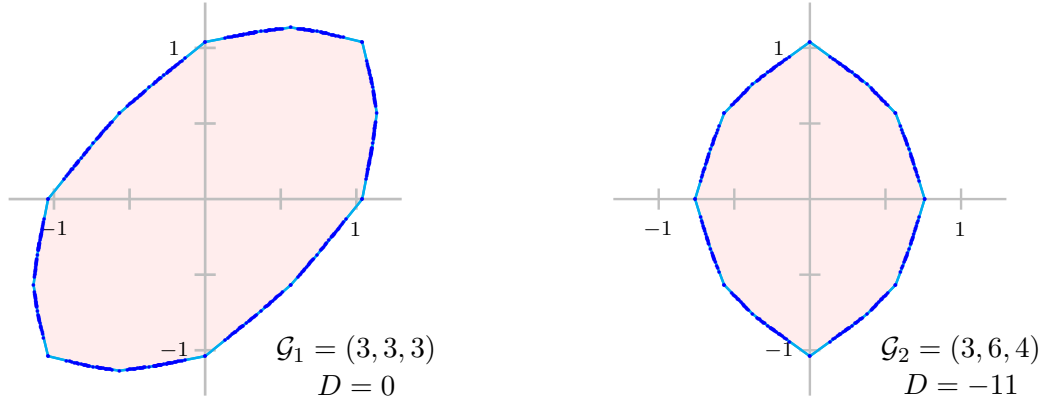
\begin{figure}[ht]
\centering

\psset{arrowscale=1.4,arrowinset=0.3,arrowlength=1.1}
\newrgbcolor{light}{0.8 0.8 1.0}
\newrgbcolor{pale}{1 0.7 1}
\newrgbcolor{pale}{1 0.7 0.4}

\newrgbcolor{markov}{0.4 0.2 0.8}
\newrgbcolor{bal}{1 0.93 0.93}
\newrgbcolor{balb}{0.984375 0.859375 0.984375}
\psset{linecolor=markov}

\begin{pspicture}(-3.1,-1.4)(3.1,1.3) 

\savedata{\mydatax}[
{{0, 1.039}, {0.2162, 1.081}, {0.2730, 1.092}, {0.3141, 
  1.100}, {0.3699, 1.110}, {0.4184, 1.116}, {0.4479, 1.120}, {0.4817, 
  1.124}, {0.5673, 1.135}, {0.6422, 1.124}, {0.6718, 1.120}, {0.6974, 
  1.116}, {0.7398, 1.110}, {0.7854, 1.100}, {0.8190, 1.092}, {1.039, 
  1.039}, {1.039, 1.039}, {1.092, 0.8190}, {1.100, 0.7854}, {1.110, 
  0.7398}, {1.116, 0.6974}, {1.120, 0.6718}, {1.124, 0.6422}, {1.135, 
  0.5673}, {1.124, 0.4817}, {1.120, 0.4479}, {1.116, 0.4184}, {1.110, 
  0.3699}, {1.100, 0.3141}, {1.092, 0.2730}, {1.081, 0.2162}, {1.039, 
  0}, {1.039, 
  0}, {0.8649, -0.2162}, {0.8190, -0.2730}, {0.7854, -0.3141}, 
{0.7398, -0.3699}, {0.6974, -0.4184}, {0.6718, -0.4479}, {0.5673, 
-0.5673}, {0.5673, -0.5673}, {0.4479, -0.6718}, {0.4184, -0.6974}, 
{0.3699, -0.7398}, {0.3141, -0.7854}, {0.2730, -0.8190}, {0.2162, 
-0.8649}, {0, -1.039}, {0, -1.039}, {-0.2162, -1.081}, {-0.2730, 
-1.092}, {-0.3141, -1.100}, {-0.3699, -1.110}, {-0.4184, -1.116}, 
{-0.4479, -1.120}, {-0.4817, -1.124}, {-0.5673, -1.135}, {-0.6422, 
-1.124}, {-0.6718, -1.120}, {-0.6974, -1.116}, {-0.7398, -1.110}, 
{-0.7854, -1.100}, {-0.8190, -1.092}, {-1.039, -1.039}, {-1.039, 
-1.039}, {-1.092, -0.8190}, {-1.100, -0.7854}, {-1.110, -0.7398}, 
{-1.116, -0.6974}, {-1.120, -0.6718}, {-1.124, -0.6422}, {-1.135, 
-0.5673}, {-1.124, -0.4817}, {-1.120, -0.4479}, {-1.116, -0.4184}, 
{-1.110, -0.3699}, {-1.100, -0.3141}, {-1.092, -0.2730}, {-1.081, 
-0.2162}, {-1.039, 0}, {-1.039, 0}, {-0.8649, 0.2162}, {-0.8190, 
  0.2730}, {-0.7854, 0.3141}, {-0.7398, 0.3699}, {-0.6974, 
  0.4184}, {-0.6718, 0.4479}, {-0.5673, 0.5673}, {-0.5673, 
  0.5673}, {-0.4479, 0.6718}, {-0.4184, 0.6974}, {-0.3699, 
  0.7398}, {-0.3141, 0.7854}, {-0.2730, 0.8190}, {-0.2162, 
  0.8649}, {0, 1.039}}
  ]

\savedata{\mydatab}[
{{0, 1.039}, {0.1331, 1.065}, {0.1527, 1.069}, {0.1648, 
  1.071}, {0.1790, 1.074}, {0.1899, 1.076}, {0.1958, 1.077}, {0.2022, 
  1.078}, {0.2162, 1.081}, {0.2281, 1.083}, {0.2323, 1.084}, {0.2358, 
  1.085}, {0.2413, 1.086}, {0.2471, 1.087}, {0.2510, 1.088}, {0.2562, 
  1.089}, {0.2730, 1.092}, {0.2881, 1.095}, {0.2921, 1.096}, {0.2951, 
  1.096}, {0.2991, 1.097}, {0.3027, 1.097}, {0.3050, 1.098}, {0.3075, 
  1.098}, {0.3141, 1.100}, {0.3211, 1.101}, {0.3239, 1.101}, {0.3264, 
  1.102}, {0.3308, 1.103}, {0.3358, 1.103}, {0.3397, 1.104}, {0.3699, 
  1.110}, {0.3976, 1.113}, {0.4051, 1.114}, {0.4107, 1.115}, {0.4184, 
  1.116}, {0.4254, 1.117}, {0.4297, 1.117}, {0.4348, 1.118}, {0.4479, 
  1.120}, {0.4617, 1.121}, {0.4675, 1.122}, {0.4727, 1.123}, {0.4817, 
  1.124}, {0.4923, 1.125}, {0.5005, 1.126}, {0.5126, 1.128}, {0.5673, 
  1.135}, {0.6151, 1.128}, {0.6257, 1.126}, {0.6329, 1.125}, {0.6422, 
  1.124}, {0.6500, 1.123}, {0.6546, 1.122}, {0.6596, 1.121}, {0.6718, 
  1.120}, {0.6832, 1.118}, {0.6876, 1.117}, {0.6913, 1.117}, {0.6974, 
  1.116}, {0.7041, 1.115}, {0.7090, 1.114}, {0.7156, 1.113}, {0.7398, 
  1.110}, {0.7644, 1.104}, {0.7718, 1.103}, {0.7774, 1.101}, {0.7854, 
  1.100}, {0.7929, 1.098}, {0.7977, 1.097}, {0.8034, 1.096}, {0.8190, 
  1.092}, {0.8368, 1.088}, {0.8446, 1.086}, {0.8519, 1.084}, {0.8649, 
  1.081}, {0.8813, 1.077}, {0.8949, 1.074}, {0.9161, 1.069}, {1.039, 
  1.039}, {1.039, 1.039}, {1.069, 0.9161}, {1.074, 0.8949}, {1.077, 
  0.8813}, {1.081, 0.8649}, {1.084, 0.8519}, {1.086, 0.8446}, {1.088, 
  0.8368}, {1.092, 0.8190}, {1.096, 0.8034}, {1.097, 0.7977}, {1.098, 
  0.7929}, {1.100, 0.7854}, {1.101, 0.7774}, {1.103, 0.7718}, {1.104, 
  0.7644}, {1.110, 0.7398}, {1.113, 0.7156}, {1.114, 0.7090}, {1.115, 
  0.7041}, {1.116, 0.6974}, {1.117, 0.6913}, {1.117, 0.6876}, {1.118, 
  0.6832}, {1.120, 0.6718}, {1.121, 0.6596}, {1.122, 0.6546}, {1.123, 
  0.6500}, {1.124, 0.6422}, {1.125, 0.6329}, {1.126, 0.6257}, {1.128, 
  0.6151}, {1.135, 0.5673}, {1.128, 0.5126}, {1.126, 0.5005}, {1.125, 
  0.4923}, {1.124, 0.4817}, {1.123, 0.4727}, {1.122, 0.4675}, {1.121, 
  0.4617}, {1.120, 0.4479}, {1.118, 0.4348}, {1.117, 0.4297}, {1.117, 
  0.4254}, {1.116, 0.4184}, {1.115, 0.4107}, {1.114, 0.4051}, {1.113, 
  0.3976}, {1.110, 0.3699}, {1.104, 0.3397}, {1.103, 0.3358}, {1.103, 
  0.3308}, {1.102, 0.3264}, {1.101, 0.3239}, {1.101, 0.3211}, {1.100, 
  0.3141}, {1.098, 0.3075}, {1.098, 0.3050}, {1.097, 0.3027}, {1.097, 
  0.2991}, {1.096, 0.2951}, {1.096, 0.2921}, {1.095, 0.2881}, {1.092, 
  0.2730}, {1.089, 0.2562}, {1.088, 0.2510}, {1.087, 0.2471}, {1.086, 
  0.2413}, {1.085, 0.2358}, {1.084, 0.2323}, {1.083, 0.2281}, {1.081, 
  0.2162}, {1.078, 0.2022}, {1.077, 0.1958}, {1.076, 0.1899}, {1.074, 
  0.1790}, {1.071, 0.1648}, {1.069, 0.1527}, {1.065, 0.1331}, {1.039, 
  0}, {1.039, 
  0}, {0.9318, -0.1331}, {0.9161, -0.1527}, {0.9063, -0.1648}, 
{0.8949, -0.1790}, {0.8861, -0.1899}, {0.8813, -0.1958}, {0.8762, 
-0.2022}, {0.8649, -0.2162}, {0.8553, -0.2281}, {0.8519, -0.2323}, 
{0.8490, -0.2358}, {0.8446, -0.2413}, {0.8400, -0.2471}, {0.8368, 
-0.2510}, {0.8190, -0.2730}, {0.8034, -0.2921}, {0.7977, -0.2991}, 
{0.7929, -0.3050}, {0.7854, -0.3141}, {0.7774, -0.3239}, {0.7718, 
-0.3308}, {0.7644, -0.3397}, {0.7398, -0.3699}, {0.7156, -0.3976}, 
{0.7090, -0.4051}, {0.7041, -0.4107}, {0.6974, -0.4184}, {0.6913, 
-0.4254}, {0.6876, -0.4297}, {0.6832, -0.4348}, {0.6718, -0.4479}, 
{0.6596, -0.4617}, {0.6546, -0.4675}, {0.6500, -0.4727}, {0.6422, 
-0.4817}, {0.6329, -0.4923}, {0.6257, -0.5005}, {0.6151, -0.5126}, 
{0.5673, -0.5673}, {0.5673, -0.5673}, {0.5126, -0.6151}, {0.5005, 
-0.6257}, {0.4923, -0.6329}, {0.4817, -0.6422}, {0.4727, -0.6500}, 
{0.4675, -0.6546}, {0.4617, -0.6596}, {0.4479, -0.6718}, {0.4348, 
-0.6832}, {0.4297, -0.6876}, {0.4254, -0.6913}, {0.4184, -0.6974}, 
{0.4107, -0.7041}, {0.4051, -0.7090}, {0.3976, -0.7156}, {0.3699, 
-0.7398}, {0.3397, -0.7644}, {0.3308, -0.7718}, {0.3239, -0.7774}, 
{0.3141, -0.7854}, {0.3050, -0.7929}, {0.2991, -0.7977}, {0.2921, 
-0.8034}, {0.2730, -0.8190}, {0.2510, -0.8368}, {0.2471, -0.8400}, 
{0.2413, -0.8446}, {0.2358, -0.8490}, {0.2323, -0.8519}, {0.2281, 
-0.8553}, {0.2162, -0.8649}, {0.2022, -0.8762}, {0.1958, -0.8813}, 
{0.1899, -0.8861}, {0.1790, -0.8949}, {0.1648, -0.9063}, {0.1527, 
-0.9161}, {0.1331, -0.9318}, {0, -1.039}, {0, -1.039}, {-0.1331, 
-1.065}, {-0.1527, -1.069}, {-0.1648, -1.071}, {-0.1790, -1.074}, 
{-0.1899, -1.076}, {-0.1958, -1.077}, {-0.2022, -1.078}, {-0.2162, 
-1.081}, {-0.2281, -1.083}, {-0.2323, -1.084}, {-0.2358, -1.085}, 
{-0.2413, -1.086}, {-0.2471, -1.087}, {-0.2510, -1.088}, {-0.2562, 
-1.089}, {-0.2730, -1.092}, {-0.2881, -1.095}, {-0.2921, -1.096}, 
{-0.2951, -1.096}, {-0.2991, -1.097}, {-0.3027, -1.097}, {-0.3050, 
-1.098}, {-0.3075, -1.098}, {-0.3141, -1.100}, {-0.3211, -1.101}, 
{-0.3239, -1.101}, {-0.3264, -1.102}, {-0.3308, -1.103}, {-0.3358, 
-1.103}, {-0.3397, -1.104}, {-0.3699, -1.110}, {-0.3976, -1.113}, 
{-0.4051, -1.114}, {-0.4107, -1.115}, {-0.4184, -1.116}, {-0.4254, 
-1.117}, {-0.4297, -1.117}, {-0.4348, -1.118}, {-0.4479, -1.120}, 
{-0.4617, -1.121}, {-0.4675, -1.122}, {-0.4727, -1.123}, {-0.4817, 
-1.124}, {-0.4923, -1.125}, {-0.5005, -1.126}, {-0.5126, -1.128}, 
{-0.5673, -1.135}, {-0.6151, -1.128}, {-0.6257, -1.126}, {-0.6329, 
-1.125}, {-0.6422, -1.124}, {-0.6500, -1.123}, {-0.6546, -1.122}, 
{-0.6596, -1.121}, {-0.6718, -1.120}, {-0.6832, -1.118}, {-0.6876, 
-1.117}, {-0.6913, -1.117}, {-0.6974, -1.116}, {-0.7041, -1.115}, 
{-0.7090, -1.114}, {-0.7156, -1.113}, {-0.7398, -1.110}, {-0.7644, 
-1.104}, {-0.7718, -1.103}, {-0.7774, -1.101}, {-0.7854, -1.100}, 
{-0.7929, -1.098}, {-0.7977, -1.097}, {-0.8034, -1.096}, {-0.8190, 
-1.092}, {-0.8368, -1.088}, {-0.8446, -1.086}, {-0.8519, -1.084}, 
{-0.8649, -1.081}, {-0.8813, -1.077}, {-0.8949, -1.074}, {-0.9161, 
-1.069}, {-1.039, -1.039}, {-1.039, -1.039}, {-1.069, -0.9161}, 
{-1.074, -0.8949}, {-1.077, -0.8813}, {-1.081, -0.8649}, {-1.084, 
-0.8519}, {-1.086, -0.8446}, {-1.088, -0.8368}, {-1.092, -0.8190}, 
{-1.096, -0.8034}, {-1.097, -0.7977}, {-1.098, -0.7929}, {-1.100, 
-0.7854}, {-1.101, -0.7774}, {-1.103, -0.7718}, {-1.104, -0.7644}, 
{-1.110, -0.7398}, {-1.113, -0.7156}, {-1.114, -0.7090}, {-1.115, 
-0.7041}, {-1.116, -0.6974}, {-1.117, -0.6913}, {-1.117, -0.6876}, 
{-1.118, -0.6832}, {-1.120, -0.6718}, {-1.121, -0.6596}, {-1.122, 
-0.6546}, {-1.123, -0.6500}, {-1.124, -0.6422}, {-1.125, -0.6329}, 
{-1.126, -0.6257}, {-1.128, -0.6151}, {-1.135, -0.5673}, {-1.128, 
-0.5126}, {-1.126, -0.5005}, {-1.125, -0.4923}, {-1.124, -0.4817}, 
{-1.123, -0.4727}, {-1.122, -0.4675}, {-1.121, -0.4617}, {-1.120, 
-0.4479}, {-1.118, -0.4348}, {-1.117, -0.4297}, {-1.117, -0.4254}, 
{-1.116, -0.4184}, {-1.115, -0.4107}, {-1.114, -0.4051}, {-1.113, 
-0.3976}, {-1.110, -0.3699}, {-1.104, -0.3397}, {-1.103, -0.3358}, 
{-1.103, -0.3308}, {-1.102, -0.3264}, {-1.101, -0.3239}, {-1.101, 
-0.3211}, {-1.100, -0.3141}, {-1.098, -0.3075}, {-1.098, -0.3050}, 
{-1.097, -0.3027}, {-1.097, -0.2991}, {-1.096, -0.2951}, {-1.096, 
-0.2921}, {-1.095, -0.2881}, {-1.092, -0.2730}, {-1.089, -0.2562}, 
{-1.088, -0.2510}, {-1.087, -0.2471}, {-1.086, -0.2413}, {-1.085, 
-0.2358}, {-1.084, -0.2323}, {-1.083, -0.2281}, {-1.081, -0.2162}, 
{-1.078, -0.2022}, {-1.077, -0.1958}, {-1.076, -0.1899}, {-1.074, 
-0.1790}, {-1.071, -0.1648}, {-1.069, -0.1527}, {-1.065, -0.1331}, 
{-1.039, 0}, {-1.039, 0}, {-0.9318, 0.1331}, {-0.9161, 
  0.1527}, {-0.9063, 0.1648}, {-0.8949, 0.1790}, {-0.8861, 
  0.1899}, {-0.8813, 0.1958}, {-0.8762, 0.2022}, {-0.8649, 
  0.2162}, {-0.8553, 0.2281}, {-0.8519, 0.2323}, {-0.8490, 
  0.2358}, {-0.8446, 0.2413}, {-0.8400, 0.2471}, {-0.8368, 
  0.2510}, {-0.8190, 0.2730}, {-0.8034, 0.2921}, {-0.7977, 
  0.2991}, {-0.7929, 0.3050}, {-0.7854, 0.3141}, {-0.7774, 
  0.3239}, {-0.7718, 0.3308}, {-0.7644, 0.3397}, {-0.7398, 
  0.3699}, {-0.7156, 0.3976}, {-0.7090, 0.4051}, {-0.7041, 
  0.4107}, {-0.6974, 0.4184}, {-0.6913, 0.4254}, {-0.6876, 
  0.4297}, {-0.6832, 0.4348}, {-0.6718, 0.4479}, {-0.6596, 
  0.4617}, {-0.6546, 0.4675}, {-0.6500, 0.4727}, {-0.6422, 
  0.4817}, {-0.6329, 0.4923}, {-0.6257, 0.5005}, {-0.6151, 
  0.5126}, {-0.5673, 0.5673}, {-0.5673, 0.5673}, {-0.5126, 
  0.6151}, {-0.5005, 0.6257}, {-0.4923, 0.6329}, {-0.4817, 
  0.6422}, {-0.4727, 0.6500}, {-0.4675, 0.6546}, {-0.4617, 
  0.6596}, {-0.4479, 0.6718}, {-0.4348, 0.6832}, {-0.4297, 
  0.6876}, {-0.4254, 0.6913}, {-0.4184, 0.6974}, {-0.4107, 
  0.7041}, {-0.4051, 0.7090}, {-0.3976, 0.7156}, {-0.3699, 
  0.7398}, {-0.3397, 0.7644}, {-0.3308, 0.7718}, {-0.3239, 
  0.7774}, {-0.3141, 0.7854}, {-0.3050, 0.7929}, {-0.2991, 
  0.7977}, {-0.2921, 0.8034}, {-0.2730, 0.8190}, {-0.2510, 
  0.8368}, {-0.2471, 0.8400}, {-0.2413, 0.8446}, {-0.2358, 
  0.8490}, {-0.2323, 0.8519}, {-0.2281, 0.8553}, {-0.2162, 
  0.8649}, {-0.2022, 0.8762}, {-0.1958, 0.8813}, {-0.1899, 
  0.8861}, {-0.1790, 0.8949}, {-0.1648, 0.9063}, {-0.1527, 
  0.9161}, {-0.1331, 0.9318}, {0, 1.039}}
]

\savedata{\mydatac}[
{{0, 1.039}, {0.2201, 0.8804}, {0.2791, 0.8373}, {0.3217, 
  0.8043}, {0.3797, 0.7593}, {0.4269, 0.7114}, {0.4551, 
  0.6827}, {0.5673, 0.5673}, {0.5673, 0.5673}, {0.6245, 
  0.4163}, {0.6574, 0.3287}, {0.6886, 0.2295}, {0.7593, 0}, {0.7593, 
  0}, {0.6886, -0.2295}, {0.6574, -0.3287}, {0.6245, -0.4163}, 
{0.5673, -0.5673}, {0.5673, -0.5673}, {0.4551, -0.6827}, {0.4269, 
-0.7114}, {0.3797, -0.7593}, {0.3217, -0.8043}, {0.2791, -0.8373}, 
{0.2201, -0.8804}, {0, -1.039}, {0, -1.039}, {-0.2201, -0.8804}, 
{-0.2791, -0.8373}, {-0.3217, -0.8043}, {-0.3797, -0.7593}, {-0.4269, 
-0.7114}, {-0.4551, -0.6827}, {-0.5673, -0.5673}, {-0.5673, -0.5673}, 
{-0.6245, -0.4163}, {-0.6574, -0.3287}, {-0.6886, -0.2295}, {-0.7593, 
  0}, {-0.7593, 0}, {-0.6886, 0.2295}, {-0.6574, 0.3287}, {-0.6245, 
  0.4163}, {-0.5673, 0.5673}, {-0.5673, 0.5673}, {-0.4551, 
  0.6827}, {-0.4269, 0.7114}, {-0.3797, 0.7593}, {-0.3217, 
  0.8043}, {-0.2791, 0.8373}, {-0.2201, 0.8804}, {0, 1.039}}
  ]

\savedata{\mydatad}[
{{0, 1.039}, {0.1346, 0.9421}, {0.1546, 0.9277}, {0.1670, 
  0.9187}, {0.1816, 0.9082}, {0.1929, 0.9001}, {0.1990, 
  0.8957}, {0.2056, 0.8909}, {0.2201, 0.8804}, {0.2324, 
  0.8715}, {0.2368, 0.8683}, {0.2404, 0.8656}, {0.2461, 
  0.8614}, {0.2521, 0.8571}, {0.2562, 0.8541}, {0.2616, 
  0.8501}, {0.2791, 0.8373}, {0.2989, 0.8220}, {0.3061, 
  0.8164}, {0.3122, 0.8117}, {0.3217, 0.8043}, {0.3318, 
  0.7964}, {0.3390, 0.7909}, {0.3483, 0.7837}, {0.3797, 
  0.7593}, {0.4066, 0.7319}, {0.4140, 0.7245}, {0.4194, 
  0.7190}, {0.4269, 0.7114}, {0.4336, 0.7046}, {0.4377, 
  0.7004}, {0.4426, 0.6955}, {0.4551, 0.6827}, {0.4684, 
  0.6691}, {0.4739, 0.6634}, {0.4788, 0.6584}, {0.4873, 
  0.6497}, {0.4973, 0.6394}, {0.5051, 0.6313}, {0.5164, 
  0.6197}, {0.5673, 0.5673}, {0.5673, 0.5673}, {0.6003, 
  0.4802}, {0.6091, 0.4569}, {0.6156, 0.4397}, {0.6245, 
  0.4163}, {0.6302, 0.4011}, {0.6324, 0.3953}, {0.6343, 
  0.3903}, {0.6373, 0.3824}, {0.6405, 0.3736}, {0.6429, 
  0.3674}, {0.6461, 0.3589}, {0.6574, 0.3287}, {0.6675, 
  0.2967}, {0.6704, 0.2873}, {0.6727, 0.2803}, {0.6758, 
  0.2703}, {0.6787, 0.2610}, {0.6805, 0.2552}, {0.6827, 
  0.2483}, {0.6886, 0.2295}, {0.6951, 0.2085}, {0.6979, 
  0.1994}, {0.7005, 0.1910}, {0.7051, 0.1763}, {0.7107, 
  0.1579}, {0.7153, 0.1431}, {0.7223, 0.1204}, {0.7593, 0}, {0.7593, 
  0}, {0.7223, -0.1204}, {0.7153, -0.1431}, {0.7107, -0.1579}, 
{0.7051, -0.1763}, {0.7005, -0.1910}, {0.6979, -0.1994}, {0.6951, 
-0.2085}, {0.6886, -0.2295}, {0.6827, -0.2483}, {0.6805, -0.2552}, 
{0.6787, -0.2610}, {0.6758, -0.2703}, {0.6727, -0.2803}, {0.6704, 
-0.2873}, {0.6675, -0.2967}, {0.6574, -0.3287}, {0.6461, -0.3589}, 
{0.6429, -0.3674}, {0.6405, -0.3736}, {0.6373, -0.3824}, {0.6343, 
-0.3903}, {0.6324, -0.3953}, {0.6302, -0.4011}, {0.6245, -0.4163}, 
{0.6156, -0.4397}, {0.6091, -0.4569}, {0.6003, -0.4802}, {0.5673, 
-0.5673}, {0.5673, -0.5673}, {0.5164, -0.6197}, {0.5051, -0.6313}, 
{0.4973, -0.6394}, {0.4873, -0.6497}, {0.4788, -0.6584}, {0.4739, 
-0.6634}, {0.4684, -0.6691}, {0.4551, -0.6827}, {0.4426, -0.6955}, 
{0.4377, -0.7004}, {0.4336, -0.7046}, {0.4269, -0.7114}, {0.4194, 
-0.7190}, {0.4140, -0.7245}, {0.4066, -0.7319}, {0.3797, -0.7593}, 
{0.3483, -0.7837}, {0.3390, -0.7909}, {0.3318, -0.7964}, {0.3217, 
-0.8043}, {0.3122, -0.8117}, {0.3061, -0.8164}, {0.2989, -0.8220}, 
{0.2791, -0.8373}, {0.2616, -0.8501}, {0.2562, -0.8541}, {0.2521, 
-0.8571}, {0.2461, -0.8614}, {0.2404, -0.8656}, {0.2368, -0.8683}, 
{0.2324, -0.8715}, {0.2201, -0.8804}, {0.2056, -0.8909}, {0.1990, 
-0.8957}, {0.1929, -0.9001}, {0.1816, -0.9082}, {0.1670, -0.9187}, 
{0.1546, -0.9277}, {0.1346, -0.9421}, {0, -1.039}, {0, -1.039}, 
{-0.1346, -0.9421}, {-0.1546, -0.9277}, {-0.1670, -0.9187}, {-0.1816, 
-0.9082}, {-0.1929, -0.9001}, {-0.1990, -0.8957}, {-0.2056, -0.8909}, 
{-0.2201, -0.8804}, {-0.2324, -0.8715}, {-0.2368, -0.8683}, {-0.2404, 
-0.8656}, {-0.2461, -0.8614}, {-0.2521, -0.8571}, {-0.2562, -0.8541}, 
{-0.2616, -0.8501}, {-0.2791, -0.8373}, {-0.2989, -0.8220}, {-0.3061, 
-0.8164}, {-0.3122, -0.8117}, {-0.3217, -0.8043}, {-0.3318, -0.7964}, 
{-0.3390, -0.7909}, {-0.3483, -0.7837}, {-0.3797, -0.7593}, {-0.4066, 
-0.7319}, {-0.4140, -0.7245}, {-0.4194, -0.7190}, {-0.4269, -0.7114}, 
{-0.4336, -0.7046}, {-0.4377, -0.7004}, {-0.4426, -0.6955}, {-0.4551, 
-0.6827}, {-0.4684, -0.6691}, {-0.4739, -0.6634}, {-0.4788, -0.6584}, 
{-0.4873, -0.6497}, {-0.4973, -0.6394}, {-0.5051, -0.6313}, {-0.5164, 
-0.6197}, {-0.5673, -0.5673}, {-0.5673, -0.5673}, {-0.6003, -0.4802}, 
{-0.6091, -0.4569}, {-0.6156, -0.4397}, {-0.6245, -0.4163}, {-0.6302, 
-0.4011}, {-0.6324, -0.3953}, {-0.6343, -0.3903}, {-0.6373, -0.3824}, 
{-0.6405, -0.3736}, {-0.6429, -0.3674}, {-0.6461, -0.3589}, {-0.6574, 
-0.3287}, {-0.6675, -0.2967}, {-0.6704, -0.2873}, {-0.6727, -0.2803}, 
{-0.6758, -0.2703}, {-0.6787, -0.2610}, {-0.6805, -0.2552}, {-0.6827, 
-0.2483}, {-0.6886, -0.2295}, {-0.6951, -0.2085}, {-0.6979, -0.1994}, 
{-0.7005, -0.1910}, {-0.7051, -0.1763}, {-0.7107, -0.1579}, {-0.7153, 
-0.1431}, {-0.7223, -0.1204}, {-0.7593, 0}, {-0.7593, 0}, {-0.7223, 
  0.1204}, {-0.7153, 0.1431}, {-0.7107, 0.1579}, {-0.7051, 
  0.1763}, {-0.7005, 0.1910}, {-0.6979, 0.1994}, {-0.6951, 
  0.2085}, {-0.6886, 0.2295}, {-0.6827, 0.2483}, {-0.6805, 
  0.2552}, {-0.6787, 0.2610}, {-0.6758, 0.2703}, {-0.6727, 
  0.2803}, {-0.6704, 0.2873}, {-0.6675, 0.2967}, {-0.6574, 
  0.3287}, {-0.6461, 0.3589}, {-0.6429, 0.3674}, {-0.6405, 
  0.3736}, {-0.6373, 0.3824}, {-0.6343, 0.3903}, {-0.6324, 
  0.3953}, {-0.6302, 0.4011}, {-0.6245, 0.4163}, {-0.6156, 
  0.4397}, {-0.6091, 0.4569}, {-0.6003, 0.4802}, {-0.5673, 
  0.5673}, {-0.5673, 0.5673}, {-0.5164, 0.6197}, {-0.5051, 
  0.6313}, {-0.4973, 0.6394}, {-0.4873, 0.6497}, {-0.4788, 
  0.6584}, {-0.4739, 0.6634}, {-0.4684, 0.6691}, {-0.4551, 
  0.6827}, {-0.4426, 0.6955}, {-0.4377, 0.7004}, {-0.4336, 
  0.7046}, {-0.4269, 0.7114}, {-0.4194, 0.7190}, {-0.4140, 
  0.7245}, {-0.4066, 0.7319}, {-0.3797, 0.7593}, {-0.3483, 
  0.7837}, {-0.3390, 0.7909}, {-0.3318, 0.7964}, {-0.3217, 
  0.8043}, {-0.3122, 0.8117}, {-0.3061, 0.8164}, {-0.2989, 
  0.8220}, {-0.2791, 0.8373}, {-0.2616, 0.8501}, {-0.2562, 
  0.8541}, {-0.2521, 0.8571}, {-0.2461, 0.8614}, {-0.2404, 
  0.8656}, {-0.2368, 0.8683}, {-0.2324, 0.8715}, {-0.2201, 
  0.8804}, {-0.2056, 0.8909}, {-0.1990, 0.8957}, {-0.1929, 
  0.9001}, {-0.1816, 0.9082}, {-0.1670, 0.9187}, {-0.1546, 
  0.9277}, {-0.1346, 0.9421}, {0, 1.039}}
]

\rput(-2,0.0){
\begin{pspicture}(-1.3,-1.3)(1.3,1.3) 


\dataplot[linewidth=0.9pt,linecolor=cyan,fillstyle=solid,
fillcolor=bal]{\mydatax}

\psset{linecolor=lightgray}

\psline(-1.3,0)(1.3,0)
\psline(0,-1.3)(0,1.3)

\psline(1,-0.07)(1,0.07)
\psline(-1,-0.07)(-1,0.07)
\psline(-0.07,1)(0.07,1)
\psline(-0.07,-1)(0.07,-1)

\psline(0.5,-0.07)(0.5,0.07)
\psline(-0.5,-0.07)(-0.5,0.07)
\psline(-0.07,0.5)(0.07,0.5)
\psline(-0.07,-0.5)(0.07,-0.5)

\rput(1,-0.2){$_{1}$}
\rput(-1,-0.2){$_{-1}$}
\rput(-0.2,1){$_{1}$}
\rput(-0.2,-1){$_{-1}$}

\rput(1,-1){$\mathcal G_1 = (3,3,3)$}
\rput(1,-1.23){$D=0$}

\psset{dotsize=1.3pt,dotstyle=*}
\psset{linewidth=0.9pt}

\dataplot[linecolor=cyan]{\mydatax}

\dataplot[plotstyle=dots,linecolor=blue]{\mydatab}



\end{pspicture}}

\rput(2,0.0){
\begin{pspicture}(-1.3,-1.3)(1.3,1.3) 


\psset{dotsize=1.3pt,dotstyle=*}
\psset{linewidth=0.9pt}

\dataplot[linewidth=0.9pt,linecolor=cyan,fillstyle=solid,
fillcolor=bal]{\mydatac} 

\psset{linecolor=lightgray}

\psline(-1.3,0)(1.3,0)
\psline(0,-1.3)(0,1.3)

\psline(1,-0.07)(1,0.07)
\psline(-1,-0.07)(-1,0.07)
\psline(-0.07,1)(0.07,1)
\psline(-0.07,-1)(0.07,-1)

\psline(0.5,-0.07)(0.5,0.07)
\psline(-0.5,-0.07)(-0.5,0.07)
\psline(-0.07,0.5)(0.07,0.5)
\psline(-0.07,-0.5)(0.07,-0.5)

\rput(1,-0.2){$_{1}$}
\rput(-1,-0.2){$_{-1}$}
\rput(-0.2,1){$_{1}$}
\rput(-0.2,-1){$_{-1}$}

\rput(1,-1){$\mathcal G_2 = (3,6,4)$}
\rput(1,-1.23){$D=-11$}

\dataplot[linecolor=cyan]{\mydatac}

\dataplot[plotstyle=dots,linecolor=blue]{\mydatad}

\psset{dotsize=1.3pt,dotstyle=*}
\psset{linewidth=0.9pt}



\end{pspicture}}

\end{pspicture}
\caption{Lattice curves for the topographs $\mathcal G_1$ and $\mathcal G_2$, enclosing the balls $\mathcal B(\mathcal G_1)$ and $\mathcal B(\mathcal G_2)$. In general these polar curves are continuous and convex for $D<4$.}
\label{balls}
\end{figure}

The $D=0$ case of Theorem \ref{ball} corresponds to the McShane-Rivin norm ball boundary. The left side of Figure \ref{balls} shows the curve for the original Markov topograph whose associated torus is known as the modular torus. This curve first appears in a diagram in \cite{mr95b}; note that our normalization has an extra $2$ so that our lattice curves have twice the radius.
On the right of Figure \ref{balls} is the lattice curve for the topograph  $\mathcal G_2 = (3,6,4)$ which has an edge well.

The constant type (ii) topograph from Theorem \ref{class} has a lattice curve  that may be thought of as a circle at infinity. 

\begin{theorem} \la{ball2}
Let $\mathcal G$ be a Farey-indexed type (iii) `leveling' topograph from Theorem \ref{class}, with $D=4$. Then $P_{\mathcal G}$ extends uniquely to  $\Gamma_{\mathcal G}$  which is a  pair of parallel lines on each side of the origin. 
\end{theorem}
\begin{proof}
With \e{tsg}, \e{r3}, suppose for example that $T_1$ is a leveling tree with $t$ and $u$ not $2$. Computing $\Phi_{T_1}$ with Corollary \ref{vee} and applying $F$ with \e{f2}, gives $R_{T_1}$ as the parts of the two lines $y=(\rho(t) x \pm 1)/(\rho(t)+\rho(u))$ between $y=0$ and $y=x$ in the first quadrant; see the right side of Figure \ref{tg}. With a central symmetry, this produces two of the eight parts of $\G_{\mathcal G}$ in \e{refl}. The trees $T_0$, $T_2$ and $T_3$ are strictly rising with their $\Phi$  functions linear and given in \e{line}. Applying $F$ and reflections confirms that these contribute the remaining pieces of the two lines. The same is true in general. If any one of the labels \e{tsg} equals $2$ then this corresponds to the lines being parallel to the axes or having slope $\pm 1$.
\end{proof}

The next result gives another approach to the curves studied by McShane in \cite{mcs}. He showed the piecewise concavity of these curves using geometrical arguments -- instead of a punctured torus the corresponding surface is a three holed sphere with a hyperbolic structure,  a pair of pants in other words.

\SpecialCoor
\psset{griddots=5,subgriddiv=0,gridlabels=0pt}
\psset{xunit=1.6cm, yunit=1.6cm, runit=1.6cm}
\psset{linewidth=1pt}
\psset{dotsize=1.4pt,dotstyle=*}
\begin{figure}[htb]
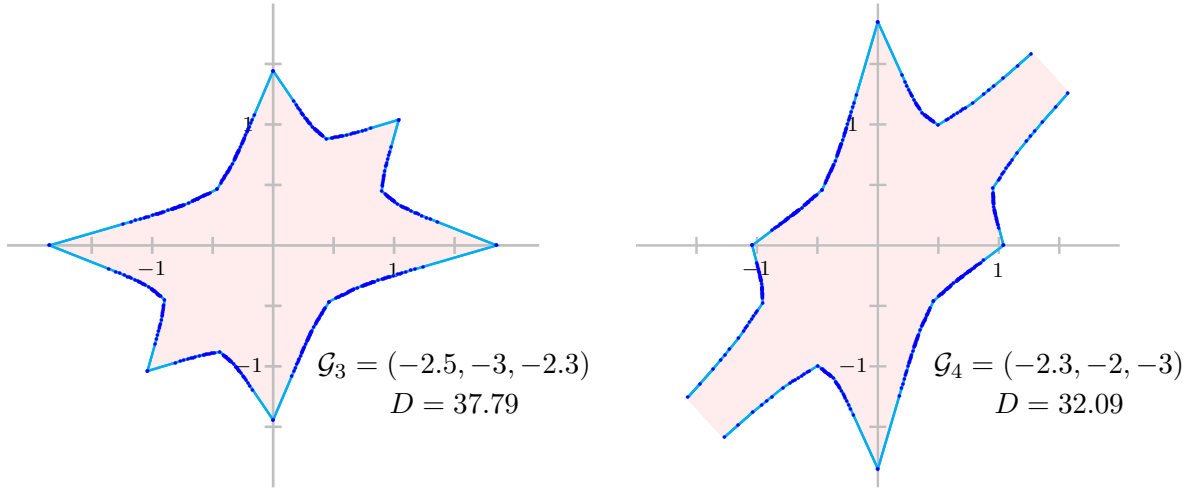

\centering

\psset{arrowscale=1.4,arrowinset=0.3,arrowlength=1.1}
\newrgbcolor{light}{0.8 0.8 1.0}
\newrgbcolor{pale}{1 0.7 1}
\newrgbcolor{pale}{1 0.7 0.4}

\newrgbcolor{markov}{0.4 0.2 0.8}
\newrgbcolor{bal}{1 0.93 0.93}
\newrgbcolor{balb}{0.984375 0.859375 0.984375}
\psset{linecolor=markov}

}

\end{pspicture}
\caption{Lattice curves for the topographs $\mathcal G_3$ and $\mathcal G_4$. The curve for $\mathcal G_4$ approaches infinity along the line with slope $1/1$ since $m^*_{1/1}=-2$. Apart from that, these polar curves are continuous. In general for $D\gqs 20$ we obtain piecewise concave  star shapes.}
\label{balls2}
\end{figure}

\begin{theorem} \la{ball3}
Let $\mathcal G$ be a  type (iv) topograph from Theorem \ref{class}, with a negative vertex well $v$ and $D\gqs 20$. Use the Farey indexing that has $v$   the central vertex between $0/1$, $1/1$ and $1/0$ in the Farey topograph. Then $P_{\mathcal G}$ extends uniquely to a  polar curve  about the origin. This curve is $\Gamma_{\mathcal G}$ with six piecewise concave parts. It is continuous except at slopes corresponding to a region $-2$ at $v$, where the curve goes to infinity.
\end{theorem}
\begin{proof}
With \e{tsg}, \e{r3}, our choice of $v$ has it surrounded by the triple $(s,u,t)$. Lemma \ref{neg} implies the three trees with root $v$ have sign variants that are rising or progressing. These are $T_0=(-s,su-t,-u)$, $T_1=(-t,ut-s,-u)$ and the third tree $(-t,r,-s)$ which splits into $T_2=(-t,s-rt,r)$ and $T_3=(-s,t-rs,r)$. Therefore $P_{\mathcal G}$ is given by \e{pgp} and, by Corollary \ref{exte}, extends uniquely to the curve $\Gamma_{\mathcal G}$. Since $D>4$, each component is strictly concave by Lemma \ref{cfv}. The pieces $\operatorname{Refl}_{y=0} \G_{T_2}$ and
  $\operatorname{Refl}_{y=x} (\operatorname{Refl}_{y=0}\G_{T_3})$ also fit together into a strictly concave union. This follows by checking the slopes where they meet: we require $S_{T_2}(1)+ S_{T_3}(1) \gqs \rho(r)$ and, as in the proof of Theorem \ref{ball}, this is equivalent to $D\gqs 4$. 
  Therefore $\Gamma_{\mathcal G}$ has six continuous piecewise concave components, between the lines through the origin of slopes $1/0, 1/1, 0/1$. If any of $s, u, t$ is $-2$ then the curve will go to infinity along each of these lines, respectively. In this case the corresponding tree is progressing. 
\end{proof}

Figure \ref{balls2} displays examples of these lattice curves with $D\gqs 20$. By 'Farey-indexed' we now specify the choice made in the statement of Theorem \ref{ball3} when the topograph has a negative vertex well.
We could also have chosen to place the well between $0/1$, $-1/1$ and $1/0$ in the Farey topograph, giving a reflected version of $\G_\mathcal G$.

The differentiability results we proved for $\Phi$ in Theorem \ref{mthm} and Section \ref{ite} give corresponding results for lattice curves under the transformation \e{f}. In particular:

\begin{cor}
The lattice curves $\G_{\mathcal G}$ for the topographs $\mathcal G$ in Theorems \ref{ball} and \ref{ball3} have distinct left and right tangents at points where $\G_{\mathcal G}$ meets lines through the origin with rational or infinite slope. They have (usual) tangents where they meet lines of irrational slope. 
\end{cor}


\subsection{Topograph numbers below a given bound}  \la{belo}
We have seen that all Farey-indexed Markov topographs $\mathcal G$
 with  labels in $(-\infty,-2] \cup [2,\infty)$, and classified by Theorem \ref{class}, have a well defined polar lattice curve $\G_{\mathcal G}$. Let $\mathcal B(\mathcal G)$ be the ball it encloses, including the boundary. 
Recall the wedges $\mathcal W$ for rising trees, with a method to calculate their areas in Proposition \ref{wedg}.

\begin{lemma} \la{few}
Let $\mathcal G$ be a Farey-indexed Markov topograph 
 with  labels in $(-\infty,-2) \cup (2,\infty)$ and $D\neq 4$. Then its lattice curve $\G_{\mathcal G}$ is a continuous polar curve enclosing a compact ball $\mathcal B(\mathcal G)$ with finite area given by
 \begin{equation}\label{wax}
   \text{\rm Area}(\mathcal B(\mathcal G)) = 2\sum_{0\lqs i \lqs 3} \text{\rm Area}(\mathcal W_{T_i}),
 \end{equation}
 where the $T_i$ are rising trees from \e{r3}.
\end{lemma}
\begin{proof}
We are in the cases of Theorems \ref{ball} and \ref{ball3}, though excluding labels $\pm 2$ so that $\G_{\mathcal G}$ remains bounded. Hence  the trees $T_i$ in \e{r3} are rising. Their wedges $\mathcal W_{T_i}$ make up $\mathcal B(\mathcal G)$ by \e{refl}.
\end{proof}

The analog of Lemma \ref{latt} is: 

\begin{lemma} \la{latt2}
The dilated set $t \cdot \G_{\mathcal G}$ for $t>0$ intersects a lattice point $(n,k) \in \Z^2$ with $\gcd(n,k)=1$ exactly when $|m^*(\mathcal G)_{k/n}|=2\cosh(t)$.
\end{lemma}
\begin{proof}
Relabel the lattice point as $(n_0,k_0)$ and suppose it intersects $t \cdot \G_{\mathcal G}$.
The intersection must happen on a line through the origin with slope in $\P^1(\Q)$. Hence 
$(n_0,k_0) \in t \cdot P_{\mathcal G}$ so that
$
  (n_0,k_0) =  \left( t n, t k \right)/\rho(|m^*_{k/n}|)$.
Since $\gcd(n_0,k_0) = \gcd(n,k)=1$ we must have $n_0=n$, $k_0=k$, $t=\rho(|m^*_{k/n}|)$ and therefore $|m^*_{k/n}|=2\cosh(t)$.
This reasoning may also be reversed.
\end{proof}

Let 
\begin{equation*}
  M(X)=M_{\mathcal G}(X) := \#\{ q\in \P^1(\Q) \, : \, |m^*_q| \lqs X \},
\end{equation*}
count (with their multiplicities) all numbers in $\mathcal G$ that are at most $X$ in absolute value. The following result generalizes  \cite[Thm. 4.1]{mr95b} of McShane and Rivin through the link of Theorem \ref{bowl}.

\begin{theorem} \la{zzzb}
Let $\mathcal G$ be a Farey-indexed Markov topograph 
 with  labels in $(-\infty,-2) \cup (2,\infty)$ and $D\neq 4$. Then
\begin{equation} \la{mzxb}
  M_{\mathcal G}(X) = \frac{3\text{\rm Area}(\mathcal B(\mathcal G))}{\pi^2}  (\log X)^2 +O\left(\log X \log \log X \right).
\end{equation}
\end{theorem}
\begin{proof}
By Lemma \ref{latt2}, $|m^*_{k/n}| \lqs X$ if and only if $(n,k) \in t \cdot \mathcal B(\mathcal G)$ for $t=\rho(X)$. Then $\mathcal B(\mathcal G)$ breaks into eight wedges and the proof of Theorem \ref{zzz} applies. We have double counted since $-k/{-}n=k/n$ and must divide by $2$ to obtain \e{mzxb}. Note that our use of the trees $T_i$ means the numbers $m^*_{1/0}$, $m^*_{1/1}$, $m^*_{0/1}$, $m^*_{-1/1}$ contribute twice to the count. This adds an extra $4$, but is covered by the error term.
\end{proof}

The next result follows directly.

\begin{cor}
For a topograph $\mathcal G$  with an edge or vertex well, the area of $\mathcal B(\mathcal G)$ is independent of the Farey indexing chosen.
\end{cor}

\begin{ex} 
{\rm Taking $\mathcal G = (3,3,3)$, the topograph completion of the original Markov tree in Figure \ref{mart}, we have $T_0=T_1=(3,6,3)$ and $T_2=T_3=(3,15,6)$ by \e{r3}.  Before \e{16}, the area of $\mathcal W_{T_2}$ was computed as $c \approx 0.297267722$ and $\mathcal W_{T_0}$ was noted to have area $2c$. Hence Lemma \ref{few} gives ${\rm Area}(\mathcal B(\mathcal G)) = 2(2c+2c+c+c)=12c$, making the area of the ball on the left of Figure \ref{balls} twelve times the area of the wedge $\mathcal W$ in Figure \ref{wplot}. The coefficient of the main term in \e{mzxb} is then $36 c/\pi^2$. This is six times the coefficient in  Corollary \ref{zm} and is explained by the numbers in $\mathcal G$ (aside from $3$ and $6$) having multiplicity $6$.}
\end{ex}

\begin{ex} 
{\rm Let $\mathcal G = (4,7,6)$ with $D=-67$, and we have $T_0=(4,22,7)$, $T_1=(6,38,7)$, $T_2=(6,98,17)$ and $T_3=(4,62,17)$; see Figure \ref{rtree}. Then computing with Lemma \ref{few} and Proposition \ref{wedg} finds ${\rm Area}(\mathcal B(\mathcal G)) \approx 1.18418755$. For $X=10^{250}$, the main term in \e{mzxb} is $\approx 119276.09$ while the true value of $M_{\mathcal G}(X)$ is $119974$. The difference is less than $700$, comparing favorably with $\log X \log \log X \approx 3658.5$.}
\end{ex}

\begin{ex} 
{\rm For the topograph  $\mathcal G = (-3,-5,-4)$ with $D=110$, we have $T_0=(3,19,5)$, $T_1=(4,23,5)$, $T_2=(4,65,17)$ and $T_3=(3,47,17)$. Calculating as before finds ${\rm Area}(\mathcal B(\mathcal G)) \approx 1.637393913$. For $X=10^{250}$, the main term in \e{mzxb} is $\approx 164924.84$ while  $M_{\mathcal G}(X)=164915$. The difference is very small, less than $10$.}
\end{ex}




{\small 
\bibliography{patbib} }

{\small 
\vskip 5mm
\noindent
\textsc{Department of Mathematics, The CUNY Graduate Center, 365 Fifth Avenue, New York, NY 10016-4309, U.S.A.}

\noindent
{\em E-mail address:} \texttt{cosullivan@gc.cuny.edu}
}

\end{document}